\documentclass[11pt]{amsart}
\usepackage[a4paper,margin=30mm]{geometry}
\usepackage{amsmath,amssymb,mathtools,mathrsfs}
\usepackage{microtype}
\usepackage[colorlinks=true,linkcolor=blue,citecolor=blue,urlcolor=blue,
pdftitle={Sharp Fractional Hardy--Maz'ya Remainders: Angular Reduction and Energy Stability},
pdfauthor={Shengbing Deng and Dehong Yang},
pdfkeywords={fractional Hardy inequality, Maz'ya remainder, dimensional reduction, spectral stability}]{hyperref}
\numberwithin{equation}{section}

\newtheorem{theorem}{Theorem}[section]
\newtheorem{proposition}[theorem]{Proposition}
\newtheorem{lemma}[theorem]{Lemma}
\newtheorem{corollary}[theorem]{Corollary}
\newtheorem{remark}[theorem]{Remark}
\theoremstyle{definition}
\newtheorem{definition}[theorem]{Definition}

\title[Angular reduction and energy stability]
{Sharp fractional Hardy--Maz'ya remainders on half-spaces:\
angular reduction and energy stability}

\author{Shengbing Deng}
\email{shbdeng@swu.edu.cn}
\author{Dehong Yang}


\begin{document}

\begin{abstract}
For $d\geq2$, $0<s<1$, $p>1$, and $\tau>0$, we study the sharp
half-space remainder
\[
 [u]_{s,p,+}^p-\mathcal D_{d,s,p}\int_{\mathbb R^d_+}x_d^{-sp}|u|^p\,dx
 \geq\Lambda_{d,s,p,\tau}\int_{\mathbb R^d_+}
 \frac{x_d^{\tau-sp}|u|^p}{(x_{d-1}^2+x_d^2)^{\tau/2}}\,dx.
\]
We first obtain exact dimensional and angular reductions by an $L^p$
contraction. For $p=2$, we identify the planar sharp coefficient with the
simple first eigenvalue of a singular angular operator with compact
resolvent. We prove a ground-state representation and a two-sided
estimate for the sharp deficit. These give quantitative stability and a
Fourier--Mellin energy criterion for extremizing sequences in all
dimensions. For the distance with multiplier $h(\xi)=\min\{\xi^2,1\}$,
the best planar stability coefficient is not attained and is strictly
smaller than the quadratic dispersion coefficient. The minimizing
frequencies are nonzero when $\tau=2$ and $s$ is close to one.
Finally, we prove normalized local limits for the sharp quadratic
constant, each fixed fiber eigenvalue, and the dispersion coefficient.
\end{abstract}

\maketitle

\section{Introduction}

Let $d\geq2$, $0<s<1$, and $1<p<\infty$. On
$\mathbb R^d_+=\{x\in\mathbb R^d:x_d>0\}$, write
\[
 [u]_{s,p,+}^p=\int_{\mathbb R^d_+}\int_{\mathbb R^d_+}
 \frac{|u(x)-u(y)|^p}{|x-y|^{d+sp}}\,dx\,dy.
\]
The sharp Hardy inequality on the half-space admits a second remainder
involving the distance to the codimension-two edge
$\{x_{d-1}=x_d=0\}$. We are concerned with the best coefficient of
this remainder and the behavior of extremizing sequences.

For $\tau>0$, put
\[
 \begin{split}
 \mathcal H_{d,s,p}[u]={}&[u]_{s,p,+}^p
 -\mathcal D_{d,s,p}\int_{\mathbb R^d_+}x_d^{-sp}|u|^p\,dx,\\
 \mathcal R_{d,s,p,\tau}[u]={}&
 \int_{\mathbb R^d_+}
 \frac{x_d^{\tau-sp}|u(x)|^p}
 {(x_{d-1}^2+x_d^2)^{\tau/2}}\,dx,
 \end{split}
\]
and
\[
 \Lambda_{d,s,p,\tau}=
 \inf_{u\in C_c^\infty(\mathbb R^d_+)\setminus\{0\}}
 \frac{\mathcal H_{d,s,p}[u]}{\mathcal R_{d,s,p,\tau}[u]}.
\]
Here $\mathcal D_{d,s,p}$ is the sharp regional half-space Hardy
constant, given in \eqref{eq:hardy-general-p}. Bogdan and Dyda
\cite{BogdanDyda} determined it for $p=2$, and Frank and Seiringer
\cite{FrankSeiringerHalf} treated general powers by the nonlinear
ground-state method of \cite{FrankSeiringer}.

The local second remainder goes back to Maz'ya
\cite{Mazya1972,MazyaBook}. Tidblom proved inequalities for a larger
family of weights \cite{Tidblom}. Maz'ya and Shaposhnikova identified the
sharp quadratic coefficient at $\tau=1$ through an angular eigenvalue
problem \cite{MazyaShaposhnikova}. The fractional inequality with the
same two-plane weight was established by Kijaczko and Szczukiewicz
\cite{KijaczkoSzczukiewicz}. Their result gives a positive coefficient. The optimal coefficient and
the corresponding loss of compactness remain to be determined.

Fractional Hardy--Sobolev--Maz'ya inequalities involve a Sobolev
remainder; see \cite{Dyda2010,DydaFrank,DydaLehrbackVahakangas,Sloane}.
Weighted versions and inequalities with flat singular sets are treated
in \cite{DydaKijaczkoWeighted,KijaczkoSahu,SahuFlat}. Related results
for Sobolev--Bregman forms can be found in
\cite{KijaczkoLenczewska,GutowskiKwasnicki}.

Banerjee, Ganguly, and Sahu \cite{BanerjeeGangulySahu} obtained
full-space Hardy stability estimates by Emden--Fowler and Mellin
methods, using Lorentz or Marcinkiewicz distances to virtual extremals.
For the half-space problem considered here, the second sharp potential
must also be subtracted. The resulting deficit is governed by a singular
angular operator.

Kijaczko and Szczukiewicz proved a positive remainder in the ranges
\cite{KijaczkoSzczukiewicz}
\[
 0<\tau<\begin{cases}
 p(sp+1)/(2(p-1)),&p\geq2,\\
 2\min\{sp,1\},&1<p<2.
 \end{cases}
\]
Since $\mathcal R_{d,s,p,\tau}\leq\mathcal R_{d,s,p,\sigma}$ for
$0<\sigma\leq\tau$, monotonicity gives $\Lambda_{d,s,p,\tau}>0$ for every
$d\geq2$, $0<s<1$, $p>1$, and $\tau>0$.
For $d=p=2$, the estimate used below is
\begin{equation}\label{eq:intro-KS}
 \begin{split}
 [u]_{s,+}^2-\mathcal D_{2,s,2}
 \int_{\mathbb R^2_+}\frac{|u|^2}{x_2^{2s}}\,dx
 \geq c_{s,\tau}
 \int_{\mathbb R^2_+}
 \frac{|u|^2}{x_2^{2s-\tau}|x|^\tau}\,dx,
 \end{split}
\end{equation}
where $c_{s,\tau}>0$ for $0<\tau<1+2s$. We use this estimate to
close the positive angular and cylinder forms.

Our first result reduces the best coefficient to the planar problem.
The dimensional factor follows from a contraction in the transverse
$L^p$ variables.

\begin{theorem}[Dimensional reduction for general powers]
\label{thm:dimension-reduction}
Let $d\geq2$, $0<s<1$, $p>1$, and $\tau>0$. Then
\begin{equation}\label{eq:intro-dimensional-reduction-p}
 \Lambda_{d,s,p,\tau}=A_{d,s,p}\Lambda_{2,s,p,\tau},
 \qquad
 A_{d,s,p}=\frac{\pi^{(d-2)/2}\Gamma(1+sp/2)}
 {\Gamma((d+sp)/2)}.
\end{equation}
For $p=2$, put $A_{d,s}=A_{d,s,2}$,
$\Lambda_{d,s,\tau}=\Lambda_{d,s,2,\tau}$, and
$\Lambda_{s,\tau}=\Lambda_{2,s,2,\tau}$. Then
\begin{equation}\label{eq:intro-dimensional-reduction}
 \Lambda_{d,s,\tau}=A_{d,s}\Lambda_{s,\tau},
 \qquad
 A_{d,s}=\frac{\pi^{(d-2)/2}\Gamma(1+s)}
 {\Gamma((d+2s)/2)}.
\end{equation}
The quadratic quotient has no nonzero extremal in the completion of
$C_c^\infty(\mathbb R^d_+)$ under
$(\mathcal H_{d,s,2}+\mathcal R_{d,s,2,\tau})^{1/2}$.
For $d=2$, the last assertion is Theorem~\ref{thm:A}.
\end{theorem}

Theorem~\ref{thm:nonlinear-angular-reduction} further reduces the
planar problem by an $L^p$ contraction in the logarithmic radial
variable. For every $p>1$, it expresses $\Lambda_{2,s,p,\tau}$ as an
angular variational infimum. We next consider $p=2$, for which the
angular problem has a compact spectral realization.

In the planar quadratic case, write
\begin{equation}\label{eq:intro-energy}
 [u]_{s,+}^2=\int_{\mathbb R^2_+}\int_{\mathbb R^2_+}
 \frac{|u(x)-u(y)|^2}{|x-y|^{2+2s}}\,dx\,dy.
\end{equation}
Then
\begin{equation*}
 [u]_{s,+}^2\geq\mathcal D_{2,s,2}
 \int_{\mathbb R^2_+}x_2^{-2s}|u|^2\,dx,
\end{equation*}
where, in this ordered-double-integral normalization,
\begin{equation}\label{eq:intro-hardyc}
 \mathcal D_{2,s,2}=
 \frac{2\sqrt\pi\,\Gamma(s+\frac12)}{\Gamma(1+s)}
 \int_0^1\frac{(1-t^{(2s-1)/2})^2}{(1-t)^{1+2s}}\,dt.
\end{equation}
In particular, $\mathcal D_{2,1/2,2}=0$.

We use
\[
 \mathcal H_{d,s}=\mathcal H_{d,s,2},\qquad
 \mathcal R_{d,s,\tau}=\mathcal R_{d,s,2,\tau}.
\]
For $d=2$, put $\mathcal R_{s,\tau}=\mathcal R_{2,s,\tau}$.
The planar deficit is
\begin{equation}\label{eq:intro-deficit}
 \mathcal H_s[u]=[u]_{s,+}^2-\mathcal D_{2,s,2}
 \int_{\mathbb R^2_+}\frac{|u|^2}{x_2^{2s}}\,dx.
\end{equation}
Both quadratic terms have degree $2s-2$ under Euclidean dilation.
The energy space is the completion of $C_c^\infty(\mathbb R^d_+)$ under
$(\mathcal H_{d,s}+\mathcal R_{d,s,\tau})^{1/2}$. The positive remainder
estimate makes this a norm. In the planar logarithmic polar coordinates,
dilation becomes translation.

Put
\[
 K_s(q,\theta,\varphi)=
 \frac{2^{-1-s}}{\bigl(\cosh q-\cos(\theta-\varphi)\bigr)^{1+s}}
\]
and $a_s=1-s$. For $\xi\in\mathbb R$, define
\begin{equation}\label{eq:intro-fiber}
 \begin{split}
 \mathfrak b_{s,\xi}[\psi]
 ={}&\int_0^\pi\int_0^\pi\int_{\mathbb R}
 K_s(q,\theta,\varphi)\\
 &\quad\times\bigl[\cosh(a_sq)(|\psi(\theta)|^2+|\psi(\varphi)|^2)\\
 &\qquad-2\operatorname{Re}(e^{i\xi q}\psi(\theta)
 \overline{\psi(\varphi)})\bigr]\,dq\,d\theta\,d\varphi\\
 &-\mathcal D_{2,s,2}\int_0^\pi
 \frac{|\psi(\theta)|^2}{(\sin\theta)^{2s}}\,d\theta.
 \end{split}
\end{equation}
The form is first defined on $C_c^\infty(0,\pi)$. Its closed extension is
constructed in Section~\ref{sec:fibers}. Let
\[
 m_{s,\tau}(\theta)=(\sin\theta)^{\tau-2s}
\]
and set
\[
 H_{s,\tau}=L^2((0,\pi),m_{s,\tau}(\theta)\,d\theta).
\]
Let $\mathscr Y_s$ be the common closed angular form domain and
$\mathscr X_s$ the closed cylinder domain constructed in
Sections~\ref{sec:coercivity} and~\ref{sec:fibers}. We use complex forms
for the Mellin transform, with angular norm
\[
 \|\psi\|_{\mathscr Y_s}^2=
 \mathfrak b_{s,0}[\psi]+\|\psi\|_{H_{s,\tau}}^2.
\]
The real conclusions follow by taking real and imaginary parts.
Let $B_s(\xi)$ be the self-adjoint operator associated with
$\mathfrak b_{s,\xi}$ in $H_{s,\tau}$. Lemma~\ref{lem:closed-fibers}
shows that its resolvent is compact. We write
\[
 0<\lambda_1(\xi)\leq\lambda_2(\xi)\leq\cdots,
 \qquad \lambda_j(\xi)\longrightarrow\infty.
\]
In particular,
\[
 \lambda_1(\xi)=
 \inf_{\psi\in\mathscr Y_s\setminus\{0\}}
 \frac{\mathfrak b_{s,\xi}[\psi]}
 {\int_0^\pi m_{s,\tau}|\psi|^2\,d\theta}.
\]

\begin{theorem}[Sharp reduction and non-attainment]\label{thm:A}
Let $0<s<1$ and $\tau>0$. Then
\[
 \Lambda_{s,\tau}=\lambda_1(0)>0.
\]
The eigenvalue $\lambda_1(0)$ is simple and has a positive eigenfunction
$\phi_{s,\tau}$, normalized by
\[
 \int_0^\pi m_{s,\tau}(\theta)\phi_{s,\tau}(\theta)^2\,d\theta=1.
\]
The value $\Lambda_{s,\tau}$ is not attained by a nonzero function in the
completion of $C_c^\infty(\mathbb R^2_+)$ under
$(\mathcal H_s+\mathcal R_{s,\tau})^{1/2}$.
\end{theorem}

The homogeneous function
$U_{s,\tau}(r\cos\theta,r\sin\theta)=r^{s-1}\phi_{s,\tau}(\theta)$
solves the weak equation for the full sharp potential.
In Theorem~\ref{thm:sharp-ground-state}, we show that
$\mathcal H_s[u]-\Lambda_{s,\tau}\mathcal R_{s,\tau}[u]$ equals a
positive jump form with weight $U_{s,\tau}(x)U_{s,\tau}(y)$.
This identity extends to the completed energy space. To state the
quantitative estimate, we introduce the following distance.

For $x=e^t(\cos\theta,\sin\theta)$, define
\[
 w(t,\theta)=e^{(1-s)t}u(e^t\cos\theta,e^t\sin\theta).
\]
Let $\widehat w(\xi,\theta)$ be the Fourier transform in $t$. Let
$P_0$ be the orthogonal projection in
$L^2((0,\pi),m_{s,\tau}d\theta)$ onto
$\operatorname{span}\{\phi_{s,\tau}\}$. Set
\begin{equation}\label{eq:intro-distance}
 \begin{split}
 \operatorname{Dist}_{s,\tau}(u)^2={}&
 \int_{\mathbb R}\|(I-P_0)\widehat w(\xi,\cdot)\|_{L^2(m_{s,\tau})}^2\,d\xi+\int_{\mathbb R}\min\{\xi^2,1\}
 \|P_0\widehat w(\xi,\cdot)\|_{L^2(m_{s,\tau})}^2\,d\xi.
 \end{split}
\end{equation}
Its kernel in the energy space is $\{0\}$, since the formal zero mode
$w(t,\theta)=c\phi_{s,\tau}(\theta)$ is not in $L^2$ in $t$.
Since $\xi^2/(1+\xi^2)$ is comparable to
$\min\{\xi^2,1\}$, Plancherel also gives the intrinsic equivalent form
\[
 \operatorname{Dist}_{s,\tau}(u)^2\asymp
 \|(I-P_0)w\|_{L^2(\mathbb R;H_{s,\tau})}^2+
 \|\partial_t(1-\partial_t^2)^{-1/2}P_0w\|_
 {L^2(\mathbb R;H_{s,\tau})}^2.
\]
The first term measures the component orthogonal to the angular
ground state. The second measures variation in the logarithmic radius.

\begin{theorem}[Quantitative stability]\label{thm:B}
Let $0<s<1$ and $\tau>0$. There is
$\gamma_{s,\tau}>0$ such that
\begin{equation}\label{eq:intro-mainB}
 \mathcal H_s[u]-\Lambda_{s,\tau}\mathcal R_{s,\tau}[u]
 \geq \gamma_{s,\tau}\operatorname{Dist}_{s,\tau}(u)^2
\end{equation}
for every $u\in C_c^\infty(\mathbb R^2_+)$. Moreover,
\begin{equation}\label{eq:intro-dispersion}
 \lambda_1(\xi)=\Lambda_{s,\tau}+\kappa_{s,\tau}\xi^2+O(\xi^4)
 \quad\text{as }\xi\to0,
\end{equation}
where
\begin{equation}\label{eq:intro-kappa}
 \kappa_{s,\tau}=
 \int_0^\pi\int_0^\pi\int_{\mathbb R}
 q^2K_s(q,\theta,\varphi)
 \phi_{s,\tau}(\theta)\phi_{s,\tau}(\varphi)
 \,dq\,d\theta\,d\varphi>0.
\end{equation}
\end{theorem}

For the fixed distance in \eqref{eq:intro-distance}, define the
optimal coefficient
\[
 \Gamma^{\rm hp}_{s,\tau}=
 \inf_{u\neq0}
 \frac{\mathcal H_s[u]-\Lambda_{s,\tau}\mathcal R_{s,\tau}[u]}
 {\operatorname{Dist}_{s,\tau}(u)^2},
\]
where the infimum is taken over the completed energy space. Put
$h(\xi)=\min\{\xi^2,1\}$ and $Q_0=I-P_0$. For $\xi\neq0$, let
\begin{equation}\label{eq:intro-mu}
 \mu_{s,\tau}(\xi)=
 \inf_{\psi\in\mathscr Y_s\setminus\{0\}}
 \frac{\mathfrak b_{s,\xi}[\psi]-\Lambda_{s,\tau}
 \|\psi\|_{L^2(m_{s,\tau})}^2}
 {\|Q_0\psi\|_{L^2(m_{s,\tau})}^2+
 h(\xi)\|P_0\psi\|_{L^2(m_{s,\tau})}^2}.
\end{equation}

\begin{theorem}[Optimal half-plane stability coefficient]
\label{thm:optimal-hp}
Let $0<s<1$ and $\tau>0$.
The number $\Gamma^{\rm hp}_{s,\tau}$ is positive and
\begin{equation}\label{eq:intro-gamma-fiber}
 \Gamma^{\rm hp}_{s,\tau}
 =\mathop{\rm ess\,inf}_{\xi\in\mathbb R\setminus\{0\}}
 \mu_{s,\tau}(\xi).
\end{equation}
For every $\xi\neq0$, the infimum in \eqref{eq:intro-mu} is attained.
The function $\mu_{s,\tau}$ extends continuously to zero, with the value
given in \eqref{eq:intro-mu-zero}, and tends to infinity as $|\xi|\to\infty$. Consequently,
\[
 \Gamma^{\rm hp}_{s,\tau}=\min_{\xi\in\mathbb R}\mu_{s,\tau}(\xi).
\]
Set
\[
 Z_{s,\tau}=\{\xi\neq0:\mu_{s,\tau}(\xi)
 =\Gamma^{\rm hp}_{s,\tau}\}.
\]
The set $Z_{s,\tau}$ is finite, and $\Gamma^{\rm hp}_{s,\tau}$ is
not attained by a nonzero energy function. The variational
characterization above is relative to the fixed multiplier $h$.
At zero,
\begin{equation}\label{eq:intro-mu-zero}
 \lim_{\xi\to0}\mu_{s,\tau}(\xi)
 =\min\{\kappa_{s,\tau},\lambda_2(0)-\lambda_1(0)\}.
\end{equation}
\end{theorem}

If $\mathcal R_{s,\tau}[u_n]=1$, denote the transform of $u_n$ by $w_n$.
Define the axial dispersion function
\begin{equation}\label{eq:intro-rho}
 \begin{split}
 \rho_{s,\tau}(\xi)
 &=\mathfrak b_{s,\xi}[\phi_{s,\tau}]-\Lambda_{s,\tau}\\
 &=2\int_0^\pi\int_0^\pi\int_{\mathbb R}
 (1-\cos(\xi q))K_s(q,\theta,\varphi)
 \phi_{s,\tau}(\theta)\phi_{s,\tau}(\varphi)
 \,dq\,d\theta\,d\varphi.
 \end{split}
\end{equation}
It is positive for $\xi\neq0$ and
$\rho_{s,\tau}(\xi)=\kappa_{s,\tau}\xi^2+O(\xi^4)$ at zero.
Proposition~\ref{prop:strict-stability-comparison} gives
$\Gamma^{\rm hp}_{s,\tau}\leq\rho_{s,\tau}(1)<\kappa_{s,\tau}$.
In particular, the quadratic dispersion coefficient does not give the
best constant for this distance.

\begin{theorem}[Energy decomposition and extremizing sequences]\label{thm:C}
Let $0<s<1$ and $\tau>0$. The canonical decomposition in
\eqref{eq:canonical-decomposition} satisfies the two-sided energy
estimate \eqref{eq:energy-decomposition} on $\mathscr X_s$.
Let $u_n$ belong to the completed energy space
and satisfy
$\mathcal R_{s,\tau}[u_n]=1$.
The following statements are equivalent:
\begin{enumerate}
\item $\mathcal H_s[u_n]\to\Lambda_{s,\tau}$;
\item there are $a_n\in L^2(\mathbb R)$ and $r_n\in\mathscr X_s$ such that
\[
 w_n(t,\theta)=a_n(t)\phi_{s,\tau}(\theta)+r_n(t,\theta),
\]
\begin{equation}\label{eq:intro-rn}
 \|r_n\|_{\mathscr X_s}\to0,
 \qquad \|a_n\|_{L^2(\mathbb R)}\to1,
\end{equation}
and
\begin{equation}\label{eq:intro-an}
 \int_{\mathbb R}\rho_{s,\tau}(\xi)
 |\widehat a_n(\xi)|^2\,d\xi\to0.
\end{equation}
\end{enumerate}
In the implication $(1)\Rightarrow(2)$, the functions $a_n$ may be chosen
by the low-frequency moving projection defined in
\eqref{eq:canonical-decomposition}. Thus the angular defect vanishes in the
energy norm. The remaining mass disperses in the logarithmic radial variable.
\end{theorem}

When $d\geq3$, the partial Fourier transform along the edge produces
an additional nonnegative term. Combining this term with the planar
decomposition gives the necessary and sufficient criterion of
Theorem~\ref{thm:joint-energy-decomposition}. This criterion is stated
in joint Fourier--Mellin energy variables.

Our last main result concerns the limit as the fractional order tends
to one. We use the fixed Hilbert space
\[
 H_{1,\tau}=L^2((0,\pi),(\sin\theta)^{\tau-1}\,d\theta).
\]
Let $\mathfrak l_{1,\xi}$ be the local angular form in
Section~\ref{sec:local-limit}, and let $y_{1,\tau}$ be its positive
normalized zero-frequency minimizer. The number $\Lambda_{1,\tau}$ is
the sharp local half-space remainder defined there.

\begin{theorem}[Sharp local limit]
\label{thm:sharp-local-limit}
Let $\tau>0$. Then
\[
 \lim_{s\uparrow1}\frac{2(1-s)}\pi\Lambda_{s,\tau}
 =\Lambda_{1,\tau}.
\]
Let $\lambda_j^{(1)}(\xi)$ be the eigenvalues of
$\mathfrak l_{1,\xi}$ in $H_{1,\tau}$. For every fixed $j\geq1$ and
$\xi\in\mathbb R$,
\[
 \lim_{s\uparrow1}\frac{2(1-s)}\pi\lambda_j(\xi)
 =\lambda_j^{(1)}(\xi).
\]
If $y_{1,\tau}$ is the positive normalized minimizer in
Lemma~\ref{lem:local-angular-form}, then
\[
 (\sin\theta)^{1/2-s}\phi_{s,\tau}
 \longrightarrow y_{1,\tau}
 \quad\text{strongly in }H_{1,\tau}.
\]
Finally,
\[
 \lim_{s\uparrow1}\frac{2(1-s)}\pi\kappa_{s,\tau}
 =\int_0^\pi\sin\theta\,y_{1,\tau}(\theta)^2\,d\theta.
\]
\end{theorem}

At $\tau=2$, Proposition~\ref{prop:local-benchmark} gives the explicit
local values
\[
 \Lambda_{1,2}=\frac14,\qquad
 \lambda_2^{(1)}(0)-\lambda_1^{(1)}(0)=2,\qquad
 \int_0^\pi\sin\theta\,y_{1,2}^2\,d\theta=1.
\]
The local stability coefficient for the same multiplier $h$ is $1$
and is attained. Theorem~\ref{thm:sharp-local-limit} and
Proposition~\ref{prop:strict-stability-comparison} also imply that
$Z_{s,2}$ contains nonzero frequencies for all $s<1$ sufficiently
close to one. This conclusion concerns the fixed spectral distance
\eqref{eq:intro-distance}. The local limits above are at fixed
frequency and index; the limit of $\Gamma^{\rm hp}_{s,\tau}$ is not
asserted.

\begin{remark}
The critical Hardy term prevents us from treating the angular energy
and potential separately on the closed domain. We instead close the
positive ground-state form. A stronger endpoint weight and interior
fractional regularity then give compactness. Since
$|x|^{s-1}\phi_{s,\tau}$ does not belong to the energy space, the
identity in Theorem~\ref{thm:sharp-ground-state} requires a
regularization of the angular ground state.

For the stability estimate, we separate small frequencies from their
complement and use the analytic first spectral projection near zero.
This also gives the two-sided energy decomposition. To prove
non-attainment of the best stability coefficient, we apply analytic
Fredholm theory to the generalized fiber problem. This argument allows
eigenvalue crossings away from zero. Finally, the local spectral limit
requires a dyadic endpoint estimate uniform as $s\uparrow1$;
convergence on the smooth core alone is insufficient.
\end{remark}

Section~\ref{sec:reductions} proves the dimensional and angular reductions
for general powers. Sections~\ref{sec:transform}--\ref{sec:fibers}
construct the quadratic closed forms and establish their compactness and
frequency estimates. Section~\ref{sec:sharp} identifies the planar sharp
constant and its full ground-state representation.
Section~\ref{sec:stability} proves stability and the energy decomposition.
Section~\ref{sec:dimension} gives the higher-dimensional splitting.
Section~\ref{sec:local-limit} proves the local spectral limits.
The appendices contain the spectral periodic model, angular Ritz
approximation, and subcritical nonlinear coercivity.

\section{Exact reductions for general powers}\label{sec:reductions}

We prove the dimensional reduction first and then the angular
reduction. Both follow from the reverse triangle inequality in $L^p$
and a matching family of test functions.

\begin{lemma}[$L^p$ transverse reduction]\label{lem:lp-dimension}
Let $d\geq3$, $0<s<1$, $p>1$, and $\tau>0$. Then
\[
 \Lambda_{d,s,p,\tau}=A_{d,s,p}\Lambda_{2,s,p,\tau}.
\]
\end{lemma}

\begin{proof}
Put $n=d-2$. Write $x=(z,X)$, where $z\in\mathbb R^n$ and
$X=(x_{d-1},x_d)\in\mathbb R^2_+$. For
$u\in C_c^\infty(\mathbb R^d_+)$, set
\[
 u_X(z)=u(z,X),\qquad
 U(X)=\|u_X\|_{L^p(\mathbb R^n)}.
\]
For every $h\in\mathbb R^n$, translation invariance and the reverse
triangle inequality give
\begin{equation}\label{eq:lp-reverse-triangle}
 \int_{\mathbb R^n}|u(z,X)-u(z-h,Y)|^p\,dz
 \geq |U(X)-U(Y)|^p.
\end{equation}
The beta integral gives
\begin{equation}\label{eq:lp-transverse-kernel}
 \begin{split}
 \int_{\mathbb R^n}\frac{dh}
 {(|X-Y|^2+|h|^2)^{(d+sp)/2}}
 &=A_{d,s,p}|X-Y|^{-2-sp},\\
 A_{d,s,p}&=
 \frac{\pi^{n/2}\Gamma(1+sp/2)}
 {\Gamma((d+sp)/2)}.
 \end{split}
\end{equation}
By Fubini's theorem, \eqref{eq:lp-reverse-triangle}, and
\eqref{eq:lp-transverse-kernel},
\begin{equation}\label{eq:lp-energy-lower}
 [u]_{s,p,+}^p\geq A_{d,s,p}[U]_{s,p,+}^p.
\end{equation}

The explicit sharp half-space Hardy constant is
\begin{equation}\label{eq:hardy-general-p}
 \mathcal D_{d,s,p}=
 \frac{2\pi^{(d-1)/2}\Gamma((1+sp)/2)}
 {\Gamma((d+sp)/2)}
 \int_0^1\frac{|1-t^{(sp-1)/p}|^p}
 {(1-t)^{1+sp}}\,dt.
\end{equation}
Comparison of dimensions $d$ and $2$ in
\eqref{eq:hardy-general-p} gives
\begin{equation}\label{eq:hardy-dimension-factor}
 \mathcal D_{d,s,p}=A_{d,s,p}\mathcal D_{2,s,p}.
\end{equation}
The two weighted integrals satisfy
\begin{equation}\label{eq:lp-potential-reduction}
 \begin{split}
 \int_{\mathbb R^d_+}x_d^{-sp}|u|^p\,dx
 &=\int_{\mathbb R^2_+}X_2^{-sp}U(X)^p\,dX,\\
 \mathcal R_{d,s,p,\tau}[u]
 &=\mathcal R_{2,s,p,\tau}[U].
 \end{split}
\end{equation}
It follows from \eqref{eq:lp-energy-lower}--
\eqref{eq:lp-potential-reduction} that
\begin{equation}\label{eq:lp-deficit-lower}
 \mathcal H_{d,s,p}[u]
 \geq A_{d,s,p}\mathcal H_{2,s,p}[U].
\end{equation}

The function $U$ has compact support in the half-plane and belongs to
$W^{s,p}(\mathbb R^2_+)$ by \eqref{eq:lp-energy-lower}. Zero extension and
mollification give $U_j\in C_c^\infty(\mathbb R^2_+)$, supported in a fixed
compact subset, with $U_j\to U$ in $W^{s,p}$. Both weights are bounded
there, so the weighted integrals also converge. The planar sharp
inequality therefore applies to $U$.
Equation \eqref{eq:lp-deficit-lower} gives
\[
 \mathcal H_{d,s,p}[u]
 \geq A_{d,s,p}\Lambda_{2,s,p,\tau}
 \mathcal R_{d,s,p,\tau}[u].
\]
Hence
$\Lambda_{d,s,p,\tau}\geq
A_{d,s,p}\Lambda_{2,s,p,\tau}$.

For the reverse inequality, choose
$g\in C_c^\infty(\mathbb R^n)$ and
$v\in C_c^\infty(\mathbb R^2_+)$ with $\|g\|_{L^p}=1$. Put
\[
 g_R(z)=R^{-n/p}g(z/R),\qquad
 u_R(z,X)=g_R(z)v(X).
\]
The weighted integrals factor. In particular,
\begin{equation}\label{eq:lp-plateau-remainder}
 \mathcal R_{d,s,p,\tau}[u_R]
 =\mathcal R_{2,s,p,\tau}[v].
\end{equation}
On the measure space with measure
\[
 d\mu=\frac{dz\,dh\,dX\,dY}
 {(|X-Y|^2+|h|^2)^{(d+sp)/2}},
\]
write
\[
 \begin{split}
 u_R(z,X)-u_R(z-h,Y)&=A_R+B_R,\\
 A_R&=g_R(z)(v(X)-v(Y)),\\
 B_R&=v(Y)(g_R(z)-g_R(z-h)).
 \end{split}
\]
Equation \eqref{eq:lp-transverse-kernel} gives
\begin{equation}\label{eq:lp-main-plateau}
 \|A_R\|_{L^p(d\mu)}^p
 =A_{d,s,p}[v]_{s,p,+}^p.
\end{equation}
For $h\neq0$, integration in $X$ gives
\[
 \int_{\mathbb R^2_+}
 \frac{dX}{(|X-Y|^2+|h|^2)^{(d+sp)/2}}
 \leq C_{d,s,p}|h|^{-n-sp}.
\]
Consequently,
\begin{equation}\label{eq:lp-plateau-error}
 \begin{split}
 \|B_R\|_{L^p(d\mu)}^p
 &\leq C_{d,s,p}\|v\|_{L^p(\mathbb R^2_+)}^p
 \int_{\mathbb R^n}\int_{\mathbb R^n}
 \frac{|g_R(z)-g_R(z-h)|^p}{|h|^{n+sp}}\,dh\,dz\\
 &=C_{d,s,p}R^{-sp}\|v\|_{L^p(\mathbb R^2_+)}^p
 [g]_{W^{s,p}(\mathbb R^n)}^p.
 \end{split}
\end{equation}
Minkowski's inequality and
\eqref{eq:lp-main-plateau}--\eqref{eq:lp-plateau-error} give
\[
 \left|[u_R]_{s,p,+}-A_{d,s,p}^{1/p}[v]_{s,p,+}\right|
 \leq \|B_R\|_{L^p(d\mu)}\longrightarrow0.
\]
Together with \eqref{eq:hardy-dimension-factor}, this proves
\[
 \mathcal H_{d,s,p}[u_R]
 \longrightarrow A_{d,s,p}\mathcal H_{2,s,p}[v].
\]
Divide by \eqref{eq:lp-plateau-remainder}, let $R\to\infty$, and minimize
over $v$. This proves the reduction formula.
\end{proof}

\begin{theorem}[Angular reduction for general powers]
\label{thm:nonlinear-angular-reduction}
Let $0<s<1$, $p>1$, and $\tau>0$. Put
\[
 \alpha_{s,p}=\frac{2-sp}{p},\qquad
 \mathcal K_{s,p}(q,\theta,\varphi)
 =\frac{2^{-1-sp/2}}
 {(\cosh q-\cos(\theta-\varphi))^{1+sp/2}}.
\]
For $\psi\in C_c^\infty(0,\pi)$, define
\[
 \begin{split}
 \mathcal A_{s,p}[\psi]
 &=\int_0^\pi\int_0^\pi\int_{\mathbb R}
 \mathcal K_{s,p}(q,\theta,\varphi)
 \left|e^{-\alpha_{s,p}q/2}\psi(\theta)
       -e^{\alpha_{s,p}q/2}\psi(\varphi)\right|^p
 \,dq\,d\theta\,d\varphi,\\
 \mathcal B_{s,p}[\psi]
 &=\mathcal A_{s,p}[\psi]
 -\mathcal D_{2,s,p}\int_0^\pi(\sin\theta)^{-sp}|\psi|^p\,d\theta.
 \end{split}
\]
Then
\begin{equation}\label{eq:nonlinear-angular-constant}
 \Lambda_{2,s,p,\tau}
 =\inf_{\psi\in C_c^\infty(0,\pi)\setminus\{0\}}
 \frac{\mathcal B_{s,p}[\psi]}
 {\displaystyle\int_0^\pi(\sin\theta)^{\tau-sp}|\psi|^p\,d\theta}.
\end{equation}
The infimum is unchanged if $\psi\geq0$ is required.
For $p=2$, one has $\mathcal B_{s,2}=\mathfrak b_{s,0}$ on the angular
core. In particular, \eqref{eq:nonlinear-angular-constant} extends the
zero-frequency variational formula to every $p>1$.
\end{theorem}

\begin{proof}
For $u\in C_c^\infty(\mathbb R^2_+)$, put
\[
 w(t,\theta)=e^{\alpha_{s,p}t}
 u(e^t\cos\theta,e^t\sin\theta),\qquad
 V(\theta)=\|w(\cdot,\theta)\|_{L^p(\mathbb R)}.
\]
The polar distance identity gives
\begin{equation}\label{eq:general-p-cylinder}
 \begin{split}
 [u]_{s,p,+}^p
 =\int_{\mathbb R}\int_0^\pi\int_0^\pi
 \mathcal K_{s,p}(q,\theta,\varphi)
 \int_{\mathbb R}
 \left|e^{-\alpha_{s,p}q/2}w(t,\theta)
 -e^{\alpha_{s,p}q/2}w(t-q,\varphi)\right|^p
 \,dt\,d\theta\,d\varphi\,dq.
 \end{split}
\end{equation}
Both weighted integrals transform without a radial weight:
\[
 \begin{split}
 \int_{\mathbb R^2_+}x_2^{-sp}|u|^p\,dx
 &=\int_0^\pi(\sin\theta)^{-sp}V(\theta)^p\,d\theta,\\
 \mathcal R_{2,s,p,\tau}[u]
 &=\int_0^\pi(\sin\theta)^{\tau-sp}V(\theta)^p\,d\theta.
 \end{split}
\]
Translation invariance of the $L^p(\mathbb R)$ norm and its reverse
triangle inequality imply
\[
 \begin{split}
 &\left\|e^{-\alpha_{s,p}q/2}w(\cdot,\theta)
 -e^{\alpha_{s,p}q/2}w(\cdot-q,\varphi)\right\|_p\geq
 \left|e^{-\alpha_{s,p}q/2}V(\theta)
 -e^{\alpha_{s,p}q/2}V(\varphi)\right|.
 \end{split}
\]
Consequently,
\begin{equation}\label{eq:general-p-angular-contraction}
 \mathcal H_{2,s,p}[u]\geq\mathcal B_{s,p}[V].
\end{equation}

We first check that the angular variational infimum applies to $V$. The map
$\theta\mapsto w(\cdot,\theta)$ is smooth with values in $L^p(\mathbb R)$.
The norm is Lipschitz. Thus $V$ is Lipschitz and has compact support in
some $I\Subset(0,\pi)$. For functions supported in a fixed such interval,
\begin{equation}\label{eq:general-p-local-continuity}
 \mathcal A_{s,p}[f]
 \leq C_I\left(\|f\|_{L^p(\mathbb R)}^p+
 [f]_{W^{s,p}(\mathbb R)}^p\right),
\end{equation}
where $f$ is extended by zero. Indeed, for $|q|<1$, split the drifted
difference into $f(\theta)-f(\varphi)$ and an error bounded by
$C|q|(|f(\theta)|+|f(\varphi)|)$. Integration in $q$ of the first
kernel gives $C|\theta-\varphi|^{-1-sp}$. Integration in the angular
difference of the error kernel gives
$C|q|^{p-1-sp}$, which is integrable since $s<1$. For $|q|\geq1$,
\[
 \mathcal K_{s,p}(q,\theta,\varphi)
 e^{p|\alpha_{s,p}q|/2}
 \leq C e^{-\min\{2,sp\}|q|}.
\]
This proves \eqref{eq:general-p-local-continuity}. The $p$th root of
$\mathcal A_{s,p}$ is the norm of a linear difference map, so the same
bound proves its continuity under $W^{s,p}$ approximation. Mollification
therefore approximates $V$ by angular core functions, with convergence of
$\mathcal A_{s,p}$ and both weighted integrals. The angular infimum in
\eqref{eq:nonlinear-angular-constant}, denoted temporarily by $L$, applies
to $V$. Equation \eqref{eq:general-p-angular-contraction} gives
$\Lambda_{2,s,p,\tau}\geq L$.

For the reverse inequality, fix a nonzero angular core function $\psi$.
Choose $\chi\in C_c^\infty(\mathbb R)$ with $\|\chi\|_p=1$, and put
\[
 \chi_R(t)=R^{-1/p}\chi(t/R),\qquad
 w_R(t,\theta)=\chi_R(t)\psi(\theta).
\]
Its inverse polar transform belongs to $C_c^\infty(\mathbb R^2_+)$.
The Hardy and remainder integrals factor exactly. On the measure space
with measure $\mathcal K_{s,p}\,dt\,dq\,d\theta\,d\varphi$, write
\[
 \begin{split}
 &e^{-\alpha_{s,p}q/2}w_R(t,\theta)
 -e^{\alpha_{s,p}q/2}w_R(t-q,\varphi)=A_R+B_R,\\
 A_R&=\chi_R(t)
 \left(e^{-\alpha_{s,p}q/2}\psi(\theta)
 -e^{\alpha_{s,p}q/2}\psi(\varphi)\right),\\
 B_R&=e^{\alpha_{s,p}q/2}\psi(\varphi)
 (\chi_R(t)-\chi_R(t-q)).
 \end{split}
\]
Then $\|A_R\|_p^p=\mathcal A_{s,p}[\psi]$ and
\[
 \|\chi_R-\chi_R(\cdot-q)\|_p
 \leq C\min\{|q|/R,1\}.
\]
For $|q|<1$, integration of $\mathcal K_{s,p}$ in $\theta$ is bounded
by $C|q|^{-1-sp}$. The preceding exponential estimate controls the tail.
It follows that
\[
 \|B_R\|_p^p
 \leq\frac{C_\psi}{R^p}
 \left(\int_0^1q^{p-1-sp}\,dq+
 \int_1^\infty q^p e^{-\min\{2,sp\}q}\,dq\right)
 \longrightarrow0.
\]
Minkowski's inequality in \eqref{eq:general-p-cylinder} now gives
\[
 \mathcal H_{2,s,p}[u_R]\longrightarrow\mathcal B_{s,p}[\psi].
\]
Divide by the factored remainder and minimize over $\psi$. This proves
the opposite inequality.

Finally, replacing $\psi$ by $|\psi|$ decreases $\mathcal A_{s,p}$ and
preserves the weighted integrals. The modulus is a compactly supported
$W^{1,p}$ function. Nonnegative mollification and
\eqref{eq:general-p-local-continuity} give nonnegative angular core
approximants. This proves the last infimum assertion. For $p=2$, expand
the square and use the evenness of $\mathcal K_{s,2}=K_s$ in $q$.
The resulting expression is \eqref{eq:intro-fiber} with $\xi=0$.
\end{proof}

\begin{remark}
The variational reduction holds for all $p>1$. Attainment in the
nonlinear critical angular completion is left open. The spectral
realization, energy decomposition, and local limits below concern $p=2$.
\end{remark}

\section{Logarithmic polar coordinates and the closed energy}\label{sec:transform}\label{sec:coercivity}

Write $x\in\mathbb R^2_+$ as
\begin{equation}\label{eq:polar}
 x=e^t(\cos\theta,\sin\theta),\qquad
 t\in\mathbb R,\quad 0<\theta<\pi.
\end{equation}
The Jacobian is $dx=e^{2t}\,dt\,d\theta$. For
$y=e^{t'}(\cos\varphi,\sin\varphi)$, we have
\begin{equation}\label{eq:distance-polar}
 |x-y|^2=2e^{t+t'}
 \bigl(\cosh(t-t')-\cos(\theta-\varphi)\bigr).
\end{equation}

\begin{definition}
For $u\in C_c^\infty(\mathbb R^2_+)$, its logarithmic polar transform is
\begin{equation}\label{eq:EF-def}
 w(t,\theta)=e^{(1-s)t}u(e^t\cos\theta,e^t\sin\theta).
\end{equation}
\end{definition}

\begin{lemma}[Exact transform]\label{lem:exact-transform}
Let $u$ and $w$ be related by \eqref{eq:EF-def}. Then
\begin{equation}\label{eq:R-transform}
 \mathcal R_{s,\tau}[u]=
 \int_{\mathbb R}\int_0^\pi
 m_{s,\tau}(\theta)|w(t,\theta)|^2\,d\theta\,dt.
\end{equation}
Moreover,
\begin{equation}\label{eq:H-transform}
 \mathcal H_s[u]=\mathfrak q_s[w],
\end{equation}
where
\begin{equation}\label{eq:q-cylinder}
 \begin{split}
 \mathfrak q_s[w]={}&
 \int_{\mathbb R^2}\int_0^\pi\int_0^\pi
 K_s(t-t',\theta,\varphi)\\
 &\quad\times
 \left|e^{-a_s(t-t')/2}w(t,\theta)
 -e^{a_s(t-t')/2}w(t',\varphi)\right|^2
 \,d\theta\,d\varphi\,dt\,dt'\\
 &-\mathcal D_{2,s,2}
 \int_{\mathbb R}\int_0^\pi
 \frac{|w(t,\theta)|^2}{(\sin\theta)^{2s}}\,d\theta\,dt.
 \end{split}
\end{equation}
\end{lemma}

\begin{proof}
By \eqref{eq:EF-def},
\[
 u(e^t\cos\theta,e^t\sin\theta)=e^{-(1-s)t}w(t,\theta).
\]
Since $x_2=e^t\sin\theta$ and $|x|=e^t$, substitution gives
\[
 \frac{e^{-2(1-s)t}|w|^2e^{2t}}
 {e^{(2s-\tau)t}(\sin\theta)^{2s-\tau}e^{\tau t}}
 =m_{s,\tau}(\theta)|w|^2.
\]
This proves \eqref{eq:R-transform}. The Hardy term is
\begin{equation}\label{eq:hardy-transform}
 \int_{\mathbb R^2_+}\frac{|u|^2}{x_2^{2s}}\,dx
 =\int_{\mathbb R}\int_0^\pi
 \frac{|w|^2}{(\sin\theta)^{2s}}\,d\theta\,dt.
\end{equation}

For the energy, \eqref{eq:distance-polar} gives
\[
 \frac{e^{2t+2t'}}
 {2^{1+s}e^{(1+s)(t+t')}
 (\cosh(t-t')-\cos(\theta-\varphi))^{1+s}}
 =e^{a_s(t+t')}K_s(t-t',\theta,\varphi).
\]
Furthermore,
\[
 e^{a_s(t+t')/2}(u(x)-u(y))=
 e^{-a_s(t-t')/2}w(t,\theta)
 -e^{a_s(t-t')/2}w(t',\varphi).
\]
Thus \eqref{eq:q-cylinder} and \eqref{eq:hardy-transform} imply
\eqref{eq:H-transform}.
\end{proof}

The following positive representation allows us to close the Hardy
deficit. On the completed space, the two terms in
\eqref{eq:intro-deficit} need not be finite separately.

\begin{lemma}[Ground-state representation]
Set
\[
 \omega_s(x)=x_2^{s-1/2},\qquad v(x)=\omega_s(x)^{-1}u(x).
\]
For $u\in C_c^\infty(\mathbb R^2_+)$,
\begin{equation}\label{eq:ground-state}
 \mathcal H_s[u]=
 \int_{\mathbb R^2_+}\int_{\mathbb R^2_+}
 \frac{|v(x)-v(y)|^2\omega_s(x)\omega_s(y)}
 {|x-y|^{2+2s}}\,dx\,dy.
\end{equation}
Consequently, if $q=t-t'$ and
$z(t,\theta)=(\sin\theta)^{1/2-s}w(t,\theta)$, then
\begin{equation}\label{eq:positive-cylinder}
 \begin{split}
 \mathfrak q_s[w]={}&
 \int_{\mathbb R^2}\int_0^\pi\int_0^\pi
 K_s(q,\theta,\varphi)(\sin\theta\sin\varphi)^{s-1/2}\\
 &\quad\times
 \left|e^{-q/4}z(t,\theta)-e^{q/4}z(t',\varphi)\right|^2
 \,d\theta\,d\varphi\,dt\,dt'.
 \end{split}
\end{equation}
Both right-hand sides are nonnegative.
\end{lemma}

\begin{proof}
We first prove that $\omega_s$ satisfies
\begin{equation}\label{eq:ground-state-equation}
 2\,\operatorname{p.v.}\int_{\mathbb R^2_+}
 \frac{\omega_s(x)-\omega_s(y)}{|x-y|^{2+2s}}\,dy
 =\mathcal D_{2,s,2}x_2^{-2s}\omega_s(x).
\end{equation}
Integration in the tangential variable gives
\[
 \int_{\mathbb R}\frac{dz}{(z^2+a^2)^{1+s}}
 =\frac{\sqrt\pi\,\Gamma(s+\frac12)}{\Gamma(1+s)}
 |a|^{-1-2s}.
\]
To justify this reduction at the singularity, fix $r_0<x_2/2$ and
subtract $\nabla\omega_s(x)\cdot(y-x)$ on $B_{r_0}(x)$. The remainder
is $O(|x-y|^2)$, and its integral over $B_\varepsilon(x)$ is
$O(\varepsilon^{2-2s})$. The linear term vanishes under both the
Euclidean cutoff and the symmetric one-dimensional cutoff. Outside
$B_{r_0}(x)$, the integrals converge absolutely. Hence the two principal
values agree.
Put $b=s-\frac12$ and
\[
 C_s=\frac{\sqrt\pi\,\Gamma(s+\frac12)}{\Gamma(1+s)}.
\]
After integration and the substitution $y_2=x_2r$, the left-hand side
of \eqref{eq:ground-state-equation}, divided by two, equals
\[
 C_sx_2^{b-2s}\operatorname{p.v.}\int_0^\infty
 \frac{1-r^b}{|1-r|^{1+2s}}\,dr.
\]
Split the integral at one and invert the variable in the second part.
Since $2s-1=2b$, we obtain
\begin{equation*}
 \begin{split}
 \operatorname{p.v.}\int_0^\infty
 \frac{1-r^b}{|1-r|^{1+2s}}\,dr
 &=\int_0^1\frac{1-2r^b+r^{2b}}
 {(1-r)^{1+2s}}\,dr\\
 &=\int_0^1\frac{(1-r^b)^2}{(1-r)^{1+2s}}\,dr.
 \end{split}
\end{equation*}
The cutoff error caused by inversion is $O(\varepsilon^{2-2s})$.
The numerator has a double zero at one, and $2b=2s-1>-1$.
Thus the last integral is finite. Formula \eqref{eq:intro-hardyc}
now gives \eqref{eq:ground-state-equation}.

For the truncated energies, use the identity
\[
 \begin{split}
 &|\omega_s(x)v(x)-\omega_s(y)v(y)|^2
 -\omega_s(x)\omega_s(y)|v(x)-v(y)|^2\\
 &\quad=\omega_s(x)(\omega_s(x)-\omega_s(y))|v(x)|^2
 +\omega_s(y)(\omega_s(y)-\omega_s(x))|v(y)|^2.
 \end{split}
\]
Interchanging $x$ and $y$ in the second term, the difference of the
truncated energies becomes
\[
 2\int_{\mathbb R^2_+}|v(x)|^2\omega_s(x)
 \int_{\substack{y\in\mathbb R^2_+\\|x-y|>\varepsilon}}
 \frac{\omega_s(x)-\omega_s(y)}{|x-y|^{2+2s}}\,dy\,dx.
\]
Let $K=\operatorname{supp}u$ and choose
$r_0<\frac12\operatorname{dist}(K,\partial\mathbb R^2_+)$. For
$x\in K$, put $h=y-x$. For $\varepsilon<r_0$, oddness gives
\[
 \begin{split}
 \int_{\varepsilon<|h|<r_0}
 \frac{\omega_s(x)-\omega_s(x+h)}{|h|^{2+2s}}\,dh =-\int_{\varepsilon<|h|<r_0}
 \frac{\omega_s(x+h)-\omega_s(x)-\nabla\omega_s(x)\cdot h}
 {|h|^{2+2s}}\,dh.
 \end{split}
\]
The last integral is bounded uniformly on $K$ by
\[
 C\int_0^{r_0}r^{1-2s}\,dr<\infty.
\]
On $|x-y|\geq r_0$, absolute integrability follows from
$s-\frac12>-1$ at the boundary and $s-\frac12<2s$ at infinity,
uniformly for $x\in K$. The same Taylor estimate shows that the
truncated potentials converge uniformly on $K$. We may therefore let
$\varepsilon\downarrow0$ in the preceding identity. By
\eqref{eq:ground-state-equation}, the potential term converges to the
Hardy term; the positive energy converges by monotone convergence.
This proves \eqref{eq:ground-state}.

Finally,
\[
 v(e^t\cos\theta,e^t\sin\theta)
 =e^{-t/2}(\sin\theta)^{1/2-s}w(t,\theta)=e^{-t/2}z(t,\theta).
\]
Furthermore,
\[
 \frac{\omega_s(x)\omega_s(y)\,dx\,dy}{|x-y|^{2+2s}}
 =e^{(t+t')/2}K_s(q,\theta,\varphi)
 (\sin\theta\sin\varphi)^{s-1/2}
 \,dt\,dt'\,d\theta\,d\varphi.
\]
Factoring $e^{-(t+t')/4}$ from $v(x)-v(y)$ cancels the radial weight
and gives \eqref{eq:positive-cylinder}.
\end{proof}

\begin{lemma}[Local kernel comparison]\label{lem:kernel-local}
If $|q|\leq1$ and $|\theta-\varphi|\leq1$, then
\begin{equation}\label{eq:kernel-comparison}
 c(q^2+|\theta-\varphi|^2)
 \leq \cosh q-\cos(\theta-\varphi)
 \leq C(q^2+|\theta-\varphi|^2).
\end{equation}
If $|q|\geq1$, then
\begin{equation}\label{eq:kernel-tail}
 \begin{aligned}
 K_s(q,\theta,\varphi)&\leq Ce^{-(1+s)|q|},\\
 K_s(q,\theta,\varphi)e^{(1-s)|q|}
 &\leq Ce^{-2s|q|},\\
 K_s(q,\theta,\varphi)e^{|q|/2}
 &\leq Ce^{-(s+1/2)|q|}.
 \end{aligned}
\end{equation}
Consequently, if $I'\Subset I\Subset(0,\pi)$, then
\begin{equation}\label{eq:local-cylinder-control}
 [w]_{H^s(\mathbb R\times I')}^2
 \leq C_{I',I}\bigl(\mathfrak q_s[w]
 +\|w\|_{L^2(\mathbb R\times I)}^2\bigr)
\end{equation}
for $w\in C_c^\infty(\mathbb R\times I)$. Here
\[
 [w]_{H^s(\mathbb R\times I')}^2=
 \int_{\mathbb R\times I'}\int_{\mathbb R\times I'}
 \frac{|w(t,\theta)-w(t',\varphi)|^2}
 {\bigl((t-t')^2+(\theta-\varphi)^2\bigr)^{1+s}}
 \,dt\,d\theta\,dt'\,d\varphi.
\]
\end{lemma}

\begin{proof}
Taylor's formula gives
\[
 \cosh q-1=\frac{q^2}{2}+O(q^4),
 \qquad 1-\cos\alpha=\frac{\alpha^2}{2}+O(\alpha^4).
\]
Since
$\cosh q-\cos\alpha=(\cosh q-1)+(1-\cos\alpha)$,
\eqref{eq:kernel-comparison} follows. For $|q|\geq1$,
$\cosh q-\cos\alpha\geq ce^{|q|}$. This proves
\eqref{eq:kernel-tail}, including the estimate with the drift factor.

For $\theta,\varphi\in I'$, absorb the angular weight in
\eqref{eq:positive-cylinder} into the square. The coefficients are smooth
and bounded above and below. For $|q|<1$, the mean-value theorem gives
\[
 \begin{split}
 &|e^{-q/4}(\sin\varphi/\sin\theta)^{(s-1/2)/2}w(t,\theta)
 -e^{q/4}(\sin\theta/\sin\varphi)^{(s-1/2)/2}
 w(t-q,\varphi)|^2\\
 &\quad\geq\frac12|w(t,\theta)-w(t-q,\varphi)|^2
 -C(q^2+|\theta-\varphi|^2)
 (|w(t,\theta)|^2+|w(t-q,\varphi)|^2).
 \end{split}
\]
By \eqref{eq:kernel-comparison}, the error is integrable near the
origin, since
\[
 \int_{q^2+\alpha^2<1}(q^2+\alpha^2)^{-s}\,dq\,d\alpha<\infty.
\]
Outside this neighborhood, the flat seminorm is bounded by an
$L^2(\mathbb R\times I)$ term. Restricting the positive cylinder integral
to $\mathbb R^2\times(I')^2$ proves
\eqref{eq:local-cylinder-control}.
\end{proof}

\subsection*{The closed deficit space}

For $\rho>0$, the normalized dilation and its transform are
\[
 u^\rho(x)=\rho^{1-s}u(\rho x),\qquad
 w^\rho(t,\theta)=w(t+\log\rho,\theta).
\]
Both $\mathcal H_s$ and $\mathcal R_{s,\tau}$ are invariant under this
normalization.

\begin{lemma}[Core coercivity]
Let $\mathscr C=C_c^\infty(\mathbb R\times(0,\pi))$. The image under
\eqref{eq:EF-def} of $C_c^\infty(\mathbb R^2_+)$ is $\mathscr C$. The form
$\mathfrak q_s$ is nonnegative and closable. For every $\sigma>0$,
\begin{equation}\label{eq:cylinder-coercive}
 \mathfrak q_s[w]\geq c_{s,\sigma}
 \int_{\mathbb R}\int_0^\pi
 (\sin\theta)^{\sigma-2s}|w|^2\,d\theta\,dt
\end{equation}
with $c_{s,\sigma}>0$.
\end{lemma}

\begin{proof}
The map in \eqref{eq:polar} is a diffeomorphism from
$\mathbb R\times(0,\pi)$ onto $\mathbb R^2_+$ and preserves compact
support in both directions. Nonnegativity follows from
\eqref{eq:positive-cylinder}.

Fix $\sigma>0$ and choose
$0<\sigma_0<\min\{\sigma,1+2s\}$. By
\cite[Theorem~1]{KijaczkoSzczukiewicz}, \eqref{eq:intro-KS} holds with
$c_{s,\sigma_0}>0$ in the ordered-integral normalization of
\eqref{eq:intro-energy}. Since $m_{s,\sigma}\leq m_{s,\sigma_0}$,
Lemma~\ref{lem:exact-transform} gives \eqref{eq:cylinder-coercive}.
For complex $w$, apply the inequality to its real and imaginary parts.

Let $d\nu_s$ be the positive measure in
\eqref{eq:positive-cylinder}, including the kernel and angular factors,
and put
\[
 (Tw)(t,t',\theta,\varphi)=
 e^{-(t-t')/4}(\sin\theta)^{1/2-s}w(t,\theta)
 -e^{(t-t')/4}(\sin\varphi)^{1/2-s}w(t',\varphi).
\]
Then $\mathfrak q_s[w]=\|Tw\|_{L^2(d\nu_s)}^2$ on $\mathscr C$.
Suppose that $w_n\to0$ in $L^2(m_{s,\tau})$ and $Tw_n\to F$ in
$L^2(d\nu_s)$. The measure $d\nu_s$ is $\sigma$-finite and absolutely
continuous off the null diagonal. Passing to subsequences gives
$w_n\to0$ almost everywhere and $Tw_n\to F$ $d\nu_s$-almost everywhere.
Product-null sets are $d\nu_s$-null. The formula for $T$ therefore gives
$F=0$. Thus $T$, and hence $\mathfrak q_s$, is closable.
\end{proof}

Denote the closed form domain by $\mathscr X_s$. It is the completion of
$\mathscr C$ under
\[
 \|w\|_{\mathscr X_s}^2=
 \mathfrak q_s[w]+\int_{\mathbb R}\int_0^\pi
 m_{s,\tau}|w|^2\,d\theta\,dt.
\]
By \eqref{eq:cylinder-coercive}, this norm is equivalent to
$\mathfrak q_s^{1/2}$ for every $\tau>0$. Thus the abstract completion
is independent of $\tau$. Its representatives also agree: a common core
sequence converges in the stronger weight corresponding to any
$0<\sigma<\min\{\tau_1,\tau_2\}$. Consequently,
$\mathscr X_s\hookrightarrow L^2(m_{s,\tau})$ continuously for every
$\tau>0$. By Lemma~\ref{lem:exact-transform} and density, the logarithmic
polar map extends to an isometric isomorphism from the energy completion
in Theorem~\ref{thm:A} onto $\mathscr X_s$.

\section{Mellin fibers and angular compactness}\label{sec:fibers}\label{sec:kernel}

For $w\in\mathscr C$, take the unitary Fourier transform in $t$:
\[
 \widehat w(\xi,\theta)=\frac1{\sqrt{2\pi}}
 \int_{\mathbb R}e^{-it\xi}w(t,\theta)\,dt.
\]

Put $H_{s,\tau}=L^2((0,\pi),m_{s,\tau}d\theta)$. For
$\psi\in C_c^\infty(0,\pi)$, the positive fiber form is
\begin{equation}\label{eq:positive-fiber}
 \begin{split}
 \mathfrak g_{s,\xi}[\psi]={}&
 \int_0^\pi\int_0^\pi\int_{\mathbb R}
 K_s(q,\theta,\varphi)(\sin\theta\sin\varphi)^{s-1/2}\\
 &\quad\times\left|
 e^{-q/4}(\sin\theta)^{1/2-s}\psi(\theta)
 -e^{q/4-i\xi q}(\sin\varphi)^{1/2-s}\psi(\varphi)
 \right|^2\,dq\,d\theta\,d\varphi.
 \end{split}
\end{equation}
We compare the two expressions with a common kernel cutoff. The
calculation below uses \eqref{eq:ground-state-equation} and gives
\begin{equation}\label{eq:two-fiber-representations}
 \mathfrak g_{s,\xi}[\psi]=\mathfrak b_{s,\xi}[\psi].
\end{equation}
Put $b=s-\frac12$, and let
$K_s^\varepsilon$ be $K_s$ multiplied by the indicator of
$q^2+|\theta-\varphi|^2>\varepsilon^2$. Expansion of the positive square
gives
\[
 \begin{split}
 \mathfrak g_{s,\xi}^\varepsilon[\psi]
 =\int K_s^\varepsilon\bigg\{&e^{-q/2}
 \left(\frac{\sin\varphi}{\sin\theta}\right)^b|\psi(\theta)|^2
 +e^{q/2}\left(\frac{\sin\theta}{\sin\varphi}\right)^b
 |\psi(\varphi)|^2\\
 &-2\operatorname{Re}\left(e^{i\xi q}\psi(\theta)
 \overline{\psi(\varphi)}\right)\bigg\}
 \,dq\,d\theta\,d\varphi.
 \end{split}
\]
The transformed ground-state equation reads
\begin{equation}\label{eq:angular-ground-state-cancellation}
 \begin{split}
 2\lim_{\varepsilon\downarrow0}\int_0^\pi\int_{\mathbb R}
 K_s^\varepsilon(q,\theta,\varphi)\bigg[
 e^{-q/2}\left(\frac{\sin\varphi}{\sin\theta}\right)^b
 -\cosh(a_sq)\bigg]\,dq\,d\varphi
 =-\mathcal D_{2,s,2}(\sin\theta)^{-2s}.
 \end{split}
\end{equation}
Indeed, take $x=(\cos\theta,\sin\theta)$ in
\eqref{eq:ground-state-equation} and put
$y=e^{-q}(\cos\varphi,\sin\varphi)$. Before symmetrization in $q$, the
two terms in brackets are $e^{-q/2}(\sin\varphi/\sin\theta)^b$ and
$e^{-a_sq}$. The evenness of $K_s$ changes the latter into
$\cosh(a_sq)$. This proves
\eqref{eq:angular-ground-state-cancellation}.

The angular cutoff agrees with the Euclidean principal value in
\eqref{eq:ground-state-equation}. Indeed, put
$\alpha=\theta-\varphi$ and $r=(q^2+\alpha^2)^{1/2}$. For fixed
$\theta\in(0,\pi)$,
\[
 |x-y|^2=2e^{-q}(\cosh q-\cos\alpha)
 =r^2+O(r^3).
\]
The numerator in the transformed principal-value integral is a linear
function of $(q,\alpha)$ plus $O(r^2)$. The linear part has zero integral
over every disk centered at the origin because $K_s$ is even in $q$ and
$\alpha$. The symmetric difference between the pulled-back Euclidean
cutoff and the disk $r>\varepsilon$ is contained in an annulus of width
$O(\varepsilon^2)$. Its contribution from the linear part is
$O(\varepsilon^{2-2s})$. The quadratic remainder on $r<2\varepsilon$ is also
$O(\varepsilon^{2-2s})$. Both tend to zero because $s<1$.
Thus the two principal values agree.

In the first two diagonal terms of
$\mathfrak g_{s,\xi}^\varepsilon$, interchange
$(q,\theta,\varphi)$ with $(-q,\varphi,\theta)$ in one copy. Their sum has
twice the first displayed coefficient. Apply
\eqref{eq:angular-ground-state-cancellation}. The cross term already
equals the cross term in \eqref{eq:intro-fiber}. Letting
$\varepsilon\downarrow0$ proves
\eqref{eq:two-fiber-representations}. Thus
\eqref{eq:intro-fiber} is understood as the limit of the complete
truncated expression. Formula \eqref{eq:positive-fiber} is used on the
completed domain.

\begin{proposition}[Direct-integral identity]\label{prop:direct-integral}
For $w\in\mathscr C$,
\begin{equation}\label{eq:direct-integral}
 \mathfrak q_s[w]=\int_{\mathbb R}
 \mathfrak b_{s,\xi}[\widehat w(\xi,\cdot)]\,d\xi.
\end{equation}
Also,
\begin{equation}\label{eq:plancherel-weight}
 \int_{\mathbb R}\int_0^\pi m_{s,\tau}|w|^2
 =\int_{\mathbb R}\int_0^\pi m_{s,\tau}|\widehat w|^2.
\end{equation}
\end{proposition}

\begin{proof}
Let $q=t-t'$. Delete
$q^2+|\theta-\varphi|^2<\varepsilon^2$ from the positive form
\eqref{eq:positive-cylinder}. Denote the result by
$\mathfrak q_s^\varepsilon[w]$. Denote by
$\mathfrak g_{s,\xi}^\varepsilon$ the same truncation of
\eqref{eq:positive-fiber}. The truncated kernel is locally bounded, and
its $q$-tail is exponentially integrable. Because $w\in\mathscr C$, its
angular support is separated from the endpoints. In a diagonal term, only
one angular variable can leave this support. The remaining factor is a
power $(\sin\theta)^{s-1/2}$. It is integrable because $s-1/2>-1$.
Estimate \eqref{eq:kernel-tail} controls the factors $e^{\pm q/2}$.
Thus the truncated expression is absolutely integrable. Fubini's theorem
and Plancherel's theorem in $t$ give
\[
 \mathfrak q_s^\varepsilon[w]
 =\int_{\mathbb R}\mathfrak g_{s,\xi}^\varepsilon
 [\widehat w(\xi,\cdot)]\,d\xi.
\]
By monotone convergence,
\[
 \mathfrak q_s[w]
 =\int_{\mathbb R}\mathfrak g_{s,\xi}
 [\widehat w(\xi,\cdot)]\,d\xi.
\]
Use \eqref{eq:two-fiber-representations} to obtain
\eqref{eq:direct-integral}. Weighted Plancherel gives
\eqref{eq:plancherel-weight}.
\end{proof}

\begin{lemma}[Closed angular forms]\label{lem:closed-fibers}
For every $\xi\in\mathbb R$, the closure of $\mathfrak b_{s,\xi}$ is a closed form in
$H_{s,\tau}$. All fibers have the same form domain,
denoted by $\mathscr Y_s$. There is $c>0$, independent of $\xi$, such that
\begin{equation}\label{eq:fiber-positive}
 \mathfrak b_{s,\xi}[\psi]\geq c
 \int_0^\pi m_{s,\tau}|\psi|^2\,d\theta.
\end{equation}
The embedding
\begin{equation}\label{eq:compact-embedding}
 \mathscr Y_s\hookrightarrow
 L^2((0,\pi),m_{s,\tau}d\theta)
\end{equation}
is compact.
\end{lemma}

\begin{proof}
Let $T_\xi\psi$ denote the expression between absolute-value signs in
\eqref{eq:positive-fiber}. Suppose that $\psi_n\to0$ in $H_{s,\tau}$ and
$T_\xi\psi_n\to F$ in the target kernel space. Pass to subsequences which
converge almost everywhere. The target measure is $\sigma$-finite and is
absolutely continuous off the null diagonal. The pointwise formula for
$T_\xi$ gives $F=0$. Thus $T_\xi$ is closable. The Cauchy-sequence
criterion then shows that $\mathfrak b_{s,\xi}=\|T_\xi\cdot\|^2$ is
closable.

Fix
$\xi_0\in\mathbb R$ and $\psi\in C_c^\infty(0,\pi)$. Let
$\chi\in C_c^\infty(\mathbb R)$ satisfy $\|\chi\|_2=1$. Put
\[
 w_R(t,\theta)=R^{-1/2}\chi(t/R)e^{i\xi_0t}\psi(\theta).
\]
Then $w_R\in\mathscr C$ and
\[
 \widehat w_R(\xi,\theta)=R^{1/2}
 \widehat\chi(R(\xi-\xi_0))\psi(\theta).
\]
Apply \eqref{eq:cylinder-coercive} with $\sigma=\tau$.
Proposition~\ref{prop:direct-integral} yields
\[
 \int_{\mathbb R}R|\widehat\chi(R(\xi-\xi_0))|^2
 \mathfrak b_{s,\xi}[\psi]\,d\xi
 \geq c_{s,\tau}\int_0^\pi m_{s,\tau}|\psi|^2.
\]
By Lemma~\ref{lem:kernel-local}, subtraction of the zero fiber and
integration in $\theta-\varphi$ give
\[
 |\mathfrak b_{s,\eta}[\psi]-\mathfrak b_{s,0}[\psi]|
 \leq C_\psi
 \begin{cases}
  |\eta|^2,&|\eta|\leq1,\\
  |\eta|^{2s},&|\eta|>1.
 \end{cases}
\]
The local integral is
$\int_0^1\min\{\eta^2q^2,1\}q^{-1-2s}\,dq$, and the tail is
exponential. The same cancellation proves continuity in $\eta$. Thus
\[
 \mathfrak b_{s,\xi_0+\zeta/R}[\psi]
 \leq C_{\psi,\xi_0}(1+|\zeta|^{2s}),\qquad R\geq1.
\]
Change variables to $\zeta=R(\xi-\xi_0)$ and apply dominated
convergence against $|\widehat\chi(\zeta)|^2$. This proves
\eqref{eq:fiber-positive} on the core. With any $\sigma>0$, it gives
\begin{equation}\label{eq:all-weight-coercivity}
 \mathfrak b_{s,\xi}[\psi]\geq c_{s,\sigma}
 \int_0^\pi m_{s,\sigma}|\psi|^2\,d\theta.
\end{equation}

Choose $0<\tau_0<\min\{\tau,2s\}$. Since
$m_{s,\tau_0}\geq1$ on $(0,\pi)$, we have
\begin{equation}\label{eq:unweighted-control}
 \|\psi\|_{L^2(0,\pi)}^2
 \leq \int_0^\pi m_{s,\tau_0}|\psi|^2\,d\theta
 \leq C\mathfrak b_{s,\xi}[\psi].
\end{equation}
For $\alpha\in(-\pi,\pi)$, set
\[
 H_{s,\xi}(\alpha)=2\int_{\mathbb R}
 (1-\cos(\xi q))K_s(q,\alpha)\,dq,
\]
where $K_s(q,\alpha)=2^{-1-s}(\cosh q-\cos\alpha)^{-1-s}$.
For $|q|<1$, Lemma~\ref{lem:kernel-local} gives
\[
 \int_{-\pi}^{\pi}K_s(q,\alpha)\,d\alpha
 \leq C|q|^{-1-2s}.
\]
For $|q|\geq1$, the kernel has an exponential tail. It follows that
\begin{equation}\label{eq:H-kernel-L1}
 \|H_{s,\xi}\|_{L^1(-\pi,\pi)}
 \leq C\begin{cases}
 \xi^2,& |\xi|\leq1,\\
 |\xi|^{2s},& |\xi|\geq1.
 \end{cases}
\end{equation}
Young's inequality therefore gives
\begin{equation}\label{eq:fiber-difference-bound}
 |\mathfrak b_{s,\xi}[\psi]-\mathfrak b_{s,0}[\psi]|
 \leq\|H_{s,\xi}\|_1\|\psi\|_{L^2(0,\pi)}^2.
\end{equation}
Use \eqref{eq:unweighted-control} first with $\xi=0$ and then with the
given $\xi$. The graph norms of the closures of
$\mathfrak b_{s,\xi}$ and $\mathfrak b_{s,0}$ are equivalent for each
fixed $\xi$. Hence all closed fibers have one domain $\mathscr Y_s$.

For a core sequence $\psi_n\to\psi$ in $\mathscr Y_s$,
\eqref{eq:all-weight-coercivity} makes $\psi_n$ Cauchy in
$L^2(m_{s,\sigma})$. Its limit equals $\psi$ almost everywhere.
Thus \eqref{eq:all-weight-coercivity} and
\eqref{eq:unweighted-control} hold on the closed domain.

For two base exponents $\tau_1,\tau_2>0$, choose
$0<\sigma<\min\{\tau_1,\tau_2,2s\}$. By
\eqref{eq:all-weight-coercivity}, both completions embed into
$L^2(m_{s,\sigma})$ with the same representatives. The form controls
both base norms, so the completions coincide. Hence $\mathscr Y_s$ is
independent of $\tau$.

For $I'\Subset I\Subset(0,\pi)$ and
$\psi\in C_c^\infty(0,\pi)$, choose
$\eta\in C_c^\infty(I)$ with $\eta=1$ on $I'$. The angular multipliers in
\eqref{eq:positive-fiber} are smooth and bounded above and below on $I$.
For $|q|<1$ and $\theta,\varphi\in I$,
\[
 \begin{split}
 |\psi(\theta)-\psi(\varphi)|^2
 \leq{}&C\left|e^{-q/4}(\sin\theta)^{1/2-s}\psi(\theta)
 -e^{q/4}(\sin\varphi)^{1/2-s}\psi(\varphi)\right|^2\\
 &+C(q^2+|\theta-\varphi|^2)
 (|\psi(\theta)|^2+|\psi(\varphi)|^2).
 \end{split}
\]
Put $b=s-1/2$. In the positive square, write
\[
 \eta(\theta)A(\theta)-\eta(\varphi)B(\varphi)
 =\eta(\theta)(A(\theta)-B(\varphi))
 +(\eta(\theta)-\eta(\varphi))B(\varphi).
\]
The first term is controlled by $C\mathfrak g_{s,0}[\psi]$. For the
second term, split $|\theta-\varphi|<c_I$ and its complement. On the first
set, a nonzero cutoff difference forces both angular variables into a fixed
compact subinterval of $(0,\pi)$. The Lipschitz estimate for $\eta$ and
\eqref{eq:kernel-comparison} give an error bounded by $C\|\psi\|_2^2$.
On the complement, the angular kernel is bounded after integration in
$q$. The sine factors in \eqref{eq:positive-fiber} give the bound
\[
 C\int_0^\pi\{1+(\sin\varphi)^{-(s-1/2)_+}\}
 |\psi(\varphi)|^2\,d\varphi.
\]
Set $\sigma_* =\min\{2s,s+1/2\}$. Then
$m_{s,\sigma_*}=(\sin\varphi)^{-(s-1/2)_+}$. Estimate
\eqref{eq:all-weight-coercivity}, with $\sigma=\sigma_*$, yields
\[
 \mathfrak g_{s,0}[\eta\psi]
 \leq C_{I',I}\mathfrak b_{s,0}[\psi].
\]
The error in the preceding pointwise comparison is integrable because
\[
 \int_{q^2+\alpha^2<1}
 (q^2+\alpha^2)^{-s}\,dq\,d\alpha<\infty.
\]
Finally,
\[
 \int_{-1}^1(q^2+\alpha^2)^{-1-s}\,dq
 \geq c|\alpha|^{-1-2s}
\]
for $0<|\alpha|<1/2$. Therefore
\begin{equation}\label{eq:local-angular}
 \|\psi\|_{H^s(I')}^2
 \leq C_{I',I}\left(
 \mathfrak b_{s,0}[\psi]+\|\psi\|_{L^2(0,\pi)}^2\right).
\end{equation}
The same estimates applied to a core sequence show that
$\eta\psi_n\to\eta\psi$ in $\mathscr Y_s$. Lower semicontinuity gives
\eqref{eq:local-angular} for every $\psi\in\mathscr Y_s$.

For compactness, estimate
\eqref{eq:all-weight-coercivity} with $\sigma=\tau_0$ gives
\[
 \int_0^\pi(\sin\theta)^{\tau_0-2s}|\psi|^2\,d\theta
 \leq C\mathfrak b_{s,0}[\psi].
\]
Let $\{\psi_n\}$ be bounded in $\mathscr Y_s$. On
$[\delta,\pi-\delta]$, \eqref{eq:local-angular} gives a uniform
$H^s$ bound. The Rellich theorem yields a subsequence convergent
in $L^2([\delta,\pi-\delta])$. At the endpoints,
\begin{equation}\label{eq:endpoint-tightness}
 \begin{split}
 \int_{(0,\delta)\cup(\pi-\delta,\pi)}
 m_{s,\tau}|\psi_n|^2\,d\theta
 &\leq C\delta^{\tau-\tau_0}
 \int_0^\pi(\sin\theta)^{\tau_0-2s}|\psi_n|^2\,d\theta\\
 &\leq C\delta^{\tau-\tau_0}.
 \end{split}
\end{equation}
Together with interior convergence, this uniform endpoint bound proves
\eqref{eq:compact-embedding} by a diagonal argument.
\end{proof}

The angular space is the closure of the core for the nonnegative
ground-state form. We always use this closure, hence its Friedrichs
operator realization. In particular, no unweighted boundary trace
condition is imposed separately. The two terms in the Hardy difference
need not be finite on $\mathscr Y_s$.

The associated self-adjoint operator is defined by
\begin{equation}\label{eq:angular-operator-domain}
 \begin{split}
 \mathcal D(B_s(\xi))
 =\{\psi\in\mathscr Y_s:\;&\text{there exists }f\in H_{s,\tau}
 \text{ such that}\\
 &\mathfrak b_{s,\xi}(\psi,\eta)
 =\langle f,\eta\rangle_{H_{s,\tau}}
 \text{ for every }\eta\in\mathscr Y_s\},\\
 B_s(\xi)\psi&=f.
 \end{split}
\end{equation}
Density of the form domain makes $f$ unique. This is the operator used in
all eigenvalue, resolvent, and spectral projection statements below.

\begin{lemma}[Continuity of the angular resolvents]
\label{lem:real-frequency-continuity}
The map $\xi\mapsto(B_s(\xi)+1)^{-1}$ is continuous from $\mathbb R$
to $\mathcal L(H_{s,\tau},\mathscr Y_s)$. In particular it is continuous
in operator norm on $H_{s,\tau}$, and every ordered eigenvalue
$\lambda_j(\xi)$ is continuous and even on $\mathbb R$.
\end{lemma}

\begin{proof}
Fix a compact frequency interval $K$. The kernel estimate leading to
\eqref{eq:H-kernel-L1} also gives
\[
 |\mathfrak b_{s,\xi}(u,v)-\mathfrak b_{s,\eta}(u,v)|
 \leq\omega_K(|\xi-\eta|)\|u\|_{L^2(0,\pi)}\|v\|_{L^2(0,\pi)},
 \qquad\omega_K(r)\longrightarrow0.
\]
Indeed, for $\xi,\eta\in K$ and $|q|<1$, the mean-value formula gives
\[
 |\cos(\xi q)-\cos(\eta q)|
 \leq C_K|\xi-\eta|q^2.
\]
For $|q|\geq1$, the difference is bounded by $2$. Hence
\[
 \begin{split}
 &\int_{-\pi}^{\pi}\int_{\mathbb R}
 |\cos(\xi q)-\cos(\eta q)|K_s(q,\alpha)\,dq\,d\alpha\\
 &\quad\leq C_K|\xi-\eta|\int_0^1q^{1-2s}\,dq
 +C\int_1^\infty
 \min\{|\xi-\eta|q,2\}e^{-(1+s)q}\,dq\longrightarrow0.
 \end{split}
\]
Young's inequality applied to the polarized cosine difference proves
the asserted modulus of continuity.

Let $\mathcal A_\xi:\mathscr Y_s\to\mathscr Y_s^*$ be the map of
$\mathfrak b_{s,\xi}+\langle\cdot,\cdot\rangle_{H_{s,\tau}}$.
The graph norms are uniformly equivalent on $K$ by
\eqref{eq:fiber-difference-bound} and \eqref{eq:unweighted-control},
applied also at the frequency $\xi$. Thus $\mathcal A_\xi$ is a
coercive isomorphism, with inverse uniformly bounded on $K$.
The preceding estimate and \eqref{eq:unweighted-control} imply
\[
 \|\mathcal A_\xi-\mathcal A_\eta\|
 _{\mathscr Y_s\to\mathscr Y_s^*}\longrightarrow0
 \quad\text{as }\xi-\eta\to0,\qquad \xi,\eta\in K.
\]
The inverse identity
\[
 \mathcal A_\xi^{-1}-\mathcal A_\eta^{-1}
 =-\mathcal A_\xi^{-1}(\mathcal A_\xi-\mathcal A_\eta)
 \mathcal A_\eta^{-1}
\]
gives continuity from $H_{s,\tau}$ to $\mathscr Y_s$ after restriction.
By \eqref{eq:angular-operator-domain}, the restricted inverse is
$(B_s(\xi)+1)^{-1}$. Compactness of the resolvents and the min--max
principle give continuity of their ordered eigenvalues, hence of
$\lambda_j(\xi)$. The cosine formula makes $B_s(-\xi)=B_s(\xi)$,
which proves evenness.
\end{proof}

\begin{proposition}[Closed direct integral]
The closure of $\mathfrak q_s$ has domain
\[
 \left\{w\in L^2(\mathbb R;H_{s,\tau}):
 \widehat w(\xi,\cdot)\in\mathscr Y_s\text{ for a.e. }\xi,
 \quad\int_{\mathbb R}\mathfrak b_{s,\xi}
 [\widehat w(\xi,\cdot)]\,d\xi<\infty\right\}.
\]
For every function in this domain,
\begin{equation}\label{eq:closed-direct-integral}
 \mathfrak q_s[w]=\int_{\mathbb R}
 \mathfrak b_{s,\xi}[\widehat w(\xi,\cdot)]\,d\xi.
\end{equation}
\end{proposition}

\begin{proof}
Let the right-hand side define $\widetilde{\mathfrak q}_s$.
On $|\xi|\leq N$, \eqref{eq:H-kernel-L1},
\eqref{eq:fiber-difference-bound}, and \eqref{eq:unweighted-control} give
\begin{equation}\label{eq:uniform-graph-equivalence}
 C_N^{-1}\bigl(\mathfrak b_{s,0}[v]+\|v\|_{H_{s,\tau}}^2\bigr)
 \leq\mathfrak b_{s,\xi}[v]+\|v\|_{H_{s,\tau}}^2
 \leq C_N\bigl(\mathfrak b_{s,0}[v]+\|v\|_{H_{s,\tau}}^2\bigr).
\end{equation}

The common form domain is separable. By
\eqref{eq:fiber-difference-bound} and polarization,
$\xi\mapsto\mathfrak b_{s,\xi}(u,v)$ is continuous for
$u,v\in\mathscr Y_s$. A countable dense subset of the angular core
therefore defines a measurable field of form spaces. On each bounded
frequency interval, \eqref{eq:uniform-graph-equivalence} identifies this
field with the constant space $\mathscr Y_s$. Finite-energy sections are
strongly measurable, and common-core simple sections are dense.

Let $\{f_n\}$ be Cauchy for the integrated graph norm. Choose a subsequence,
still denoted by $\{f_n\}$, such that
\[
 \|f_{n+1}-f_n\|_{\rm graph}\leq2^{-n}.
\]
On $[-N,N]$, Cauchy--Schwarz gives
\[
 \int_{-N}^N
 \bigl(\mathfrak b_{s,\xi}[f_{n+1}(\xi)-f_n(\xi)]
 +\|f_{n+1}(\xi)-f_n(\xi)\|_{H_{s,\tau}}^2\bigr)^{1/2}\,d\xi
 \leq (2N)^{1/2}2^{-n}.
\]
Fubini's theorem shows that the sum of the pointwise fiber graph
increments is finite for almost every $\xi\in[-N,N]$. Diagonalize in
$N$. Fiber completeness gives a pointwise limit $f(\xi)\in\mathscr Y_s$
for almost every $\xi$. Fatou's lemma yields
\[
 \|f_n-f\|_{\rm graph}^2
 \leq\liminf_{m\to\infty}\|f_n-f_m\|_{\rm graph}^2.
\]
The right-hand side tends to zero. Thus the integrated form is closed.

To identify a core, truncate a finite-energy section to $|\xi|\leq N$.
By \eqref{eq:uniform-graph-equivalence}, Bochner simple functions with
values in $\mathscr Y_s$ are dense.
Since $C_c^\infty(0,\pi)$ is the common form core, their values may be
replaced by angular core functions. By
\eqref{eq:uniform-graph-equivalence}, their scalar coefficients may next
be approximated in $L^2(-N,N)$ by functions in
$C_c^\infty(\mathbb R_\xi)$.

It remains to obtain compact support in $t$. Consider one term
$g(\xi)\psi(\theta)$, where $g\in C_c^\infty(\mathbb R)$ and
$\psi\in C_c^\infty(0,\pi)$. Its inverse Fourier transform
$\check g$ is a Schwartz function. Choose
$\chi\in C_c^\infty(\mathbb R)$ with $\chi=1$ near zero. Then
$\chi(t/R)\check g(t)\to\check g(t)$ in $H^s(\mathbb R)$. Indeed, use
dominated convergence for the $L^2$ term. For the difference quotient,
split $|t-t'|<1$ and $|t-t'|\geq1$. The Lipschitz bound for
$\chi(t/R)$ controls the first part. The Schwartz decay of $\check g$
controls both parts. For the fixed angular core function,
\eqref{eq:H-kernel-L1} and \eqref{eq:fiber-difference-bound} give
\[
 \mathfrak b_{s,\xi}[\psi]+\|\psi\|_{H_{s,\tau}}^2
 \leq C_\psi(1+|\xi|^{2s}).
\]
Plancherel's theorem therefore gives convergence in the integrated graph
norm. Each product
$\chi(t/R)\check g(t)\psi(\theta)$ belongs to $\mathscr C$. Thus
$\mathscr C$ is a form core for $\widetilde{\mathfrak q}_s$.
Proposition~\ref{prop:direct-integral} shows that the two forms agree on
this core. Their closures are equal.
\end{proof}

We now establish the modulus inequality and the positivity of the first eigenfunction.

\begin{lemma}[Modulus stability of the closed angular form]
\label{lem:modulus-domain}
For every $\psi\in\mathscr Y_s$, one has $|\psi|\in\mathscr Y_s$ and
\[
 \mathfrak b_{s,0}[|\psi|]\leq\mathfrak b_{s,0}[\psi].
\]
Moreover, there are $\eta_j\in C_c^\infty(0,\pi)$ such that
$\eta_j\geq0$ and $\eta_j\to|\psi|$ in $\mathscr Y_s$.
\end{lemma}

\begin{proof}
Let $I\Subset I_1\Subset(0,\pi)$. If $f\in C_c^\infty(I)$ and $f$ is
extended by zero to $\mathbb R$, then
\[
 \mathfrak b_{s,0}[f]\leq C_I\left\{\|f\|_{L^2(\mathbb R)}^2+
 \int_{\mathbb R}\int_{\mathbb R}
 \frac{|f(\theta)-f(\varphi)|^2}
 {|\theta-\varphi|^{1+2s}}\,d\theta\,d\varphi\right\}.
\]
To prove this, use \eqref{eq:positive-fiber} and put $b=s-1/2$. On
$|q|+|\theta-\varphi|<\delta_I$, a nonzero term forces both angular
variables into $I_1$. The functions $(\sin\theta)^{-b}$ and
$e^{-q/4}$ are smooth there. The mean-value theorem gives
\[
 \begin{split}
 &|e^{-q/4}(\sin\theta)^{-b}f(\theta)
 -e^{q/4}(\sin\varphi)^{-b}f(\varphi)|^2\\
 &\quad\leq C_I|f(\theta)-f(\varphi)|^2
 +C_I(q^2+|\theta-\varphi|^2)
 (|f(\theta)|^2+|f(\varphi)|^2).
 \end{split}
\]
The sine factor outside the square is bounded on $I_1$. By
Lemma~\ref{lem:kernel-local}, integration in $q$ of the first term gives
$C_I|\theta-\varphi|^{-1-2s}$. The second term is integrable because
\[
 \int_{q^2+\alpha^2<1}(q^2+\alpha^2)^{-s}\,dq\,d\alpha<\infty.
\]
On the complementary set, expand the square. At least one function value
has its angular variable in $I$. If $|q|\geq\delta_I/2$, the exponential
tail in \eqref{eq:kernel-tail} absorbs $e^{|q|/2}$. If
$|\theta-\varphi|\geq\delta_I/2$, the kernel is bounded after integration
in $q$. The remaining endpoint factor is a fixed power of sine. Its worst
power is $(\sin\vartheta)^{s-1/2}$, which is integrable because
$s-1/2>-1$. This bounds the complementary part by
$C_I\|f\|_2^2$. The local upper estimate follows.

Now let $\psi\in C_c^\infty(0,\pi)$. The function $|\psi|$, extended by
zero, belongs to $H^1(\mathbb R)$. Indeed, the modulus map is Lipschitz and
$|(|\psi|)'|\leq|\psi'|$ almost everywhere. Let $\rho_\varepsilon$ be a
nonnegative Friedrichs mollifier and put
$\eta_\varepsilon=\rho_\varepsilon*|\psi|$. For small $\varepsilon$,
$\eta_\varepsilon$ is nonnegative and belongs to
$C_c^\infty(0,\pi)$. It converges to $|\psi|$ in $H^s(\mathbb R)$. Apply
the local upper estimate to
$\eta_\varepsilon-\eta_\delta$. The family
$\{\eta_\varepsilon\}$ is Cauchy in $\mathscr Y_s$. Its limit in
$H_{s,\tau}$ is $|\psi|$. Hence $|\psi|\in\mathscr Y_s$.

For $A,B>0$ and $z,w\in\mathbb C$,
$|A|z|-B|w||\leq|Az-Bw|$. Apply this to
\eqref{eq:positive-fiber}. A subsequence of $\eta_\varepsilon$ and its
image under the closed difference operator converges almost everywhere,
so the closed form agrees with that integral at $|\psi|$. Hence
\[
 \mathfrak b_{s,0}[|\psi|]\leq\mathfrak b_{s,0}[\psi]
\]
for every core function $\psi$. The preceding mollifiers also give its
nonnegative core approximation.

Finally, let $\psi\in\mathscr Y_s$. Choose core functions
$\psi_n\to\psi$ in $\mathscr Y_s$. Then
$|\psi_n|\to|\psi|$ in $H_{s,\tau}$. The core inequality shows that
$\{|\psi_n|\}$ is bounded in $\mathscr Y_s$. A weakly convergent
subsequence has weak limit $|\psi|$, because its strong
$H_{s,\tau}$ limit is $|\psi|$. Weak lower semicontinuity gives
\[
 \mathfrak b_{s,0}[|\psi|]
 \leq\liminf_{n\to\infty}\mathfrak b_{s,0}[|\psi_n|]
 \leq\lim_{n\to\infty}\mathfrak b_{s,0}[\psi_n]
 =\mathfrak b_{s,0}[\psi].
\]
Mazur's lemma gives convex combinations of the nonnegative functions
$|\psi_n|$ which converge strongly to $|\psi|$ in $\mathscr Y_s$.
Approximate the finitely many terms in each combination by the nonnegative
smooth approximants constructed above. The resulting functions are
nonnegative, belong to $C_c^\infty(0,\pi)$, and converge to $|\psi|$ in
$\mathscr Y_s$.
\end{proof}

The next lemma gives simplicity and positivity at zero frequency.

\begin{lemma}[Positive simple ground state]\label{lem:positive-simple}
The eigenvalue $\lambda_1(0)$ is simple. Its eigenfunction can be chosen
strictly positive almost everywhere in $(0,\pi)$.
\end{lemma}

\begin{proof}
Lemma~\ref{lem:modulus-domain} permits replacement of every form-domain
function by its modulus. Fix $\alpha>0$. The shifted resolvent of $B_s(0)$
is positivity preserving.
Indeed, for $F\geq0$, the vector
$(B_s(0)+\alpha)^{-1}F$ uniquely minimizes
\[
 \mathfrak b_{s,0}[v]+\alpha\|v\|_{H_{s,\tau}}^2
 -2\operatorname{Re}\langle F,v\rangle_{H_{s,\tau}}.
\]
Replacing $v$ by $|v|$ cannot increase this functional. Uniqueness gives
a nonnegative minimizer.

For irreducibility, the jump kernel is
\[
 J_s(\theta,\varphi)=2\int_{\mathbb R}
 K_s(q,\theta,\varphi)\,dq>0
\]
for $\theta\neq\varphi$. Suppose that a measurable set $A$ and its
complement both have positive measure and that the ideal $L^2(A)$ is
invariant. We first construct $h\in\mathscr Y_s$ with $h>0$ almost
everywhere. Choose nonnegative core functions $h_n$ which are positive on
$[1/n,\pi-1/n]$. Choose $c_n>0$ so rapidly decreasing that
$\sum_nc_n\|h_n\|_{\mathscr Y_s}<\infty$. Then
$h=\sum_nc_nh_n$ has the required properties. The invariant-ideal
criterion for closed forms gives
$f=\boldsymbol 1_Ah\in\mathscr Y_s$,
$g=\boldsymbol 1_{A^c}h\in\mathscr Y_s$, and
$\mathfrak b_{s,0}(f,g)=0$. Both functions are nonzero. We compute the mixed term using the positive
representation \eqref{eq:positive-fiber}, rather than separating the
possibly infinite Hardy terms. Let $\mathfrak g_{s,0}^\varepsilon$ be its
kernel truncation. Since $fg=0$ almost everywhere, expansion gives
\[
 \mathfrak g_{s,0}^\varepsilon[f+g]
 -\mathfrak g_{s,0}^\varepsilon[f]
 -\mathfrak g_{s,0}^\varepsilon[g]
 =-4\int_A\int_{A^c}\int_{\mathbb R}
 K_s^\varepsilon(q,\theta,\varphi)f(\theta)g(\varphi)
 \,dq\,d\theta\,d\varphi.
\]
Each positive energy on the left converges to its finite closed-form
value. Monotone convergence on the nonnegative integral on the right
therefore proves that this integral is finite. Polarization gives
\[
 \mathfrak b_{s,0}(f,g)
 =-\int_A\int_{A^c}J_s(\theta,\varphi)
 f(\theta)g(\varphi)\,d\theta\,d\varphi<0.
\]
This contradicts the assumed decomposition. Thus the form is irreducible.

For completeness, the invariant-ideal criterion used above states that, for a
closed symmetric positivity-preserving form, an orthogonal band projection
$P_A$ commutes with one shifted resolvent if and only if
$P_A\mathscr Y_s\subset\mathscr Y_s$ and
$\mathfrak b_{s,0}(P_Av,(I-P_A)v)=0$ for every $v\in\mathscr Y_s$.
The forward implication follows by inserting resolvent vectors in the weak
resolvent equation; these vectors are form-dense. The reverse implication
follows by splitting that equation into the two bands. Thus the preceding
calculation excludes every nontrivial invariant band of each shifted
resolvent.

Write $R_\alpha=(B_s(0)+\alpha)^{-1}$. Fix $\beta>0$ and choose
$\alpha>\beta$. All these resolvents preserve positivity, and
$\|R_\alpha\|\leq\alpha^{-1}$. The resolvent identity therefore gives
the norm-convergent Neumann series
\[
 R_\beta=\sum_{n=1}^\infty
 (\alpha-\beta)^{n-1}R_\alpha^n.
\]
For $F\geq0$, $F\neq0$, let $A$ be the union, up to null sets, of the
sets where $R_\alpha^nF>0$, $n\geq1$. Choose positive coefficients
$c_n$ so that $H=\sum_{n\geq1}c_nR_\alpha^nF$ converges in
$H_{s,\tau}$. Then $H>0$ almost everywhere on $A$, and
$R_\alpha H$ vanishes outside $A$. For every $u\geq0$ supported in
$A$, the functions $\min\{u,jH\}$ converge to $u$ in $H_{s,\tau}$.
Their images under $R_\alpha$ vanish outside $A$, since they are
bounded above by $jR_\alpha H$. Hence $L^2(A,m_{s,\tau}d\theta)$ is
invariant under $R_\alpha$. Self-adjointness makes it a reducing band.
Irreducibility and $R_\alpha F\neq0$ imply that $A$ has full measure.
Consequently, for every $G\geq0$, $G\neq0$, there is $n\geq1$ with
$\langle G,R_\alpha^nF\rangle_{H_{s,\tau}}>0$. The Neumann series gives
$\langle G,R_\beta F\rangle_{H_{s,\tau}}>0$. Thus $R_\beta F>0$
almost everywhere, and every shifted resolvent is positivity improving.

The compact embedding
\eqref{eq:compact-embedding} makes $R_\alpha$ compact. The
Perron--Frobenius theorem for a compact positivity-improving self-adjoint
operator now gives a strictly positive, simple eigenvector at its spectral
radius. This vector is the first eigenfunction of $B_s(0)$.

\end{proof}

\begin{lemma}[Diamagnetic comparison]\label{lem:diamagnetic}
For $\psi\in\mathscr Y_s$,
\begin{equation}\label{eq:diamagnetic}
 \mathfrak b_{s,\xi}[\psi]\geq
 \mathfrak b_{s,0}[|\psi|].
\end{equation}
Consequently,
\begin{equation}\label{eq:fiber-minimum}
 \lambda_1(\xi)\geq\lambda_1(0).
\end{equation}
The inequality in \eqref{eq:fiber-minimum} is strict when $\xi\neq0$.
\end{lemma}

\begin{proof}
For a core function, truncate the kernel. The mixed term satisfies
\[
 \operatorname{Re}(e^{i\xi q}\psi(\theta)
 \overline{\psi(\varphi)})
 \leq |\psi(\theta)||\psi(\varphi)|.
\]
The diagonal terms are unchanged by the modulus. Integration, monotone
convergence in the positive representations, and core approximation give
\eqref{eq:diamagnetic}. Taking the infimum proves
\eqref{eq:fiber-minimum}.

Assume equality for some $\xi\neq0$. Compactness gives a minimizer $\psi$.
Then $|\psi|$ minimizes the zero fiber. By
Lemma~\ref{lem:positive-simple}, $|\psi|>0$ almost everywhere. Equality in
the truncated phase inequality, followed by monotone convergence, implies
\[
 \operatorname{Re}\bigl(e^{i\xi q}\psi(\theta)
 \overline{\psi(\varphi)}\bigr)
 =|\psi(\theta)||\psi(\varphi)|
\]
for almost every $(q,\theta,\varphi)$. The defect is nonnegative. Since $|\psi|>0$ almost everywhere, Fubini's
theorem gives a pair $(\theta,\varphi)$ for which the equality holds for
almost every $q$. It forces $e^{i\xi q}$ to have a fixed phase, which
is impossible when $\xi\neq0$.
\end{proof}

\begin{lemma}[Interior oscillatory estimate]
Let $0<\tau_0<\min\{\tau,2s\}$. For $0<\delta<1/8$, set
$I_\delta=(\delta,\pi-\delta)$. Then
\begin{equation}\label{eq:interior-frequency-new}
 |\xi|^{2s}\int_{I_{2\delta}}|\psi|^2\,d\theta
 \leq C\delta^{-2}\left(
 \mathfrak b_{s,\xi}[\psi]
 +\int_0^\pi m_{s,\tau_0}|\psi|^2\,d\theta\right)
\end{equation}
for every $\psi\in\mathscr Y_s$. The constant is independent of
$\delta$, $\xi$, and $\psi$.
\end{lemma}

\begin{proof}
For $\psi\in C_c^\infty(0,\pi)$, put
$b=s-\frac12$. Choose $\eta_\delta\in C_c^\infty(I_{3\delta/2})$ such that
$\eta_\delta=1$ on $I_{2\delta}$ and
$|\eta_\delta'|\leq C\delta^{-1}$. Extend
$f=\eta_\delta\psi$ by zero to $\mathbb R$.

Absorbing the angular weight in \eqref{eq:positive-fiber} into the
square gives
\[
 \mathcal S_{s,\xi}(q,\theta,\varphi;\psi):=
 \left|e^{-q/4}\left(\frac{\sin\varphi}{\sin\theta}\right)^{b/2}
 \psi(\theta)-e^{q/4-i\xi q}
 \left(\frac{\sin\theta}{\sin\varphi}\right)^{b/2}
 \psi(\varphi)\right|^2.
\]
If $\theta,\varphi\in I_{5\delta/4}$ and
$q^2+|\theta-\varphi|^2<(\delta/4)^2$, the sine ratios are bounded above
and below by absolute constants. The mean-value theorem gives
\[
 \begin{split}
 &\left|e^{-q/4}\left(\frac{\sin\varphi}{\sin\theta}\right)^{b/2}-1\right|
 +\left|e^{q/4}\left(\frac{\sin\theta}{\sin\varphi}\right)^{b/2}-1\right| \leq C_s\delta^{-1}
 (|q|+|\theta-\varphi|).
 \end{split}
\]
It follows from $|A+B|^2\leq2|A|^2+2|B|^2$ that
\begin{equation}\label{eq:flat-weighted-comparison}
 \begin{split}
 |\psi(\theta)-e^{-i\xi q}\psi(\varphi)|^2
 \leq{}&C_s\mathcal S_{s,\xi}(q,\theta,\varphi;\psi)\\
 &+C_s\delta^{-2}(q^2+|\theta-\varphi|^2)
 (|\psi(\theta)|^2+|\psi(\varphi)|^2).
 \end{split}
\end{equation}
Use also the cutoff identity
\[
 \begin{split}
 \eta_\delta(\theta)\psi(\theta)
 -e^{-i\xi q}\eta_\delta(\varphi)\psi(\varphi)
 ={}&\eta_\delta(\theta)
 [\psi(\theta)-e^{-i\xi q}\psi(\varphi)]+e^{-i\xi q}[\eta_\delta(\theta)-\eta_\delta(\varphi)]
 \psi(\varphi)
 \end{split}
\]
and $|\eta_\delta(\theta)-\eta_\delta(\varphi)|\leq
C\delta^{-1}|\theta-\varphi|$. On the ball
$q^2+|\theta-\varphi|^2<(\delta/4)^2$, any pair for which the cutoff
term is nonzero lies in $I_{5\delta/4}^2$. Thus
\eqref{eq:kernel-comparison}, \eqref{eq:positive-fiber}, and
\eqref{eq:flat-weighted-comparison} apply. The error integral is bounded by
\[
 C\delta^{-2}\left(\int_0^{\delta/4}r^{1-2s}\,dr\right)
 \|\psi\|_2^2
 \leq C\delta^{-2s}\|\psi\|_2^2.
\]
On the complement of this ball, use
$|f(\theta)-e^{-i\xi q}f(\varphi)|^2
\leq2(|f(\theta)|^2+|f(\varphi)|^2)$. Direct integration gives another
$C\delta^{-2s}\|\psi\|_2^2$ bound. Since
$\delta^{-2s}\leq\delta^{-2}$, we have proved
\begin{equation}\label{eq:localized-symbol-bound}
 \begin{split}
 &\int_{\mathbb R^2}\int_{\mathbb R}
 \frac{|f(\theta)-e^{-i\xi q}f(\varphi)|^2}
 {(q^2+|\theta-\varphi|^2)^{1+s}}
 \,dq\,d\theta\,d\varphi\leq C\delta^{-2}\left(
 \mathfrak b_{s,\xi}[\psi]
 +\|\psi\|_{L^2(0,\pi)}^2\right).
 \end{split}
\end{equation}

Let $\widehat f$ be the unitary Fourier transform on $\mathbb R$. Put
$\alpha=\theta-\varphi$ and apply Plancherel in $\theta$. We use the identity
\[
 \int_{\mathbb R^2}
 \frac{1-\cos(\xi q+\zeta\alpha)}
 {(q^2+\alpha^2)^{1+s}}\,dq\,d\alpha
 =c_s(\xi^2+\zeta^2)^s,
\]
where $c_s>0$. Plancherel's theorem gives, after absorbing the harmless
factor two into $c_s$,
\[
 \begin{split}
 &\int_{\mathbb R^2}\int_{\mathbb R}
 \frac{|f(\theta)-e^{-i\xi q}f(\varphi)|^2}
 {(q^2+|\theta-\varphi|^2)^{1+s}}
 \,dq\,d\theta\,d\varphi =c_s\int_{\mathbb R}(\xi^2+\zeta^2)^s
 |\widehat f(\zeta)|^2\,d\zeta
 \geq c_s|\xi|^{2s}\|f\|_2^2.
 \end{split}
\]
Since $\eta_\delta=1$ on $I_{2\delta}$,
$\|f\|_2^2\geq\int_{I_{2\delta}}|\psi|^2$. Finally,
\eqref{eq:unweighted-control} controls the last $L^2$ term in
\eqref{eq:localized-symbol-bound}. This proves
\eqref{eq:interior-frequency-new} for core functions. Approximate an
arbitrary $\psi\in\mathscr Y_s$ by core functions $\psi_n$. Apply
\eqref{eq:localized-symbol-bound} to $\psi_n-\psi_m$. It shows that
$\eta_\delta\psi_n$ is Cauchy in the flat Fourier form. Its $L^2$ limit is
$\eta_\delta\psi$. Passing to the limit in the Fourier identity and using
lower semicontinuity proves the general case.
\end{proof}

\begin{lemma}[Large-frequency coercivity]\label{lem:large-frequency}
There are $c,C>0$ and $\beta>0$ such that
\begin{equation}\label{eq:large-frequency}
 \mathfrak b_{s,\xi}[\psi]
 \geq (c|\xi|^\beta-C)
 \int_0^\pi m_{s,\tau}|\psi|^2\,d\theta
\end{equation}
for every $\psi\in\mathscr Y_s$. Hence
\begin{equation}\label{eq:lambda-infinity}
 \lambda_1(\xi)\longrightarrow\infty
 \quad\text{as }|\xi|\to\infty.
\end{equation}
\end{lemma}

\begin{proof}
Choose $0<\tau_0<\min\{\tau,2s\}$. By
\eqref{eq:all-weight-coercivity} and \eqref{eq:diamagnetic},
\[
 \int_0^\pi m_{s,\tau_0}|\psi|^2\,d\theta
 \leq C\mathfrak b_{s,0}[|\psi|]
 \leq C\mathfrak b_{s,\xi}[\psi].
\]
On the boundary part,
\begin{equation}\label{eq:boundary-frequency}
 \int_{(0,2\delta)\cup(\pi-2\delta,\pi)}
 m_{s,\tau}|\psi|^2\,d\theta
 \leq C\delta^{\tau-\tau_0}\mathfrak b_{s,\xi}[\psi].
\end{equation}
On $I_{2\delta}$, $m_{s,\tau}\leq
C\delta^{-|\tau-2s|}$. Combine
\eqref{eq:interior-frequency-new} and \eqref{eq:boundary-frequency}. We obtain
\[
 \int_0^\pi m_{s,\tau}|\psi|^2\,d\theta
 \leq C\left(
 |\xi|^{-2s}\delta^{-2-|\tau-2s|}
 +\delta^{\tau-\tau_0}\right)
 \mathfrak b_{s,\xi}[\psi].
\]
Put $\eta=\tau-\tau_0$ and $M=|\tau-2s|$. Choose
\[
 \varepsilon=\frac{2s}{2+M+\eta},\qquad
 \delta=|\xi|^{-\varepsilon}.
\]
For large $|\xi|$, we have $\delta<1/8$. Both
coefficients in the preceding estimate are bounded by
$C|\xi|^{-\beta}$, where
\[
 \beta=\frac{2s(\tau-\tau_0)}
 {2+|\tau-2s|+\tau-\tau_0}>0.
\]
Thus
\[
 \mathfrak b_{s,\xi}[\psi]\geq
 c|\xi|^\beta\|\psi\|_{H_{s,\tau}}^2
\]
for large $|\xi|$. Adjusting the lower bound on a bounded frequency
interval gives \eqref{eq:large-frequency}. Taking the infimum gives \eqref{eq:lambda-infinity}.
\end{proof}
\subsection*{Kernel moments and frequency growth}

\begin{lemma}[Moment bounds]\label{lem:kernel-moments}
For every integer $j\geq1$, the form
\begin{equation}\label{eq:moment-form}
 \int_0^\pi\int_0^\pi\int_{\mathbb R}
 |q|^{2j}K_s(q,\theta,\varphi)
 |\psi(\theta)||\psi(\varphi)|
 \,dq\,d\theta\,d\varphi
\end{equation}
is finite on $\mathscr Y_s$ and is bounded by
$C_j(\mathfrak b_{s,0}[\psi]+\|\psi\|_{L^2(m_{s,\tau})}^2)$.
For $\psi,\chi\in\mathscr Y_s$ and
$|\operatorname{Im}z|<1+s$, define
\begin{equation*}
 \begin{split}
 \mathfrak d_{s,z}(\psi,\chi)
 =\int_0^\pi\int_0^\pi\int_{\mathbb R}
 &(1-\cos(zq))K_s(q,\theta,\varphi)\bigl(\psi(\theta)\overline{\chi(\varphi)}
 +\psi(\varphi)\overline{\chi(\theta)}\bigr)
 \,dq\,d\theta\,d\varphi
 \end{split}
\end{equation*}
and $\mathfrak b_{s,z}=\mathfrak b_{s,0}+\mathfrak d_{s,z}$.
Then $z\mapsto\mathfrak d_{s,z}$ is a holomorphic family of bounded forms
on $\mathscr Y_s$ in this strip. It is even. For $|z|$ sufficiently
small, $\mathfrak b_{s,z}$ is a closed sectorial form with common domain
$\mathscr Y_s$. The resulting family is of type (a). Moreover,
\begin{equation}\label{eq:analytic-form-expansion}
 \mathfrak b_{s,z}
 =\mathfrak b_{s,0}+z^2\mathfrak c_{s,2}+O_{\mathscr Y_s}(z^4)
 \quad\text{as }z\to0.
\end{equation}
Here
\begin{equation}\label{eq:second-form-coefficient}
 \begin{split}
 \mathfrak c_{s,2}(\psi,\chi)
 =\frac12\int_0^\pi\int_0^\pi\int_{\mathbb R}
 q^2K_s(q,\theta,\varphi)
 \bigl(\psi(\theta)\overline{\chi(\varphi)}
 +\psi(\varphi)\overline{\chi(\theta)}\bigr)
 \,dq\,d\theta\,d\varphi.
 \end{split}
\end{equation}
\end{lemma}

\begin{proof}
For $|q|\geq1$, use \eqref{eq:kernel-tail} and
$\|\psi\|_{L^1(0,\pi)}^2\leq\pi\|\psi\|_{L^2(0,\pi)}^2$. Put
$\alpha=\theta-\varphi$. For $|q|<1$ and $|\alpha|\leq1$, use
\eqref{eq:kernel-comparison}. On $1\leq|\alpha|\leq\pi$, the kernel
is uniformly bounded because $\cosh q-\cos\alpha\geq1-\cos1>0$.
Integration in $q$ of the local bound gives
\begin{equation}\label{eq:q-moment-integral}
 \int_{-1}^1\frac{|q|^{2j}}
 {(q^2+\alpha^2)^{1+s}}\,dq
 \leq C_j
 \begin{cases}
  1+|\alpha|^{2j-1-2s},&2j-1-2s<0,\\
  1+|\log|\alpha||,&2j-1-2s=0,\\
  1,&2j-1-2s>0.
 \end{cases}
\end{equation}
Indeed, split at $|q|=|\alpha|$ and integrate
$q^{2j-2-2s}$ on the outer part. All three kernels are integrable in
$\alpha$, because $2j-1-2s>-1$. Young's inequality and
\eqref{eq:unweighted-control} give
\[
 \text{the left-hand side of \eqref{eq:moment-form}}
 \leq C_j\|\psi\|_{L^2(0,\pi)}^2
 \leq C_j\mathfrak b_{s,0}[\psi].
\]

For analyticity, let $\mathcal K$ be a compact subset of
$\{z:|\operatorname{Im}z|<1+s\}$. Put
\[
 r=\max_{z\in\mathcal K}|\operatorname{Im}z|<1+s.
\]
For $z\in\mathcal K$ and $|q|\leq1$,
\[
 |1-\cos(zq)|\leq C_{\mathcal K}q^2.
\]
For $|q|\geq1$,
\[
 |1-\cos(zq)|\leq C_{\mathcal K}e^{r|q|}.
\]
The first bound is controlled by \eqref{eq:q-moment-integral} with $j=1$.
The second is integrable because
$K_s(q,\theta,\varphi)\leq Ce^{-(1+s)|q|}$. For every integer
$\ell\geq1$, the tail of the $\ell$th $z$-derivative is bounded by
$C_{\mathcal K,\ell}|q|^\ell e^{r|q|}$, uniformly on $\mathcal K$.
After multiplication by $K_s$, the polynomial factor is absorbed by
the positive exponential margin $1+s-r$, giving an integrable tail.
Near $q=0$, even derivatives of positive order contain
$q^\ell\cos(zq)$, whereas odd derivatives contain
$q^\ell\sin(zq)=O(|q|^{\ell+1})$. Each has at least quadratic
cancellation. Differentiation under the integral is therefore valid in
bounded-form norm on $\mathscr Y_s$. For $z$ near zero,
the moment estimate and
$|1-\cos(zq)|\leq C|z|^2q^2e^{|\operatorname{Im}z||q|}$ give
\[
 |\mathfrak d_{s,z}[\psi]|
 \leq C|z|^2\bigl(\mathfrak b_{s,0}[\psi]
 +\|\psi\|_{H_{s,\tau}}^2\bigr)
 \leq C'|z|^2\mathfrak b_{s,0}[\psi].
\]
By \eqref{eq:fiber-positive}, the relative bound is less than one
for small $|z|$. Hence $\mathfrak b_{s,z}$ is closed and sectorial on
$\mathscr Y_s$, of type (a); the associated operators are of type (B).

For $z$ near zero, Taylor's formula gives
\[
 1-\cos(zq)=\frac{z^2q^2}{2}+O\bigl(
 |z|^4q^4e^{|\operatorname{Im}z||q|}\bigr).
\]
Use \eqref{eq:q-moment-integral} with $j=2$ for $|q|\leq1$ and the
exponential tail for $|q|\geq1$. This proves
\eqref{eq:analytic-form-expansion} with
\eqref{eq:second-form-coefficient}. Evenness follows directly from the
cosine.
\end{proof}

\begin{lemma}[Two-sided fiber growth]\label{lem:kernel-growth}
There is $C\geq1$ such that
\begin{equation}\label{eq:fiber-growth}
 \mathfrak b_{s,\xi}[\psi]\leq
 C(1+|\xi|^{2s})\bigl(
 \mathfrak b_{s,0}[\psi]+\|\psi\|_{L^2(m_{s,\tau})}^2\bigr).
\end{equation}
For $\psi=\phi_{s,\tau}$,
\begin{equation}\label{eq:ground-fiber-growth}
 c(1+|\xi|^{2s})-C
 \leq \mathfrak b_{s,\xi}[\phi_{s,\tau}]
 \leq C(1+|\xi|^{2s}).
\end{equation}
\end{lemma}

\begin{proof}
By \eqref{eq:H-kernel-L1}, \eqref{eq:fiber-difference-bound}, and
\eqref{eq:unweighted-control},
\[
 |\mathfrak b_{s,\xi}[\psi]-\mathfrak b_{s,0}[\psi]|
 \leq C\min\{\xi^2,|\xi|^{2s}\}\|\psi\|_{L^2(0,\pi)}^2
 \leq C(1+|\xi|^{2s})\mathfrak b_{s,0}[\psi].
\]
This proves \eqref{eq:fiber-growth} and the upper bound in
\eqref{eq:ground-fiber-growth}.

For the lower bound, choose $0<\tau_0<\min\{\tau,2s\}$ and fix
$0<\delta<1/8$ so small that
\[
 \int_{I_{2\delta}}\phi_{s,\tau}^2\,d\theta>0.
\]
This is possible because $\phi_{s,\tau}>0$ almost everywhere. Apply
\eqref{eq:interior-frequency-new} to $\phi_{s,\tau}$. The term with
$m_{s,\tau_0}$ is finite by \eqref{eq:all-weight-coercivity}. Rearranging
gives
\[
 \mathfrak b_{s,\xi}[\phi_{s,\tau}]
 \geq C^{-1}\delta^2|\xi|^{2s}
 \int_{I_{2\delta}}\phi_{s,\tau}^2\,d\theta
 -\int_0^\pi m_{s,\tau_0}\phi_{s,\tau}^2\,d\theta.
\]
Hence $\mathfrak b_{s,\xi}[\phi_{s,\tau}]
\geq c|\xi|^{2s}-C$. Positivity handles bounded frequencies.
\end{proof}
\section{The sharp constant and its ground-state representation}\label{sec:sharp}

\begin{lemma}[Logarithmic cutoff]\label{lem:log-cutoff}
Let $\chi\in C_c^\infty(\mathbb R)$ satisfy $\|\chi\|_2=1$. Let
$\psi\in\mathscr Y_s$ and
\[
 w_R(t,\theta)=R^{-1/2}\chi(t/R)\psi(\theta).
\]
Then
\begin{equation*}
 \|w_R\|_{L^2(m_{s,\tau})}^2=\|\psi\|_{L^2(m_{s,\tau})}^2,
 \qquad
 \mathfrak q_s[w_R]\longrightarrow\mathfrak b_{s,0}[\psi].
\end{equation*}
If $\psi\in C_c^\infty(0,\pi)$, then $w_R$ is the transform of a function
in $C_c^\infty(\mathbb R^2_+)$.
\end{lemma}

\begin{proof}
The change of variables $t=Rr$ gives the norm identity. Also,
\[
 \widehat w_R(\xi,\theta)=R^{1/2}\widehat\chi(R\xi)
 \psi(\theta).
\]
By \eqref{eq:closed-direct-integral},
\[
 \mathfrak q_s[w_R]=
 \int_{\mathbb R}R|\widehat\chi(R\xi)|^2
 \mathfrak b_{s,\xi}[\psi]\,d\xi.
\]
Putting $\zeta=R\xi$, we have
\[
 \mathfrak q_s[w_R]=\int_{\mathbb R}
 |\widehat\chi(\zeta)|^2\mathfrak b_{s,\zeta/R}[\psi]\,d\zeta.
\]
For $R\geq1$, \eqref{eq:fiber-difference-bound} and
\eqref{eq:H-kernel-L1} give
\[
 \mathfrak b_{s,\zeta/R}[\psi]
 \leq C(1+|\zeta|^{2s})
 \bigl(\mathfrak b_{s,0}[\psi]+\|\psi\|_{L^2(m_{s,\tau})}^2\bigr).
\]
Since $\widehat\chi$ is rapidly decreasing and the fiber form is
continuous at zero, dominated convergence proves the energy limit.
If $\psi\in C_c^\infty(0,\pi)$, the inverse transform has compact
support in the half-plane by \eqref{eq:polar}.
\end{proof}

\begin{proof}[Proof of Theorem~\ref{thm:A}]
Let $u\in C_c^\infty(\mathbb R^2_+)$ and let $w$ be its transform.
Proposition~\ref{prop:direct-integral} and
Lemma~\ref{lem:diamagnetic} yield
\[
 \mathcal H_s[u]=\int\mathfrak b_{s,\xi}[\widehat w]\,d\xi
 \geq\lambda_1(0)
 \int\|\widehat w(\xi,\cdot)\|_{L^2(m_{s,\tau})}^2\,d\xi
 =\lambda_1(0)\mathcal R_{s,\tau}[u].
\]
Hence $\Lambda_{s,\tau}\geq\lambda_1(0)$. By the definition of
$\mathscr Y_s$, choose normalized $\psi_j\in C_c^\infty(0,\pi)$ with
\[
 \psi_j\longrightarrow\phi_{s,\tau}\quad\hbox{in }\mathscr Y_s.
\]
For each $j$, the cutoffs in Lemma~\ref{lem:log-cutoff} correspond to
half-plane test functions. Letting first $R\to\infty$ and then
$j\to\infty$, we obtain
\[
 \Lambda_{s,\tau}\leq\lim_{j\to\infty}
 \mathfrak b_{s,0}[\psi_j]=\lambda_1(0).
\]
The positive lower bound follows from \eqref{eq:fiber-positive}.
Lemmas~\ref{lem:closed-fibers} and \ref{lem:positive-simple} give
the unique normalized positive angular minimizer.

Suppose that a nonzero $w\in\mathscr X_s$ attains the quotient. Then
\[
 0=\int_{\mathbb R}
 \left(\mathfrak b_{s,\xi}[\widehat w(\xi,\cdot)]
 -\lambda_1(0)\|\widehat w(\xi,\cdot)\|_{L^2(m_{s,\tau})}^2\right)d\xi.
\]
For almost every $\xi$, the integrand is at least
\[
 (\lambda_1(\xi)-\lambda_1(0))
 \|\widehat w(\xi,\cdot)\|_{L^2(m_{s,\tau})}^2\geq0.
\]
By the strict inequality in Lemma~\ref{lem:diamagnetic},
$\widehat w(\xi,\cdot)=0$ for almost every $\xi\neq0$.
Since $\widehat w$ is an $L^2$ function, it follows that $w=0$.
This contradiction proves non-attainment.
\end{proof}

\begin{corollary}[Angular Euler--Lagrange equation]
For every $\psi\in\mathscr Y_s$,
\begin{equation}\label{eq:EL}
 \mathfrak b_{s,0}(\phi_{s,\tau},\psi)
 =\Lambda_{s,\tau}
 \int_0^\pi m_{s,\tau}\phi_{s,\tau}\overline{\psi}\,d\theta,
\end{equation}
where $\mathfrak b_{s,0}(\cdot,\cdot)$ is the polarization of
\eqref{eq:intro-fiber}. The same equation characterizes
$\phi_{s,\tau}$ among normalized nonnegative functions in $\mathscr Y_s$.
\end{corollary}

\begin{proof}
For $\psi\in\mathscr Y_s$ and $r\in\mathbb R$, minimality gives
\[
 \mathfrak b_{s,0}[\phi_{s,\tau}+r\psi]
 -\Lambda_{s,\tau}
 \|\phi_{s,\tau}+r\psi\|_{H_{s,\tau}}^2\geq0.
\]
Differentiating at $r=0$ gives the real part of \eqref{eq:EL}.
Replacing $\psi$ by $i\psi$ gives the imaginary part.
Conversely, testing the same equation for a normalized nonnegative
$\phi$ against $\phi$ shows that $\phi$ is a minimizer.
Lemma~\ref{lem:positive-simple} then gives $\phi=\phi_{s,\tau}$.
\end{proof}

\begin{theorem}[Ground-state representation at the sharp level]
\label{thm:sharp-ground-state}
Let $0<s<1$ and $\tau>0$. For
$x=r(\cos\theta,\sin\theta)\in\mathbb R^2_+$, put
\[
 U_{s,\tau}(x)=r^{s-1}\phi_{s,\tau}(\theta).
\]
This positive homogeneous function satisfies the weak equation
\begin{equation}\label{eq:sharp-ground-state-equation}
 \begin{split}
 &\int_{\mathbb R^2_+}\int_{\mathbb R^2_+}
 \frac{(U_{s,\tau}(x)-U_{s,\tau}(y))
 (\overline{\eta(x)}-\overline{\eta(y)})}
 {|x-y|^{2+2s}}\,dx\,dy\\
 &\quad=\int_{\mathbb R^2_+}
 \left(\mathcal D_{2,s,2}x_2^{-2s}
 +\Lambda_{s,\tau}x_2^{\tau-2s}|x|^{-\tau}\right)
 U_{s,\tau}(x)\overline{\eta(x)}\,dx
 \end{split}
\end{equation}
for $\eta\in C_c^\infty(\mathbb R^2_+)$. For every core function $u$,
\begin{equation}\label{eq:sharp-ground-state-representation}
 \begin{split}
 &\mathcal H_s[u]-\Lambda_{s,\tau}\mathcal R_{s,\tau}[u]=\int_{\mathbb R^2_+}\int_{\mathbb R^2_+}
 \frac{U_{s,\tau}(x)U_{s,\tau}(y)}{|x-y|^{2+2s}}
 \left|\frac{u(x)}{U_{s,\tau}(x)}
       -\frac{u(y)}{U_{s,\tau}(y)}\right|^2\,dx\,dy.
 \end{split}
\end{equation}
The identity extends to the completed energy space, with the left-hand
side interpreted as the closed sharp deficit. In cylinder coordinates it reads
\begin{equation}\label{eq:sharp-cylinder-ground-state}
 \begin{split}
 &\mathfrak q_s[w]-\Lambda_{s,\tau}
 \|w\|_{L^2(\mathbb R;H_{s,\tau})}^2\\
 &\quad=\int_{\mathbb R^2}\int_0^\pi\int_0^\pi
 K_s(t-t',\theta,\varphi)\phi_{s,\tau}(\theta)
 \phi_{s,\tau}(\varphi)
 \left|\frac{w(t,\theta)}{\phi_{s,\tau}(\theta)}
       -\frac{w(t',\varphi)}{\phi_{s,\tau}(\varphi)}\right|^2
 \,d\theta\,d\varphi\,dt\,dt'.
 \end{split}
\end{equation}
\end{theorem}

\begin{proof}
Write $\phi=\phi_{s,\tau}$, $\Lambda=\Lambda_{s,\tau}$, and
$a_s=1-s$. All double integrals below are ordered integrals. Put
\[
 \begin{split}
 \mathcal J_s(\theta,\varphi)&=\int_{\mathbb R}
 K_s(q,\theta,\varphi)\,dq,\\
 W_s(\theta)&=2\int_0^\pi\int_{\mathbb R}
 (\cosh(a_sq)-1)K_s(q,\theta,\varphi)\,dq\,d\varphi,\\
 V_s(\theta)&=W_s(\theta)
 -\mathcal D_{2,s,2}(\sin\theta)^{-2s}.
 \end{split}
\]
For $|q|<1$, integration in $\varphi$ gives
\[
 \int_0^\pi K_s(q,\theta,\varphi)\,d\varphi
 \leq C_s|q|^{-1-2s},\qquad
 |\cosh(a_sq)-1|\leq C_sq^2.
\]
For $|q|\geq1$, the product is at most $C_se^{-2s|q|}$.
Consequently,
\[
 0\leq W_s(\theta)\leq C_s\left(
 \int_0^1q^{1-2s}\,dq+\int_1^\infty e^{-2sq}\,dq\right)<\infty
\]
uniformly in $\theta$. In particular,
$V_s-\Lambda m_{s,\tau}\in L^\infty(I)$ for every
$I\Subset(0,\pi)$.

For $w\in\mathscr C$, define
\[
 \mathfrak j_s[w]=\int_{\mathbb R^2}\int_0^\pi\int_0^\pi
 K_s(t-t',\theta,\varphi)
 |w(t,\theta)-w(t',\varphi)|^2
 \,d\theta\,d\varphi\,dt\,dt'.
\]
Near the diagonal, the smooth difference is bounded by
$C(|t-t'|+|\theta-\varphi|)$. Its energy is therefore bounded by
$C\int_0^1r^{1-2s}\,dr$. Off the diagonal, the compact support of $w$
and the exponential kernel tail give an integrable majorant. Thus
$\mathfrak j_s[w]<\infty$. To expand \eqref{eq:q-cylinder}, first use
\[
 K_s^\delta(q,\theta,\varphi)
 =K_s(q,\theta,\varphi)
 \boldsymbol 1_{\{q^2+|\theta-\varphi|^2>\delta^2\}}.
\]
In the coefficient of $|w(t,\theta)|^2$, symmetry in $q$ changes
$e^{-a_sq}$ into $\cosh(a_sq)$. Interchanging
$(t,\theta)$ and $(t',\varphi)$ in the other diagonal term gives
\[
 \begin{split}
 &\int K_s^\delta
 \left|e^{-a_s(t-t')/2}w(t,\theta)
       -e^{a_s(t-t')/2}w(t',\varphi)\right|^2\\
 &\quad=\int K_s^\delta|w(t,\theta)-w(t',\varphi)|^2
 +\int_{\mathbb R}\int_0^\pi W_s^\delta(\theta)|w(t,\theta)|^2
 \,d\theta\,dt,
 \end{split}
\]
where $W_s^\delta$ has the same cutoff. The preceding bounds permit
$\delta\downarrow0$ and prove
\begin{equation}\label{eq:bare-cylinder-jump}
 \mathfrak q_s[w]=\mathfrak j_s[w]
 +\int_{\mathbb R}\int_0^\pi V_s(\theta)|w(t,\theta)|^2
 \,d\theta\,dt.
\end{equation}
Integrals without displayed limits in this proof are over the four
cylinder variables, unless their variables indicate otherwise.

We next record the localized angular polarization:
\begin{equation}\label{eq:localized-angular-polarization}
 \begin{split}
 \mathfrak b_{s,0}(f,\eta)
 ={}&\int_0^\pi\int_0^\pi\mathcal J_s(\theta,\varphi)
 (f(\theta)-f(\varphi))
 (\overline{\eta(\theta)}-\overline{\eta(\varphi)})
 \,d\theta\,d\varphi\\
 &+\int_0^\pi V_sf\overline\eta\,d\theta,
 \end{split}
\end{equation}
for $f\in\mathscr Y_s$ and $\eta\in H^s(0,\pi)$ with compact support.
For core functions, expand \eqref{eq:intro-fiber} at $\xi=0$ with
$K_s^\delta$ and use the preceding symmetry. To pass to the stated
domains, choose intervals
$\operatorname{supp}\eta\subset I\Subset I_1\Subset(0,\pi)$.
The local kernel bound gives
\[
 \mathcal J_s(\theta,\varphi)
 \leq C_s|\theta-\varphi|^{-1-2s}.
\]
Hence the part with both variables in $I_1$ is bounded in absolute value by
\[
 C_{I_1}\|f\|_{H^s(I_1)}\|\eta\|_{H^s(I_1)}.
\]
If one variable is outside $I_1$ and the integrand is nonzero, the other
belongs to $I$; the variables are separated. This part is bounded by
\[
 C_{I,I_1}\bigl(\|f\|_{L^2(0,\pi)}+\|f\|_{L^2(I)}\bigr)
 \|\eta\|_{L^2(I)}.
\]
The potential term is bounded by
$\|V_s\|_{L^\infty(I)}\|f\|_{L^2(I)}\|\eta\|_{L^2(I)}$.
Use \eqref{eq:local-angular} and core approximation for $f$.
For $\eta$, zero extension and mollification converge in $H^s$ with
support in $I_1$; the local upper estimate in
Lemma~\ref{lem:modulus-domain} gives convergence in $\mathscr Y_s$.
This proves \eqref{eq:localized-angular-polarization}, including absolute
convergence of its mixed integral. We do not split the angular form this
way for two arbitrary elements of $\mathscr Y_s$.

Fix $w\in\mathscr C$ and put
\[
 N(\theta)=\int_{\mathbb R}|w(t,\theta)|^2\,dt,
 \qquad \Phi_\varepsilon=\phi+\varepsilon,
 \qquad \eta_\varepsilon=\frac{N}{\Phi_\varepsilon},
 \quad\varepsilon>0.
\]
Here $N\in C_c^\infty(0,\pi)$. On a compact interval containing its
support, $\phi\in H^s$ and $\phi\geq0$. For $a,b\geq0$,
\[
 \left|\frac1{a+\varepsilon}-\frac1{b+\varepsilon}\right|
 \leq\varepsilon^{-2}|a-b|,\qquad
 \frac1{a+\varepsilon}\leq\varepsilon^{-1}.
\]
Multiplication by $N$ therefore gives
\[
 \|\eta_\varepsilon\|_{H^s(\mathbb R)}
 \leq C_N(\varepsilon^{-1}
 +\varepsilon^{-2}\|\phi\|_{H^s(I_1)}),
\]
where $\eta_\varepsilon$ is extended by zero. Thus
$\eta_\varepsilon\in\mathscr Y_s$ and is an admissible real test in
\eqref{eq:EL}. By \eqref{eq:localized-angular-polarization},
\[
 \int_0^\pi\int_0^\pi\mathcal J_s
 (\phi(\theta)-\phi(\varphi))
 (\eta_\varepsilon(\theta)-\eta_\varepsilon(\varphi))
 \,d\theta\,d\varphi
 =-\int_0^\pi(V_s-\Lambda m_{s,\tau})
 \frac{\phi N}{\phi+\varepsilon}\,d\theta.
\]
For $A,A'>0$ and $z,z'\in\mathbb C$, direct expansion gives
\[
 |z-z'|^2-AA'|z/A-z'/A'|^2
 =(A-A')\left(\frac{|z|^2}{A}-\frac{|z'|^2}{A'}\right).
\]
Apply this identity with $A=\Phi_\varepsilon(\theta)$,
$A'=\Phi_\varepsilon(\varphi)$ and integrate against $K_s^\delta$.
Integration in $t,t'$ of the right-hand side gives exactly
\[
 \int_0^\pi\int_0^\pi\mathcal J_s^\delta
 (\phi(\theta)-\phi(\varphi))
 (\eta_\varepsilon(\theta)-\eta_\varepsilon(\varphi))
 \,d\theta\,d\varphi.
\]
Indeed, in the first summand put $q=t-t'$ and integrate $|w(t,\theta)|^2$
in $t$; in the second integrate $|w(t',\varphi)|^2$ in $t'$.
The mixed angular integral converges absolutely by the estimates above.
Monotone convergence applies to each of the two nonnegative cylinder
energies. Letting $\delta\downarrow0$ yields
\begin{equation}\label{eq:epsilon-sharp-transform}
 \begin{split}
 I_\varepsilon:={}&\int K_s\Phi_\varepsilon(\theta)
 \Phi_\varepsilon(\varphi)
 \left|\frac{w(t,\theta)}{\Phi_\varepsilon(\theta)}
 -\frac{w(t',\varphi)}{\Phi_\varepsilon(\varphi)}\right|^2\\
 ={}&\mathfrak j_s[w]+\int_0^\pi(V_s-\Lambda m_{s,\tau})
 \frac{\phi}{\phi+\varepsilon}N\,d\theta.
 \end{split}
\end{equation}
Since $\phi>0$ almost everywhere and the potential is bounded on
$\operatorname{supp}N$, dominated convergence on the right and \eqref{eq:bare-cylinder-jump} give
\[
 \lim_{\varepsilon\downarrow0}I_\varepsilon
 =\mathfrak j_s[w]+\int_0^\pi(V_s-\Lambda m_{s,\tau})N\,d\theta
 =\mathfrak q_s[w]-\Lambda\|w\|_{L^2(m_{s,\tau})}^2.
\]
Denote the integral on the right of
\eqref{eq:sharp-cylinder-ground-state} by $I_0$. Fatou's lemma gives
$I_0\leq\liminf_{\varepsilon\downarrow0}I_\varepsilon<\infty$.
To prove convergence to $I_0$, put $r=A'/A$ and
$r_\varepsilon=(A'+\varepsilon)/(A+\varepsilon)$. There is
$\gamma_\varepsilon\in[0,1]$ such that
$\log r_\varepsilon=\gamma_\varepsilon\log r$.
The function
\[
 F(t)=e^t|z|^2+e^{-t}|z'|^2
 -2\operatorname{Re}(z\overline{z'})
\]
is nonnegative and convex. Hence
\[
 \begin{split}
 &(A+\varepsilon)(A'+\varepsilon)
 |z/(A+\varepsilon)-z'/(A'+\varepsilon)|^2\\
 &\quad=F(\log r_\varepsilon)
 \leq\gamma_\varepsilon F(\log r)
 +(1-\gamma_\varepsilon)F(0)
 \leq AA'|z/A-z'/A'|^2+|z-z'|^2.
 \end{split}
\]
The majorant is integrable against $K_s$, since its integral is
$I_0+\mathfrak j_s[w]$. Dominated convergence gives
$I_\varepsilon\to I_0$ in \eqref{eq:epsilon-sharp-transform} and proves
\eqref{eq:sharp-cylinder-ground-state}.

We now verify the weak equation. For $v\in\mathscr C$, set
$\eta_0(\theta)=\int_{\mathbb R}v(t,\theta)\,dt$.
Using a common cutoff in the stationary-cylinder pairing, Fubini gives
\[
 \begin{split}
 &\int K_s(\phi(\theta)-\phi(\varphi))
 (\overline{v(t,\theta)}-\overline{v(t',\varphi)})
 +\int_{\mathbb R}\int_0^\pi V_s\phi\overline v\,d\theta\,dt\\
 &\quad=\mathfrak b_{s,0}(\phi,\eta_0)
 =\Lambda\int_{\mathbb R}\int_0^\pi
 m_{s,\tau}\phi\overline v\,d\theta\,dt.
 \end{split}
\]
Absolute convergence follows locally from the $H^s$ bound and
Cauchy--Schwarz; on separated angular pairs it follows from
$\phi\in L^2(0,\pi)$. The same bounds justify removal of the cutoff.
For a physical test $\eta\in C_c^\infty(\mathbb R^2_+)$, take
\[
 v(t,\theta)=e^{(1-s)t}\eta(e^t\cos\theta,e^t\sin\theta).
\]
The stationary profile corresponds to
$U(x)=r^{s-1}\phi(\theta)$. Polarization of the calculation proving
\eqref{eq:H-transform}, with the same cutoff, identifies the preceding
pairing with the physical energy pairing minus the Hardy term. Also,
\[
 \begin{split}
 \int_{\mathbb R^2_+}x_2^{-2s}U\overline\eta\,dx
 &=\int_{\mathbb R}\int_0^\pi
 (\sin\theta)^{-2s}\phi\overline v\,d\theta\,dt,\\
 \int_{\mathbb R^2_+}x_2^{\tau-2s}|x|^{-\tau}U\overline\eta\,dx
 &=\int_{\mathbb R}\int_0^\pi
 m_{s,\tau}\phi\overline v\,d\theta\,dt.
 \end{split}
\]
For completeness, the physical mixed energy is absolutely convergent.
On a compact neighborhood of $\operatorname{supp}\eta$, the smooth polar
change of variables and $\phi\in H^s_{\rm loc}$ give $U\in H^s$.
Cauchy--Schwarz bounds the local mixed integral. If $x$ is in the test
support and $|y|$ is small or large, its remaining part is bounded by
\[
 C_\eta\left[1+\|\phi\|_{L^1(0,\pi)}
 \left(\int_0^1r^s\,dr+\int_1^\infty r^{-2-s}\,dr\right)\right]<\infty.
\]
The separated bounded region is controlled by $U\in L^1_{\rm loc}$.
Thus the cutoff may also be removed in physical coordinates. This proves
\eqref{eq:sharp-ground-state-equation}.

Finally, $u/U=w/\phi$, and \eqref{eq:distance-polar} gives
\[
 \frac{U(x)U(y)}{|x-y|^{2+2s}}\,dx\,dy
 =K_s(t-t',\theta,\varphi)\phi(\theta)\phi(\varphi)
 \,dt\,dt'\,d\theta\,d\varphi.
\]
Substitution in \eqref{eq:sharp-cylinder-ground-state} proves
\eqref{eq:sharp-ground-state-representation} on the core.

To pass to the completion, let $w_n\to w$ in $\mathscr X_s$ and define
\[
 (Dw)(t,t',\theta,\varphi)=
 \sqrt{\phi(\theta)\phi(\varphi)}
 \left(\frac{w(t,\theta)}{\phi(\theta)}
 -\frac{w(t',\varphi)}{\phi(\varphi)}\right).
\]
Use the kernel measure $K_s\,dt\,dt'\,d\theta\,d\varphi$.
The core identity gives
\[
 \|D(w_n-w_m)\|_{L^2(K_s)}^2
 =\mathfrak q_s[w_n-w_m]
 -\Lambda\|w_n-w_m\|_{L^2(m_{s,\tau})}^2
 \leq\mathfrak q_s[w_n-w_m]\longrightarrow0.
\]
Thus $Dw_n\to F$ in the kernel space. A subsequence converges almost
everywhere both in the base space and in the kernel space. Since
$\phi>0$ almost everywhere and the kernel measure is absolutely
continuous off the null diagonal, its pointwise limit is $Dw$.
Therefore $F=Dw$. Taking norms and using continuity of the closed form,
\[
 \|Dw\|_{L^2(K_s)}^2
 =\lim_{n\to\infty}\|Dw_n\|_{L^2(K_s)}^2
 =\mathfrak q_s[w]-\Lambda\|w\|_{L^2(m_{s,\tau})}^2.
\]
This proves both completed-space identities.
\end{proof}

The next proposition identifies the angular stability constant with the
spectral gap. The coefficient in \eqref{eq:intro-mainB} also depends on
variation in the logarithmic radial variable.

\begin{proposition}[Attainment of the angular stability constant]
\label{prop:angular-stability}
Define
\[
 \Gamma^{\rm ang}_{s,\tau}
 =\inf_{\substack{\psi\in\mathscr Y_s\setminus\{0\}\\
 \int_0^\pi m_{s,\tau}\psi\phi_{s,\tau}\,d\theta=0}}
 \frac{\mathfrak b_{s,0}[\psi]
 -\Lambda_{s,\tau}\|\psi\|_{L^2(m_{s,\tau})}^2}
 {\|\psi\|_{L^2(m_{s,\tau})}^2}.
\]
Then
\[
 \Gamma^{\rm ang}_{s,\tau}
 =\lambda_2(0)-\lambda_1(0)>0.
\]
The infimum is attained. Equality in
\begin{equation}\label{eq:angular-stability-inequality}
 \mathfrak b_{s,0}[\psi]-\Lambda_{s,\tau}\|\psi\|_{H_{s,\tau}}^2
 \geq\Gamma^{\rm ang}_{s,\tau}
 \operatorname{dist}_{H_{s,\tau}}
 (\psi,\operatorname{span}\{\phi_{s,\tau}\})^2
\end{equation}
holds precisely when the orthogonal component of $\psi$ belongs to the
second eigenspace of $B_s(0)$.
\end{proposition}

\begin{proof}
Write
\[
 \psi=c\phi_{s,\tau}+\psi^\perp,\qquad
 \int_0^\pi m_{s,\tau}\psi^\perp\phi_{s,\tau}\,d\theta=0.
\]
By \eqref{eq:EL}, the left-hand side of
\eqref{eq:angular-stability-inequality} equals
\[
 \mathfrak b_{s,0}[\psi^\perp]
 -\lambda_1(0)\|\psi^\perp\|_{H_{s,\tau}}^2.
\]
The min--max principle and \eqref{eq:compact-embedding} give the
attained value $\lambda_2(0)-\lambda_1(0)$. It is positive by
Lemma~\ref{lem:positive-simple}. Expanding
$\psi^\perp=\sum_{j\geq2}c_je_j$ in an orthonormal eigenbasis, we have
\[
 \mathfrak b_{s,0}[\psi^\perp]
 -\lambda_1(0)\|\psi^\perp\|^2
 =\sum_{j\geq2}(\lambda_j(0)-\lambda_1(0))|c_j|^2.
\]
Equality holds if and only if $c_j=0$ whenever
$\lambda_j(0)>\lambda_2(0)$.
\end{proof}
\section{Spectral stability and energy decomposition}
\label{sec:stability}\label{sec:sequences}

\begin{proposition}[Dispersion at the threshold]\label{prop:dispersion}
The first eigenvalue is even and real analytic for $|\xi|$ small. It satisfies
\[
 \lambda_1'(0)=0,
 \qquad \frac12\lambda_1''(0)=\kappa_{s,\tau},
\]
where $\kappa_{s,\tau}$ is given by \eqref{eq:intro-kappa}.
\end{proposition}

\begin{proof}
By Lemma~\ref{lem:kernel-moments}, the forms $\mathfrak b_{s,z}$ are a
holomorphic family of type (a) on the common domain $\mathscr Y_s$.
The associated operators form a family of type (B)
\cite[Chapter~VII, Section~4]{Kato}. Since $\lambda_1(0)$ is simple and
isolated, a Riesz projection of rank one gives analytic branches
$\lambda_1(z)$ and $\phi(z)$ near zero. For real $\xi$, normalize by
\[
 \|\phi(\xi)\|_{H_{s,\tau}}=1,\qquad
 \langle\phi(\xi),\phi_{s,\tau}\rangle_{H_{s,\tau}}>0.
\]
The real form depends on $\xi$ through $\cos(\xi q)$. Uniqueness of the
simple branch gives $\lambda_1(-\xi)=\lambda_1(\xi)$ and
$\phi(-\xi)=\phi(\xi)$.

Equation \eqref{eq:analytic-form-expansion} gives
$\mathfrak b'_{s,0}=0$ in $\mathcal L(\mathscr Y_s,\mathscr Y_s^*)$ and
$\mathfrak b''_{s,0}=2\mathfrak c_{s,2}$. The form version of the
Feynman--Hellmann formula therefore yields
\[
 \lambda_1'(0)=0,\qquad
 \lambda_1''(0)=\mathfrak b''_{s,0}(\phi_{s,\tau},\phi_{s,\tau}).
\]
Here the derivative of the eigenvector can be checked directly.
Differentiating the weak eigenvalue equation at zero gives
\[
 (B_s(0)-\Lambda_{s,\tau})\phi'(0)=0
 \quad\text{in }\mathscr Y_s^*.
\]
The real normalization and phase condition give
$\langle\phi'(0),\phi_{s,\tau}\rangle_{H_{s,\tau}}=0$.
Simplicity implies $\phi'(0)=0$. Differentiating a second time and
pairing with $\phi_{s,\tau}$ cancels the terms containing $\phi''(0)$ by
\eqref{eq:EL}. Thus only $\mathfrak b''_{s,0}(\phi_{s,\tau},\phi_{s,\tau})$
remains. By \eqref{eq:second-form-coefficient},
\[
 \lambda_1''(0)=
 2\int_0^\pi\int_0^\pi\int_{\mathbb R}
 q^2K_s(q,\theta,\varphi)
 \phi_{s,\tau}(\theta)\phi_{s,\tau}(\varphi)
 \,dq\,d\theta\,d\varphi.
\]
This integral is finite by Lemma~\ref{lem:kernel-moments} and positive
because $K_s>0$ and $\phi_{s,\tau}>0$ almost everywhere. Even analyticity
gives the $O(\xi^4)$ remainder.
\end{proof}

\begin{lemma}[Axial dispersion function]
The function $\rho_{s,\tau}$ in \eqref{eq:intro-rho} is even and real
analytic. It satisfies
\begin{equation}\label{eq:rho-expansion}
 \rho_{s,\tau}(\xi)
 =\kappa_{s,\tau}\xi^2
 -\frac{\xi^4}{12}
 \int_0^\pi\int_0^\pi\int_{\mathbb R}
 q^4K_s(q,\theta,\varphi)
 \phi_{s,\tau}(\theta)\phi_{s,\tau}(\varphi)
 \,dq\,d\theta\,d\varphi+O(\xi^6)
\end{equation}
at zero. Moreover,
\begin{equation}\label{eq:rho-positive-growth}
 \rho_{s,\tau}(\xi)>0\quad(\xi\neq0),\qquad
 c|\xi|^{2s}-C\leq\rho_{s,\tau}(\xi)
 \leq C(1+|\xi|^{2s}).
\end{equation}
\end{lemma}

\begin{proof}
First use truncated kernels in \eqref{eq:intro-rho}. The integrand is
nonnegative and, for $\xi\neq0$, positive for almost every
$(q,\theta,\varphi)$. This proves strict positivity.
By Lemma~\ref{lem:kernel-moments}, Taylor expansion under the integral is
valid.
Since
\[
 2(1-\cos(\xi q))
 =\xi^2q^2-\frac{\xi^4q^4}{12}+O(\xi^6q^6),
\]
we obtain \eqref{eq:rho-expansion} and real analyticity. Finally,
\[
 \rho_{s,\tau}(\xi)
 =\mathfrak b_{s,\xi}[\phi_{s,\tau}]
 -\mathfrak b_{s,0}[\phi_{s,\tau}].
\]
The two estimates in \eqref{eq:rho-positive-growth} follow from
\eqref{eq:ground-fiber-growth} and
$\mathfrak b_{s,0}[\phi_{s,\tau}]=\Lambda_{s,\tau}$.
\end{proof}

Let $P(\xi)$ be the first spectral projection for $|\xi|$ small.

\begin{lemma}[Uniform threshold estimate]\label{lem:threshold-estimate}
There are $\xi_0,c_0,c_1>0$ such that
\begin{equation}\label{eq:small-frequency-gap}
 \begin{split}
 \mathfrak b_{s,\xi}[\psi]-\Lambda_{s,\tau}
 \|\psi\|_{L^2(m_{s,\tau})}^2
 \geq{}&c_0\xi^2\|P(\xi)\psi\|_{L^2(m_{s,\tau})}^2+c_1\|(I-P(\xi))\psi\|_{L^2(m_{s,\tau})}^2
 \end{split}
\end{equation}
for $|\xi|\leq\xi_0$. Moreover,
\begin{equation}\label{eq:projection-close}
 \|P(\xi)-P_0\|_{H_{s,\tau}\to H_{s,\tau}}
 +\|P(\xi)-P_0\|_{H_{s,\tau}\to\mathscr Y_s}
 \leq C\xi^2.
\end{equation}
\end{lemma}

\begin{proof}
Put $g_0=\lambda_2(0)-\lambda_1(0)>0$. By continuity and
Proposition~\ref{prop:dispersion}, choose $\xi_0>0$ such that
\[
 \lambda_1(\xi)-\Lambda_{s,\tau}\geq
 \frac12\kappa_{s,\tau}\xi^2,
 \qquad
 \lambda_2(\xi)-\Lambda_{s,\tau}\geq\frac12g_0.
\]
The spectral theorem gives \eqref{eq:small-frequency-gap}. The Riesz
formula
\[
 P(\xi)=-\frac1{2\pi i}\int_\Gamma(B_s(\xi)-z)^{-1}\,dz
\]
around the first eigenvalue gives $P'(0)=0$, since the family is even.
To estimate this integral in the form norm, define
$\mathcal A_\xi(z):\mathscr Y_s\to\mathscr Y_s^*$ by
\[
 \langle\mathcal A_\xi(z)u,v\rangle
 =\mathfrak b_{s,\xi}(u,v)-z\langle u,v\rangle_{H_{s,\tau}}.
\]
For $c>0$, the map $\mathcal A_0(-c)$ is a coercive isomorphism.
The map $J$ induced by the $H_{s,\tau}$ inner product is compact by
\eqref{eq:compact-embedding}, and
\[
 \mathcal A_0(z)=\mathcal A_0(-c)
 \{I-(z+c)\mathcal A_0(-c)^{-1}J\}.
\]
Thus $\mathcal A_0(z)$ is Fredholm of index zero. For $z\in\Gamma$, its
kernel is zero, so it is invertible with a bound uniform on $\Gamma$.
Equation \eqref{eq:analytic-form-expansion} and a Neumann series give the
same bound for $R_\xi(z)=\mathcal A_\xi(z)^{-1}$ when $|\xi|$ is small.
The restriction of $R_\xi(z):\mathscr Y_s^*\to\mathscr Y_s$ to
$H_{s,\tau}$ is the operator resolvent. Moreover,
\[
 \|\mathcal A_\xi(z)-\mathcal A_0(z)\|
 _{\mathscr Y_s\to\mathscr Y_s^*}\leq C\xi^2.
\]
The resolvent identity in these spaces is
\[
 R_\xi(z)-R_0(z)
 =-R_\xi(z)(\mathcal A_\xi(z)-\mathcal A_0(z))R_0(z).
\]
The right-hand side is $O(\xi^2)$ from $H_{s,\tau}$ to $\mathscr Y_s$,
uniformly on $\Gamma$. Integration proves \eqref{eq:projection-close}.
For real $\xi$ near zero, take
\[
 \phi(\xi)=\frac{P(\xi)\phi_{s,\tau}}
 {\|P(\xi)\phi_{s,\tau}\|_{H_{s,\tau}}}.
\]
Consequently,
$\|\phi(\xi)-\phi_{s,\tau}\|_{\mathscr Y_s}\leq C\xi^2$.
\end{proof}

\begin{lemma}[Gap away from zero]\label{lem:away-gap}
For every $\xi_0>0$, there is $c(\xi_0)>0$ such that
\begin{equation}\label{eq:away-gap}
 \mathfrak b_{s,\xi}[\psi]-\Lambda_{s,\tau}
 \|\psi\|_{L^2(m_{s,\tau})}^2
 \geq c(\xi_0)\|\psi\|_{L^2(m_{s,\tau})}^2
\end{equation}
when $|\xi|\geq\xi_0$.
\end{lemma}

\begin{proof}
By Lemmas~\ref{lem:real-frequency-continuity} and~\ref{lem:diamagnetic},
$\lambda_1(\xi)-\Lambda_{s,\tau}$ has a positive minimum on
$\xi_0\leq|\xi|\leq R$. For large $R$,
Lemma~\ref{lem:large-frequency} gives a positive bound on $|\xi|\geq R$.
The variational definition of $\lambda_1(\xi)$ proves \eqref{eq:away-gap}.
\end{proof}

\begin{proof}[Proof of Theorem~\ref{thm:B}]
Apply Lemma~\ref{lem:threshold-estimate} to
$\widehat w(\xi,\cdot)$ for $|\xi|\leq\xi_0$. By
\eqref{eq:projection-close}, the two squares with $P(\xi)$ control
\[
 \xi^2\|P_0\widehat w\|_{L^2(m_{s,\tau})}^2
 +\|(I-P_0)\widehat w\|_{L^2(m_{s,\tau})}^2.
\]
Indeed, put $P=P(\xi)$ and $Q=I-P$. Then
\[
 \begin{split}
 \xi^2\|P_0f\|^2
 &\leq C\xi^2\|Pf\|^2+C\xi^6\|Qf\|^2,\\
 \|Q_0f\|^2
 &\leq C\|Qf\|^2+C\xi^4\|Pf\|^2.
 \end{split}
\]
Take $\xi_0\leq1$. Adding these inequalities gives
\[
 \xi^2\|P_0f\|^2+\|Q_0f\|^2
 \leq C\bigl(\xi^2\|P(\xi)f\|^2
 +\|(I-P(\xi))f\|^2\bigr)
 \leq C D_\xi[f],
\]
where $D_\xi[f]=\mathfrak b_{s,\xi}[f]-\Lambda_{s,\tau}\|f\|^2$.
For $|\xi|\geq\xi_0$, Lemma~\ref{lem:away-gap} gives
\[
 h(\xi)\|P_0f\|^2+\|Q_0f\|^2
 \leq\|f\|^2\leq c(\xi_0)^{-1}D_\xi[f].
\]
Integrate with $f=\widehat w(\xi,\cdot)$ and use
\eqref{eq:closed-direct-integral} and \eqref{eq:intro-distance}.
This proves \eqref{eq:intro-mainB}; density extends the estimate from
the core to $\mathscr X_s$. Proposition~\ref{prop:dispersion} gives
\eqref{eq:intro-dispersion}.
\end{proof}

\begin{proof}[Proof of Theorem~\ref{thm:optimal-hp}]
\smallskip
\noindent\emph{Step 1: fixed-frequency minimizers.}
For $\xi\neq0$, write
\[
 \begin{split}
 \mathfrak d_\xi[\psi]
 &=\mathfrak b_{s,\xi}[\psi]
 -\Lambda_{s,\tau}\|\psi\|_{H_{s,\tau}}^2,\\
 \mathfrak n_\xi[\psi]
 &=\|Q_0\psi\|_{H_{s,\tau}}^2
 +h(\xi)\|P_0\psi\|_{H_{s,\tau}}^2.
 \end{split}
\]
For fixed $\xi\neq0$, the form $\mathfrak n_\xi$ is equivalent to the
squared $H_{s,\tau}$ norm. The proof of Theorem~\ref{thm:B} gives
$\mathfrak d_\xi\geq c\mathfrak n_\xi$. A minimizing sequence for
\eqref{eq:intro-mu}, normalized by $\mathfrak n_\xi[\psi_j]=1$, satisfies
\[
 \mathfrak b_{s,\xi}[\psi_j]
 =\mathfrak d_\xi[\psi_j]
 +\Lambda_{s,\tau}\|\psi_j\|_{H_{s,\tau}}^2
\]
with a bounded right-hand side. Hence a subsequence converges weakly in
$\mathscr Y_s$ and strongly in $H_{s,\tau}$ by
\eqref{eq:compact-embedding}. Lower semicontinuity gives a minimizer, and
$\mu_{s,\tau}(\xi)\geq c>0$.

\smallskip
\noindent\emph{Step 2: continuity away from zero and coercivity at
infinity.}
Let $\xi_j\to\xi\neq0$. Testing with a minimizer at $\xi$ gives upper
semicontinuity. For the lower bound, normalize minimizers at $\xi_j$ by
$\mathfrak n_{\xi_j}[\psi_j]=1$. These norms are uniformly equivalent to
$\|\cdot\|_{H_{s,\tau}}^2$ near $\xi$. A fixed trial function bounds
$\mu_{s,\tau}(\xi_j)$, so $\psi_j$ is bounded in $\mathscr Y_s$.
For every compact frequency interval $K$, the proof of
\eqref{eq:fiber-difference-bound} gives
\[
 |\mathfrak b_{s,\xi}[\psi]-\mathfrak b_{s,\eta}[\psi]|
 \leq\omega_K(|\xi-\eta|)\|\psi\|_{L^2(0,\pi)}^2,
 \qquad \omega_K(r)\longrightarrow0.
\]
After passage to a subsequence,
$\psi_j\rightharpoonup\psi$ in $\mathscr Y_s$ and
$\psi_j\to\psi$ in $H_{s,\tau}$. Continuity of $h$ gives
$\mathfrak n_\xi[\psi]=1$. The two-frequency estimate and
\eqref{eq:unweighted-control} give
\[
 |\mathfrak b_{s,\xi_j}[\psi_j]
 -\mathfrak b_{s,\xi}[\psi_j]|\longrightarrow0.
\]
Weak lower semicontinuity now yields
\[
 \mu_{s,\tau}(\xi)\leq\mathfrak d_\xi[\psi]
 \leq\liminf_{j\to\infty}\mu_{s,\tau}(\xi_j).
\]
Thus $\mu_{s,\tau}$ is continuous away from zero. For $|\xi|\geq1$,
$\mathfrak n_\xi[\psi]=\|\psi\|_{H_{s,\tau}}^2$, and
\[
 \mu_{s,\tau}(\xi)=\lambda_1(\xi)-\Lambda_{s,\tau}
 \longrightarrow\infty
\]
as $|\xi|\to\infty$.

\smallskip
\noindent\emph{Step 3: the direct-integral minimum.}
By the closed direct-integral identity,
\begin{equation*}
 \frac{\mathfrak q_s[w]-\Lambda_{s,\tau}\|w\|_{L^2(m_{s,\tau})}^2}
 {\operatorname{Dist}_{s,\tau}(u)^2}
 =
 \frac{\int_{\mathbb R}\mathfrak d_\xi[\widehat w(\xi)]\,d\xi}
 {\int_{\mathbb R}\mathfrak n_\xi[\widehat w(\xi)]\,d\xi}.
\end{equation*}
This quotient is at least $\operatorname*{ess\,inf}\mu_{s,\tau}$.
Conversely, choose $\xi_*\neq0$ with
$\mu_{s,\tau}(\xi_*)<\operatorname*{ess\,inf}\mu_{s,\tau}+\varepsilon$.
Let $\psi_*$ be its minimizing vector, normalized by
$\mathfrak n_{\xi_*}[\psi_*]=1$. Choose
$\chi\in C_c^\infty((-1,1))$ with $\|\chi\|_2=1$ and set
\[
 \widehat w_\delta(\xi,\theta)
 =\delta^{-1/2}\chi((\xi-\xi_*)/\delta)\psi_*(\theta),
 \qquad 0<\delta<|\xi_*|/2.
\]
Uniform graph equivalence on its bounded frequency support shows that
$w_\delta\in\mathscr X_s$. Changing variables to
$\xi=\xi_*+\delta r$ gives
\[
 \begin{split}
 \int\mathfrak d_\xi[\widehat w_\delta(\xi)]\,d\xi
 &=\int_{-1}^1|\chi(r)|^2
 \mathfrak d_{\xi_*+\delta r}[\psi_*]\,dr
 \longrightarrow\mu_{s,\tau}(\xi_*),\\
 \int\mathfrak n_\xi[\widehat w_\delta(\xi)]\,d\xi
 &\longrightarrow1.
 \end{split}
\]
Both limits follow from the form continuity in Step~2. For sufficiently
small $\delta$ the quotient is less than
$\operatorname*{ess\,inf}\mu_{s,\tau}+2\varepsilon$.
The distance is bounded by the base norm, so core approximation preserves
this quotient whenever its denominator is nonzero. If real test
functions are required, use the sum of this packet and its reflected
complex conjugate at $-\xi_*$; their frequency supports are disjoint.
This proves \eqref{eq:intro-gamma-fiber}; positivity follows from
Theorem~\ref{thm:B}.

\smallskip
\noindent\emph{Step 4: exclusion of flat minimizing bands.}
If a nonzero $w$ attains the global quotient, then
\[
 \mathfrak d_\xi[\widehat w(\xi)]
 -\Gamma^{\rm hp}_{s,\tau}\mathfrak n_\xi[\widehat w(\xi)]
\]
is nonnegative and has zero integral. Thus $\widehat w$ is supported in
$Z_{s,\tau}$, which must have positive measure. We prove that this set is
finite.

Write $H=H_{s,\tau}$, $Y=\mathscr Y_s$, $P=P_0$, $Q=I-P$, and
$\Lambda=\Lambda_{s,\tau}$. For $u,v\in Y$, put
\[
 \langle Ju,v\rangle=\langle u,v\rangle_H,\qquad
 \langle J_Pu,v\rangle=\langle Pu,v\rangle_H,\qquad
 \langle J_Qu,v\rangle=\langle Qu,v\rangle_H.
\]
These maps are compact by $Y\hookrightarrow H$.
We also have $Y\hookrightarrow L^2(0,\pi)$ compactly. Indeed, choose
$0<\tau_0<\min\{\tau,2s\}$ and a bounded sequence $v_n$ in $Y$.
Equation \eqref{eq:local-angular} and Rellich's theorem give compactness
on $[\delta,\pi-\delta]$. For
$E_\delta=(0,\delta)\cup(\pi-\delta,\pi)$,
\eqref{eq:all-weight-coercivity} gives
\[
 \begin{split}
 \int_{E_\delta}|v_n|^2\,d\theta
 \leq C\delta^{2s-\tau_0}
 \int_{E_\delta}m_{s,\tau_0}|v_n|^2\,d\theta\leq C\delta^{2s-\tau_0}.
 \end{split}
\]
A diagonal argument gives compactness.
Let $\mathcal S=\{z\in\mathbb C:|\operatorname{Im}z|<1+s\}$.
For $z\in\mathcal S$ and almost every $\alpha\in(-\pi,\pi)$, define
\[
 D_z(\alpha)=\int_{\mathbb R}(1-\cos(zq))K_s(q,\alpha)\,dq,
\]
and define $B(z):Y\to Y^*$ by
\[
 \begin{split}
 \langle(B(z)-B(0))u,v\rangle
 =\int_0^\pi\int_0^\pi D_z(\theta-\varphi)
 \bigl(&u(\theta)\overline{v(\varphi)}+u(\varphi)\overline{v(\theta)}\bigr)
 \,d\theta\,d\varphi.
 \end{split}
\]
For real $z=\xi$, this is the form map of $\mathfrak b_{s,\xi}$.
For $\mathcal K\Subset\mathcal S$, put
$r=\sup_{z\in\mathcal K}|\operatorname{Im}z|<1+s$. For every integer
$k\geq0$,
\[
 |\partial_z^k(1-\cos(zq))|
 \leq C_{\mathcal K,k}q^2\quad(|q|\leq1),
 \qquad
 |\partial_z^k(1-\cos(zq))|
 \leq C_{\mathcal K,k}(1+|q|^k)e^{r|q|}\quad(|q|\geq1).
\]
Since
\[
 \int_{q^2+\alpha^2<1}
 \frac{q^2}{(q^2+\alpha^2)^{1+s}}\,dq\,d\alpha<\infty,
\]
the exponential tail of $K_s$ gives holomorphy of
$z\mapsto D_z$ in $L^1(-\pi,\pi)$. Young's inequality yields
\[
 |\langle(B(z)-B(0))u,v\rangle|
 \leq2\|D_z\|_{L^1(-\pi,\pi)}
 \|u\|_{L^2(0,\pi)}\|v\|_{L^2(0,\pi)}.
\]
Thus $B(z)$ is holomorphic in $\mathcal L(Y,Y^*)$, and $B(z)-B(0)$ is
compact by the preceding $L^2$ embedding. Fix $c>0$ and put
\[
 \begin{split}
 \mathcal F_{<}(z)&=B(z)-\Lambda J-c(J_Q+z^2J_P),\\
 \mathcal F_{>}(z)&=B(z)-(\Lambda+c)J.
 \end{split}
\]
By the form norm and Lax--Milgram, $B(0)+J:Y\to Y^*$ is an isomorphism.
Moreover,
\[
 \begin{split}
 (B(0)+J)^{-1}\mathcal F_{<}(z)
 &=I+(B(0)+J)^{-1}
 \{B(z)-B(0)-(\Lambda+1)J-cJ_Q-cz^2J_P\},\\
 (B(0)+J)^{-1}\mathcal F_{>}(z)
 &=I+(B(0)+J)^{-1}
 \{B(z)-B(0)-(\Lambda+c+1)J\}.
 \end{split}
\]
Both pencils are Fredholm of index zero. By the analytic Fredholm theorem
\cite[Chapter~VII]{Kato}, each singular set is discrete in $\mathcal S$
provided the pencil has one invertible point.
By \eqref{eq:lambda-infinity}, choose $R>1$ with
\[
 \lambda_1(R)-\Lambda>c.
\]
The variational principle gives $\ker\mathcal F_{>}(R)=\{0\}$, hence
$\mathcal F_{>}(R)$ is invertible.
For $R>1$, $\mathcal F_{<}(R)$ is the analytic continuation of the
low-frequency pencil; the actual multiplier in the stability quotient
is $h(R)=1$.
Let $A_R$ be associated with $\mathfrak d_R$ and set $T_R=A_R-cQ$.
Its form on $Y$ is
$\mathfrak t_R[u,v]=\mathfrak d_R(u,v)-c\langle Qu,v\rangle_H$, and
\[
 \mathfrak t_R[v]\geq
 (\lambda_1(R)-\Lambda-c)\|v\|_H^2.
\]
Thus $T_R$ is positive and invertible. Since $\mathcal F_{<}(R)$ is the
form map of $T_R-cR^2P$, the rank-one resolvent identity gives
\[
 \ker\mathcal F_{<}(R)\neq\{0\}
 \quad\Longleftrightarrow\quad
 cR^2\langle\phi_{s,\tau},T_R^{-1}\phi_{s,\tau}\rangle_H=1.
\]
Indeed, a nonzero kernel vector $v$ satisfies
\[
 v=cR^2\langle v,\phi_{s,\tau}\rangle_H
 T_R^{-1}\phi_{s,\tau}.
\]
The coefficient $\langle v,\phi_{s,\tau}\rangle_H$ is nonzero. Taking
the scalar product with $\phi_{s,\tau}$ proves necessity; inserting
$T_R^{-1}\phi_{s,\tau}$
proves sufficiency. Since $\phi_{s,\tau}\in Y$, Cauchy--Schwarz gives
\[
 \begin{split}
 1=|\langle T_R^{-1/2}\phi_{s,\tau},
 T_R^{1/2}\phi_{s,\tau}\rangle_H|^2\leq
 \langle\phi_{s,\tau},T_R^{-1}\phi_{s,\tau}\rangle_H
 \mathfrak t_R[\phi_{s,\tau}].
 \end{split}
\]
Since $Q\phi_{s,\tau}=0$ and $\|\phi_{s,\tau}\|_H=1$,
\[
 \mathfrak t_R[\phi_{s,\tau}]
 =\mathfrak d_R[\phi_{s,\tau}]
 =\rho_{s,\tau}(R).
\]
By \eqref{eq:rho-positive-growth},
\[
 cR^2\langle\phi_{s,\tau},T_R^{-1}\phi_{s,\tau}\rangle_H
 \geq\frac{cR^2}{\rho_{s,\tau}(R)}
 \geq\frac{cR^2}{C(1+R^{2s})}\longrightarrow\infty.
\]
For larger $R$, the left-hand side exceeds one. The rank-one criterion
and index zero imply that $\mathcal F_{<}(R)$ is invertible.

Take $c=\Gamma^{\rm hp}_{s,\tau}$. A fiber minimizer at a point of
$Z_{s,\tau}$ belongs to $\ker\mathcal F_{<}(\xi)$ for $0<|\xi|<1$,
and to $\ker\mathcal F_{>}(\xi)$ for $|\xi|>1$.
By Step~2, $\mu_{s,\tau}(\xi)\longrightarrow\infty$ as
$|\xi|\to\infty$, so $Z_{s,\tau}$ lies in a compact real frequency
segment. Both singular sets have no accumulation point in
$\mathcal S$, and hence each meets this compact segment in only finitely
many points. Adding the junction points $\pm1$ therefore still gives a
finite set. Thus $Z_{s,\tau}$ is finite and has measure zero. The
necessary condition above proves nonattainment.

\smallskip
\noindent\emph{Step 5: the zero-frequency endpoint.}
It remains to prove \eqref{eq:intro-mu-zero}. Put
$g_0=\lambda_2(0)-\lambda_1(0)$. For small $\xi$, decompose
$\psi=P(\xi)\psi+(I-P(\xi))\psi=p+q$. The spectral theorem gives
\[
 \mathfrak d_\xi[\psi]\geq
 (\kappa_{s,\tau}\xi^2+O(\xi^4))\|p\|^2
 +(g_0+o(1))\|q\|^2.
\]
By \eqref{eq:projection-close},
\[
 \begin{split}
 &(1-C|\xi|)\bigl(\xi^2\|p\|^2+\|q\|^2\bigr)
 \leq\mathfrak n_\xi[p+q]\leq
 (1+C|\xi|)\bigl(\xi^2\|p\|^2+\|q\|^2\bigr).
 \end{split}
\]
Indeed, $Q_0p=O(\xi^2)\|p\|$ and
$P_0q=O(\xi^2)\|q\|$. The mixed term is bounded by
\[
 C\xi^2\|p\|\|q\|
 \leq C|\xi|(\xi^2\|p\|^2+\|q\|^2).
\]
The remaining errors are smaller. Hence
\[
 \liminf_{\xi\to0}\mu_{s,\tau}(\xi)
 \geq\min\{\kappa_{s,\tau},g_0\}.
\]
For the first eigenvector, \eqref{eq:projection-close} and
Proposition~\ref{prop:dispersion} give
\[
 \mathfrak n_\xi[\phi(\xi)]=\xi^2+O(\xi^4),
 \qquad
 \mathfrak d_\xi[\phi(\xi)]
 =\kappa_{s,\tau}\xi^2+O(\xi^4).
\]
Its quotient tends to $\kappa_{s,\tau}$. Let $e_2$ be a normalized second
eigenvector of the zero fiber. Then $P_0e_2=0$, so
$\mathfrak n_\xi[e_2]=1$. The form expansion gives
\[
 \mathfrak d_\xi[e_2]=g_0+O(\xi^2).
\]
Its quotient tends to $g_0$. This gives the second upper bound and proves
\eqref{eq:intro-mu-zero}.
Together with Step~2, this defines a continuous extension of
$\mu_{s,\tau}$ to $\mathbb R$ which is coercive at infinity. Its infimum
is therefore a minimum. Continuity also turns the essential infimum in
\eqref{eq:intro-gamma-fiber} into this same minimum.
\end{proof}

\begin{proposition}[Strict comparison and nonzero minimizing frequencies]
\label{prop:strict-stability-comparison}
For every $0<s<1$ and $\tau>0$,
\begin{equation}\label{eq:strict-stability-comparison}
 0<\Gamma^{\rm hp}_{s,\tau}
 \leq\lambda_1(1)-\Lambda_{s,\tau}
 \leq\rho_{s,\tau}(1)<\kappa_{s,\tau}.
\end{equation}
If $\lambda_2(0)-\lambda_1(0)\geq\kappa_{s,\tau}$, then
$Z_{s,\tau}$ is nonempty and contains a pair of opposite nonzero
frequencies. In particular, this conclusion holds for $\tau=2$ and
all $s<1$ sufficiently close to one.
\end{proposition}

\begin{proof}
Since $h(1)=1$, the variational formula gives
$\mu_{s,\tau}(1)=\lambda_1(1)-\Lambda_{s,\tau}$. Testing its Rayleigh
quotient with $\phi_{s,\tau}$ gives the second upper bound.
The strict scalar inequality $2(1-\cos q)<q^2$ for $q\neq0$, together
with positivity of $K_s$ and $\phi_{s,\tau}$, gives
\[
 \rho_{s,\tau}(1)
 <\int_0^\pi\int_0^\pi\int_{\mathbb R}
 q^2K_s(q,\theta,\varphi)\phi_{s,\tau}(\theta)
 \phi_{s,\tau}(\varphi)\,dq\,d\theta\,d\varphi
 =\kappa_{s,\tau}.
\]
The moment integral is finite. This proves
\eqref{eq:strict-stability-comparison}.

Under the gap assumption, \eqref{eq:intro-mu-zero} gives
$\mu_{s,\tau}(0)=\kappa_{s,\tau}>\Gamma^{\rm hp}_{s,\tau}$.
Continuity and coercivity of $\mu_{s,\tau}$ force a minimizer away from
zero. Evenness gives its opposite, and
Theorem~\ref{thm:optimal-hp} makes the set finite.
For $\tau=2$, Proposition~\ref{prop:local-benchmark} and
Theorem~\ref{thm:sharp-local-limit} give
\[
 \frac{2(1-s)}\pi\{\lambda_2(0)-\lambda_1(0)\}\longrightarrow2,
 \qquad
 \frac{2(1-s)}\pi\kappa_{s,2}\longrightarrow1.
\]
The gap assumption therefore holds for $s$ sufficiently close to one.
\end{proof}

\begin{remark}
The minimizing frequencies refer to the fixed distance with multiplier
$h(\xi)=\min\{\xi^2,1\}$. The strict inequality
\eqref{eq:strict-stability-comparison} shows that its global coefficient
is smaller than the quadratic dispersion coefficient. It does not locate
the minimizing frequencies explicitly.
\end{remark}

\begin{corollary}[Sharp axial order]
Let $0<\alpha<2$. There is no $c>0$ such that the left-hand side of
\eqref{eq:intro-mainB} controls
\[
 \int_{\mathbb R}\min\{|\xi|^\alpha,1\}
 \|P_0\widehat w(\xi,\cdot)\|_{H_{s,\tau}}^2\,d\xi
\]
for every test function. Thus the quadratic power in the axial term of
\eqref{eq:intro-distance} is optimal.
\end{corollary}

\begin{proof}
Take
\[
 \widehat w_\varepsilon(\xi,\theta)=
 \varepsilon^{-1/2}\eta(\xi/\varepsilon)
 \phi_{s,\tau}(\theta),
\]
where $\eta\in C_c^\infty((-1,1))$ and $\|\eta\|_2=1$.
Equation \eqref{eq:rho-expansion} gives
\[
 \mathfrak q_s[w_\varepsilon]
 -\Lambda_{s,\tau}\|w_\varepsilon\|_{L^2(\mathbb R;H_{s,\tau})}^2
 =\kappa_{s,\tau}\varepsilon^2
 \int \zeta^2|\eta(\zeta)|^2\,d\zeta+O(\varepsilon^4).
\]
A direct change of variables gives
\[
 \int\min\{|\xi|^\alpha,1\}
 \|P_0\widehat w_\varepsilon(\xi)\|^2\,d\xi
 =\varepsilon^\alpha\int|\zeta|^\alpha|\eta(\zeta)|^2\,d\zeta
\]
for small $\varepsilon$. Since $\alpha<2$, this is larger in order than
the preceding deficit. No positive uniform coefficient is possible. The
displayed functions belong to the completed form domain. For each
$\varepsilon$, core density gives $v_\varepsilon\in\mathscr C$ such that
\[
 \|v_\varepsilon-w_\varepsilon\|_{\mathscr X_s}=o(\varepsilon^2).
\]
The form Cauchy--Schwarz inequality and the uniform boundedness of
$\|w_\varepsilon\|_{\mathscr X_s}$ show that the deficit of
$v_\varepsilon$ differs from that of $w_\varepsilon$ by
$o(\varepsilon^2)$. The multiplier $\min\{|\xi|^\alpha,1\}$ is bounded by
one. Plancherel's theorem therefore shows that the corresponding
multiplier integrals differ by $o(\varepsilon^2)=o(\varepsilon^\alpha)$.
Their quotient still tends to zero. The inverse logarithmic polar
transform of $v_\varepsilon$ belongs to
$C_c^\infty(\mathbb R^2_+)$. This proves the assertion on the test core.
\end{proof}

\subsection*{Energy decomposition and extremizing sequences}

Fix $\delta>0$ such that the analytic projection $P(\xi)$ and the
normalized eigenfunction $\phi(\xi)$ are defined for $|\xi|<4\delta$.
Let $\chi_0\in C_c^\infty((-2\delta,2\delta))$ be even and satisfy
$0\leq\chi_0\leq1$ and $\chi_0=1$ on $[-\delta,\delta]$. For
$w\in\mathscr X_s$, define
\begin{equation}\label{eq:canonical-decomposition}
 \begin{split}
 c(\xi)&=\langle\widehat w(\xi,\cdot),\phi(\xi)\rangle_{H_{s,\tau}},
 \qquad |\xi|<2\delta,\\
 c(\xi)&=0,\qquad |\xi|\geq2\delta,\\
 \widehat a(\xi)&=\chi_0(\xi)c(\xi),\\
 \widehat r(\xi,\theta)&=\widehat w(\xi,\theta)
 -\widehat a(\xi)\phi_{s,\tau}(\theta).
 \end{split}
\end{equation}
The cutoff restricts the moving projection to a neighborhood of zero.

\begin{proposition}[Two-sided energy decomposition]
\label{lem:energy-decomposition}
There is $C\geq1$, depending on $s$, $\tau$, and the fixed cutoff, such
that for every $w\in\mathscr X_s$,
\begin{equation}\label{eq:energy-decomposition}
 \begin{split}
 C^{-1}\left(\|r\|_{\mathscr X_s}^2+
 \int_{\mathbb R}\rho_{s,\tau}(\xi)|\widehat a(\xi)|^2\,d\xi\right)
 &\leq\mathfrak q_s[w]-\Lambda_{s,\tau}\|w\|_{L^2(m_{s,\tau})}^2\\
 &\leq C\left(\|r\|_{\mathscr X_s}^2+
 \int_{\mathbb R}\rho_{s,\tau}(\xi)|\widehat a(\xi)|^2\,d\xi\right).
 \end{split}
\end{equation}
If $\mathfrak q_s[w]\leq C_0$ and $\|w\|_{L^2(m_{s,\tau})}=1$, then
$\|a\|_2+\|r\|_{\mathscr X_s}$ is bounded by a constant depending only
on $C_0$ and the fixed parameters.
\end{proposition}

\begin{proof}
Write
\[
 D_\xi[f]=\mathfrak b_{s,\xi}[f]
 -\Lambda_{s,\tau}\|f\|_{H_{s,\tau}}^2.
\]
We divide the proof into three frequency ranges.

If $|\xi|<\delta$, then
\[
 \widehat w=c(\xi)\phi(\xi)+z(\xi),\qquad
 z(\xi)=(I-P(\xi))\widehat w.
\]
The two terms are orthogonal for both the Hilbert norm and the fiber
operator. Hence
\begin{equation}\label{eq:exact-low-decomposition}
 D_\xi[\widehat w]
 =(\lambda_1(\xi)-\Lambda_{s,\tau})|c(\xi)|^2+D_\xi[z(\xi)].
\end{equation}
The spectral gap above $\lambda_1(\xi)$ gives
\[
 \mathfrak b_{s,\xi}[z]\leq C D_\xi[z].
\]
Equations \eqref{eq:projection-close} and
\eqref{eq:intro-dispersion} give
\[
 \|\phi(\xi)-\phi_{s,\tau}\|_{\mathscr Y_s}\leq C\xi^2,
 \qquad \xi^2|c(\xi)|^2\leq C D_\xi[\widehat w].
\]
Since
$\widehat r=z+c(\phi(\xi)-\phi_{s,\tau})$ in this region, the form
triangle inequality gives
\[
 \mathfrak b_{s,\xi}[\widehat r]^{1/2}
 \leq\mathfrak b_{s,\xi}[z]^{1/2}
 +|c(\xi)|\mathfrak b_{s,\xi}
 [\phi(\xi)-\phi_{s,\tau}]^{1/2}.
\]
Uniform graph equivalence and \eqref{eq:projection-close} imply
\[
 \mathfrak b_{s,\xi}[\phi(\xi)-\phi_{s,\tau}]
 \leq C\xi^4.
\]
Combining these estimates with $\rho_{s,\tau}(\xi)\leq C\xi^2$
and \eqref{eq:exact-low-decomposition}, we obtain
\begin{equation}\label{eq:low-r-control}
 \mathfrak b_{s,\xi}[\widehat r]
 +\rho_{s,\tau}(\xi)|\widehat a(\xi)|^2
 \leq C D_\xi[\widehat w].
\end{equation}

On $\delta\leq|\xi|\leq2\delta$, Lemma~\ref{lem:away-gap} gives
\[
 \|\widehat w\|_{H_{s,\tau}}^2\leq C D_\xi[\widehat w],
 \qquad
 \mathfrak b_{s,\xi}[\widehat w]\leq C D_\xi[\widehat w].
\]
By Cauchy--Schwarz and $\|\phi(\xi)\|_{H_{s,\tau}}=1$,
\[
 |\widehat a(\xi)|\leq\|\widehat w(\xi)\|_{H_{s,\tau}}.
\]
Moreover, $\mathfrak b_{s,\xi}[\phi_{s,\tau}]$ and
$\rho_{s,\tau}(\xi)$ are bounded on this compact interval. Thus
\[
 \begin{split}
 \mathfrak b_{s,\xi}[\widehat r]
 &\leq2\mathfrak b_{s,\xi}[\widehat w]
 +2|\widehat a|^2\mathfrak b_{s,\xi}[\phi_{s,\tau}]
 \leq C D_\xi[\widehat w],\\
 \rho_{s,\tau}(\xi)|\widehat a|^2
 &\leq C\|\widehat w\|_{H_{s,\tau}}^2
 \leq C D_\xi[\widehat w].
 \end{split}
\]
This proves \eqref{eq:low-r-control} in the middle range.

If $|\xi|>2\delta$, then $\widehat a=0$ and
$\widehat r=\widehat w$. Lemma~\ref{lem:away-gap} again gives
\[
 \mathfrak b_{s,\xi}[\widehat r]
 =D_\xi[\widehat w]+\Lambda_{s,\tau}
 \|\widehat w\|_{H_{s,\tau}}^2
 \leq C D_\xi[\widehat w].
\]
Integrate the three estimates and use
\eqref{eq:closed-direct-integral}. Uniform fiber coercivity adds
$\|r\|_{L^2(\mathbb R;H_{s,\tau})}^2$ to the left-hand side. This proves
the lower bound in \eqref{eq:energy-decomposition}.
For the upper bound, use the nonnegative quadratic form
\[
 E_{s,\tau}[v]=\mathfrak q_s[v]
 -\Lambda_{s,\tau}\|v\|_{L^2(m_{s,\tau})}^2.
\]
The canonical $\widehat a$ is supported in a fixed bounded frequency
interval. Hence $a\phi_{s,\tau}\in\mathscr X_s$, and the direct-integral
identity gives
\[
 E_{s,\tau}[a\phi_{s,\tau}]
 =\int_{\mathbb R}\rho_{s,\tau}(\xi)|\widehat a(\xi)|^2\,d\xi.
\]
Cauchy--Schwarz for $E_{s,\tau}$ yields
\[
 E_{s,\tau}[w]\leq2E_{s,\tau}[r]+2E_{s,\tau}[a\phi_{s,\tau}]
 \leq2\|r\|_{\mathscr X_s}^2+
 2\int_{\mathbb R}\rho_{s,\tau}(\xi)|\widehat a(\xi)|^2\,d\xi.
\]
This proves the upper bound. The boundedness assertion follows
from the lower bound, Plancherel's identity, and
$w=a\phi_{s,\tau}+r$.
\end{proof}

\begin{proof}[Proof of Theorem~\ref{thm:C}]
The two-sided estimate is Proposition~\ref{lem:energy-decomposition}.
Assume (1). The transformed sequence is normalized in
$L^2(m_{s,\tau})$, and its deficit tends to zero. Apply
Proposition~\ref{lem:energy-decomposition}. We obtain
\eqref{eq:intro-rn} and \eqref{eq:intro-an}. Since
$w_n=a_n\phi_{s,\tau}+r_n$ and $\|\phi_{s,\tau}\|_{H_{s,\tau}}=1$,
\[
 \left|\|a_n\|_2-1\right|
 \leq\|r_n\|_{L^2(\mathbb R;H_{s,\tau})}.
\]
Hence $\|a_n\|_2\to1$. This proves (2).

Conversely, assume (2). Since
$a_n\phi_{s,\tau}=w_n-r_n\in\mathscr X_s$, the direct-integral identity gives
\begin{equation}\label{eq:pure-channel}
 \mathfrak q_s[a_n\phi_{s,\tau}]
 -\Lambda_{s,\tau}\|a_n\|_2^2
 =\int_{\mathbb R}\rho_{s,\tau}(\xi)
 |\widehat a_n(\xi)|^2\,d\xi.
\end{equation}
By \eqref{eq:intro-an} and $\|a_n\|_2\to1$, the functions
$a_n\phi_{s,\tau}$ are bounded in $\mathscr X_s$. The form
Cauchy--Schwarz inequality gives
\[
 |\mathfrak q_s[w_n]-\mathfrak q_s[a_n\phi_{s,\tau}]|
 \leq\mathfrak q_s[r_n]
 +2\mathfrak q_s[a_n\phi_{s,\tau}]^{1/2}
 \mathfrak q_s[r_n]^{1/2}=o(1).
\]
Equation \eqref{eq:pure-channel} now proves (1).
\end{proof}

\begin{corollary}[No compact extremizing sequence]
Let $\{u_n\}$ be normalized and extremizing, and let $w_n$ be its
logarithmic polar transform. For every $R>0$,
\[
 \sup_{t_0\in\mathbb R}
 \int_{t_0-R}^{t_0+R}\int_0^\pi
 m_{s,\tau}(\theta)|w_n(t,\theta)|^2\,d\theta\,dt
 \longrightarrow0.
\]
Consequently, no choice of normalized dilations of $\{u_n\}$ admits a
strongly convergent subsequence in the remainder norm.
\end{corollary}

\begin{proof}
Use the decomposition $w_n=a_n\phi_{s,\tau}+r_n$ from
Theorem~\ref{thm:C}. We have $\|r_n\|_{L^2(\mathbb R;H_{s,\tau})}\to0$
and $\|a_n\|_2\to1$. For $\delta>0$, split
$\widehat a_n=\boldsymbol 1_{\{|\xi|\leq\delta\}}\widehat a_n+
\boldsymbol 1_{\{|\xi|>\delta\}}\widehat a_n$ and denote the inverse
transforms by $a_n^{\rm low}$ and $a_n^{\rm high}$. Cauchy--Schwarz in
frequency gives
\[
 \|a_n^{\rm low}\|_\infty^2\leq C\delta\|a_n\|_2^2.
\]
Since $\rho_{s,\tau}$ is continuous, positive away from zero, and
coercive at infinity,
\[
 \|a_n^{\rm high}\|_2^2
 \leq\left(\inf_{|\xi|\geq\delta}\rho_{s,\tau}(\xi)\right)^{-1}
 \int_{\mathbb R}\rho_{s,\tau}(\xi)|\widehat a_n(\xi)|^2\,d\xi
 \longrightarrow0.
\]
It follows that
\[
 \limsup_{n\to\infty}\sup_{t_0\in\mathbb R}
 \int_{t_0-R}^{t_0+R}|a_n(t)|^2\,dt
 \leq CR\delta.
\]
Let $\delta\downarrow0$ and use $\|\phi_{s,\tau}\|_{H_{s,\tau}}=1$
and $\|r_n\|_{L^2(\mathbb R;H_{s,\tau})}\to0$. This proves local
vanishing. Normalized dilations translate $t$. A nonzero strong limit
after such translations would have positive mass on a bounded interval,
which is impossible.
\end{proof}

\begin{remark}
Theorem~\ref{thm:C} fixes the limiting angular profile and controls the
error in the energy norm. The remaining loss of compactness lies in the
logarithmic radial amplitude.
\end{remark}

\section{Higher-dimensional splitting and extremizing sequences}
\label{sec:dimension}

We now set $p=2$ and retain the notation $x=(z,X)$ and $n=d-2$. Let
$\widehat u(\zeta,X)$ denote the unitary Fourier transform in $z$. For
$f\in C_c^\infty(\mathbb R^2_+)$, set
\begin{equation}\label{eq:transverse-fiber}
 \begin{split}
 \mathcal E_{d,s,\zeta}[f]
 =\int_{\mathbb R^2_+}\int_{\mathbb R^2_+}\int_{\mathbb R^n}
 \frac{|f(X)-e^{-i\zeta\cdot h}f(Y)|^2}
 {(|X-Y|^2+|h|^2)^{(d+2s)/2}}\,dh\,dX\,dY.
 \end{split}
\end{equation}

\begin{lemma}[Transverse Fourier identity]\label{lem:transverse-identity}
For every $\zeta\in\mathbb R^n$, there is a nonnegative quadratic form
$\mathcal T_{d,s,\zeta}$ such that
\begin{equation}\label{eq:transverse-splitting}
 \mathcal E_{d,s,\zeta}[f]
 =A_{d,s}[f]_{s,+}^2+\mathcal T_{d,s,\zeta}[f].
\end{equation}
It satisfies
\begin{equation}\label{eq:transverse-bound}
 0\leq\mathcal T_{d,s,\zeta}[f]
 \leq C_{d,s}|\zeta|^{2s}\|f\|_{L^2(\mathbb R^2_+)}^2.
\end{equation}
If $\zeta\neq0$, then
$\mathcal T_{d,s,\zeta}[f]=0$ only when $f=0$.
Moreover,
\begin{equation}\label{eq:dimension-form-splitting}
 \mathcal H_{d,s}[u]
 =\int_{\mathbb R^n}\left\{
 A_{d,s}\mathcal H_s[\widehat u(\zeta,\cdot)]
 +\mathcal T_{d,s,\zeta}[\widehat u(\zeta,\cdot)]
 \right\}\,d\zeta.
\end{equation}
\end{lemma}

\begin{proof}
Put $\nu=(d+2s)/2$. For $a>0$, define
\[
 \ell_a(\zeta)=\int_{\mathbb R^n}
 \frac{e^{-ih\cdot\zeta}}{(|h|^2+a^2)^\nu}\,dh.
\]
The Laplace representation and the Gaussian Fourier transform give
\begin{equation}\label{eq:ell-laplace}
 \ell_a(\zeta)=\frac{\pi^{n/2}}{\Gamma(\nu)}
 \int_0^\infty t^s
 \exp\left(-a^2t-\frac{|\zeta|^2}{4t}\right)\,dt.
\end{equation}
In particular,
\[
 \ell_a(0)=
 \frac{\pi^{n/2}\Gamma(1+s)}{\Gamma(\nu)}a^{-2-2s}
 =A_{d,s}a^{-2-2s}.
\]
Plancherel's theorem applied to the $z$ variable gives
\[
 \int_{\mathbb R^d_+}\int_{\mathbb R^d_+}
 \frac{|u(x)-u(y)|^2}{|x-y|^{d+2s}}\,dx\,dy
 =\int_{\mathbb R^n}
 \mathcal E_{d,s,\zeta}[\widehat u(\zeta,\cdot)]\,d\zeta.
\]
To justify this identity, first restrict to $|X-Y|>\varepsilon$. The
kernel in $h$ is then integrable, and Fubini and Plancherel apply.
Monotone convergence as $\varepsilon\downarrow0$ gives the full positive
energy.

Expansion of the square in \eqref{eq:transverse-fiber} yields
\[
 \begin{split}
 \mathcal E_{d,s,\zeta}[f]
 =\int_{\mathbb R^2_+}\int_{\mathbb R^2_+}
 \bigl\{&\ell_{|X-Y|}(0)(|f(X)|^2+|f(Y)|^2)-2\ell_{|X-Y|}(\zeta)
 \operatorname{Re}(f(X)\overline{f(Y)})\bigr\}\,dX\,dY.
 \end{split}
\]
With the same cutoff and
$\ell_{|X-Y|}(0)=A_{d,s}|X-Y|^{-2-2s}$, subtraction of the zero fiber gives
\begin{equation}\label{eq:T-kernel}
 \mathcal T_{d,s,\zeta}[f]
 =2\int_{\mathbb R^2_+}\int_{\mathbb R^2_+}
 L_\zeta(X-Y)\operatorname{Re}(f(X)\overline{f(Y)})\,dX\,dY,
\end{equation}
where $L_\zeta(Z)=\ell_{|Z|}(0)-\ell_{|Z|}(\zeta)$.

By \eqref{eq:ell-laplace}, $L_\zeta\geq0$. The bound
$1-\cos t\leq\min\{t^2/2,2\}$ gives
\begin{equation*}
 0\leq L_\zeta(Z)\leq C_{d,s}
 \min\{|\zeta|^2|Z|^{-2s},|Z|^{-2-2s}\}.
\end{equation*}
For this estimate, scale $h=|Z|r$ in the defining integral; both
comparison integrals are finite. Since $0<s<1$, we have
$L_\zeta\in L^1(\mathbb R^2)$. Thus \eqref{eq:T-kernel} is absolutely
convergent for $f\in L^2(\mathbb R^2_+)$ and is the limit of the cutoff
calculation.

Extend $f$ by zero to $\mathbb R^2$. The multiplier can be computed
directly from \eqref{eq:ell-laplace}. With the unitary Fourier transform
in two variables,
\[
 \mathcal F_{Z\to\kappa}(e^{-t|Z|^2})
 =\frac1{2t}e^{-|\kappa|^2/(4t)}.
\]
For $a\geq0$ and $b>0$, the change of variables $r=1/t$ and integration
by parts give
\[
 \begin{split}
 \int_0^\infty t^{s-1}(e^{-a/t}-e^{-(a+b)/t})\,dt
 &=\int_0^\infty r^{-s-1}(e^{-ar}-e^{-(a+b)r})\,dr\\
 &=\frac{\Gamma(1-s)}s\{(a+b)^s-a^s\}.
 \end{split}
\]
The boundary terms vanish: the difference is $O(r)$ at zero and is
bounded at infinity. The formula remains valid at $a=0$ because
$0<s<1$. The $L^1$ bound for $L_\zeta$ permits Fubini in the Gaussian
representation. Taking $a=|\kappa|^2/4$ and $b=|\zeta|^2/4$ gives
\begin{equation}\label{eq:L-symbol}
 \mathcal F_{Z\to\kappa}L_\zeta
 =c'_{d,s}\left((|\kappa|^2+|\zeta|^2)^s
 -|\kappa|^{2s}\right),\qquad c'_{d,s}>0.
\end{equation}
The cutoff calculation gives \eqref{eq:L-symbol} in distributions. Since
$L_\zeta\in L^1$, Plancherel in $X$ gives
\[
 \mathcal T_{d,s,\zeta}[f]
 =c''_{d,s}\int_{\mathbb R^2}
 \left((|\kappa|^2+|\zeta|^2)^s-|\kappa|^{2s}\right)
 |\mathcal F_X\widetilde f(\kappa)|^2\,d\kappa,
\]
where $\widetilde f$ is the zero extension. The multiplier is nonnegative.
It is strictly positive if $\zeta\neq0$. The concavity inequality
$(a+b)^s-a^s\leq b^s$ proves \eqref{eq:transverse-bound}. It also proves
the strict zero statement and \eqref{eq:transverse-splitting}.

By \eqref{eq:hardy-dimension-factor},
$\mathcal D_{d,s,2}=A_{d,s}\mathcal D_{2,s,2}$. The Hardy term also
splits by Plancherel. Subtracting it and integrating
\eqref{eq:transverse-splitting} in $\zeta$ proves
\eqref{eq:dimension-form-splitting}.
\end{proof}

\begin{proof}[Proof of Theorem~\ref{thm:dimension-reduction}]
Lemma~\ref{lem:lp-dimension} proves
\eqref{eq:intro-dimensional-reduction-p} for $d\geq3$.
For $d=2$ the identity is immediate. Fix $p>1$ and $\tau>0$.
Choose
\[
 0<\sigma<\min\{\tau,\tau_*(s,p)\},
 \qquad
 \tau_*(s,p)=
 \begin{cases}
 2\min\{sp,1\},&1<p<2,\\
 p(sp+1)/(2(p-1)),&p\geq2.
 \end{cases}
\]
The theorem of Kijaczko and Szczukiewicz gives
$\Lambda_{d,s,p,\sigma}>0$. Since
$\mathcal R_{d,s,p,\tau}\leq\mathcal R_{d,s,p,\sigma}$, one has
$\Lambda_{d,s,p,\tau}\geq\Lambda_{d,s,p,\sigma}>0$.
For $p=2$, the sharp identity follows from the same reduction. The planar
non-attainment is Theorem~\ref{thm:A}. Suppose $d\geq3$. Plancherel gives
\begin{equation}\label{eq:dimension-remainder-splitting}
 \mathcal R_{d,s,\tau}[u]
 =\int_{\mathbb R^{d-2}}
 \mathcal R_{s,\tau}[\widehat u(\zeta,\cdot)]\,d\zeta.
\end{equation}
We first show that the energy completion embeds into the weighted
remainder space. Put
$\rho(x)=X_2^{\tau-2s}|X|^{-\tau}$, $\omega(x)=x_d^{s-1/2}$, and
$d\nu=|x-y|^{-d-2s}\,dx\,dy$ on
$\mathbb R^d_+\times\mathbb R^d_+$. Define on the core
\[
 (Su)(x,y)=\sqrt{\omega(x)\omega(y)}
 \left(\frac{u(x)}{\omega(x)}-\frac{u(y)}{\omega(y)}\right).
\]
The $p=2$ identity in Lemma~\ref{lem:nonlinear-Hardy-ground-state},
applied to real and imaginary parts, gives
$\mathcal H_{d,s}[u]=\|Su\|_{L^2(d\nu)}^2$.
Suppose $u_j\to0$ in $L^2(\rho\,dx)$ and $Su_j\to F$ in
$L^2(d\nu)$. Since $\rho$ is positive and finite almost everywhere,
subsequences give $u_j\to0$ Lebesgue-almost everywhere and
$Su_j\to F$ $d\nu$-almost everywhere. The measure $d\nu$ is
$\sigma$-finite and absolutely continuous off the null diagonal, so
the displayed difference formula gives $F=0$. Thus $S$ is closable,
and the completion under
$(\mathcal H_{d,s}+\mathcal R_{d,s,\tau})^{1/2}$ embeds faithfully into
$L^2(\rho\,dx)$.

We now extend \eqref{eq:dimension-remainder-splitting} and
\eqref{eq:dimension-form-splitting} to this closed energy space.
Let $\{u_j\}$ be a Cauchy sequence in the core. Apply
\eqref{eq:dimension-form-splitting} and
\eqref{eq:dimension-remainder-splitting} to $u_j-u_k$. Since every term is
nonnegative,
\[
 \begin{split}
 &\min\{A_{d,s},1\}\int_{\mathbb R^{d-2}}
 \bigl(\mathcal H_s[\widehat u_j-\widehat u_k]
 +\mathcal R_{s,\tau}[\widehat u_j-\widehat u_k]\bigr)\,d\zeta\\
 &\quad+\int_{\mathbb R^{d-2}}
 \mathcal T_{d,s,\zeta}[\widehat u_j-\widehat u_k]\,d\zeta
 \leq
 \mathcal H_{d,s}[u_j-u_k]+\mathcal R_{d,s,\tau}[u_j-u_k].
 \end{split}
\]
The Fourier slices converge in the direct integral of the planar closed
energy space, and the transverse terms converge in their seminorm. The
partial Fourier transform is unitary on
\[
 L^2\left(\mathbb R^{d-2};
 L^2\left(\mathbb R^2_+,
 X_2^{\tau-2s}|X|^{-\tau}\,dX\right)\right),
\]
which is the remainder space. It identifies the slice limit with the
limit of $u_j$. Both splitting identities pass to the limit; the
transverse term is defined by seminorm completion.

Suppose that a nonzero completed energy function attains the sharp
coefficient. Subtract the sharp right-hand side from
\eqref{eq:dimension-form-splitting}. Equations
\eqref{eq:dimension-remainder-splitting} and Theorem~\ref{thm:A} give
\[
 \int_{\mathbb R^{d-2}}\left\{
 A_{d,s}\bigl(\mathcal H_s[\widehat u(\zeta)]
 -\Lambda_{s,\tau}\mathcal R_{s,\tau}[\widehat u(\zeta)]\bigr)
 +\mathcal T_{d,s,\zeta}[\widehat u(\zeta)]\right\}\,d\zeta=0.
\]
Both summands are nonnegative. By Theorem~\ref{thm:A}, the first can
vanish only for the zero planar slice. Hence
$\widehat u(\zeta,\cdot)=0$ almost everywhere. Plancherel in the
remainder space gives $u=0$, a contradiction.
\end{proof}

\begin{corollary}[Lifted stability]\label{cor:lifted-stability}
Let $d\geq3$. For every $u\in C_c^\infty(\mathbb R^d_+)$,
\[
 \begin{split}
 &\mathcal H_{d,s}[u]-\Lambda_{d,s,\tau}
 \mathcal R_{d,s,\tau}[u]\geq A_{d,s}\gamma_{s,\tau}
 \int_{\mathbb R^{d-2}}
 \operatorname{Dist}_{s,\tau}(\widehat u(\zeta,\cdot))^2\,d\zeta
 +\int_{\mathbb R^{d-2}}
 \mathcal T_{d,s,\zeta}[\widehat u(\zeta,\cdot)]\,d\zeta.
 \end{split}
\]
The estimate extends to the closed energy space, with the transverse
term defined by completion as in the proof of
Theorem~\ref{thm:dimension-reduction}.
\end{corollary}

\begin{proof}
Use \eqref{eq:intro-dimensional-reduction} and
\eqref{eq:dimension-form-splitting}, and integrate Theorem~\ref{thm:B}
over the planar Fourier slices. The distance seminorm is bounded by the
planar energy, so it is continuous under the direct-integral convergence
above. The transverse seminorm is continuous by its definition. This
also proves the estimate on the closed space.
\end{proof}

\begin{theorem}[Joint Fourier--Mellin energy decomposition]
\label{thm:joint-energy-decomposition}
Let $d\geq3$, $0<s<1$, and $\tau>0$. For a function in the closed
energy space, let $w_\zeta(t,\theta)$ be the planar logarithmic polar
transform of $\widehat u(\zeta,X)$. Apply the fixed decomposition
\eqref{eq:canonical-decomposition} to each slice:
\[
 w_\zeta(t,\theta)=a_\zeta(t)\phi_{s,\tau}(\theta)
 +r_\zeta(t,\theta).
\]
Define
\[
 \begin{split}
 \mathcal S_{d,s,\tau}[u]
 =\int_{\mathbb R^{d-2}}
 \left\{\|r_\zeta\|_{\mathscr X_s}^2+
 \int_{\mathbb R}\rho_{s,\tau}(\xi)
 |\widehat a_\zeta(\xi)|^2\,d\xi\right\}\,d\zeta+\int_{\mathbb R^{d-2}}
 \mathcal T_{d,s,\zeta}[\widehat u(\zeta,\cdot)]\,d\zeta.
 \end{split}
\]
There are $c,C>0$ such that
\begin{equation}\label{eq:joint-energy-equivalence}
 c\mathcal S_{d,s,\tau}[u]
 \leq\mathcal H_{d,s}[u]
 -\Lambda_{d,s,\tau}\mathcal R_{d,s,\tau}[u]
 \leq C\mathcal S_{d,s,\tau}[u].
\end{equation}
Consequently, a sequence with $\mathcal R_{d,s,\tau}[u_n]=1$ is
extremizing if and only if all three nonnegative terms defining
$\mathcal S_{d,s,\tau}[u_n]$ tend to zero. In that case,
\[
 \int_{\mathbb R^{d-2}}\|a_{n,\zeta}\|_{L^2(\mathbb R)}^2\,d\zeta
 \longrightarrow1.
\]
\end{theorem}

\begin{proof}
Subtract the sharp remainder in \eqref{eq:dimension-form-splitting}
and use \eqref{eq:dimension-remainder-splitting}. The result is
\[
 \begin{split}
 &\mathcal H_{d,s}[u]-\Lambda_{d,s,\tau}\mathcal R_{d,s,\tau}[u]\\
 &\quad=A_{d,s}\int_{\mathbb R^{d-2}}
 \left\{\mathfrak q_s[w_\zeta]
 -\Lambda_{s,\tau}\|w_\zeta\|_{L^2(m_{s,\tau})}^2\right\}\,d\zeta
 +\int_{\mathbb R^{d-2}}\mathcal T_{d,s,\zeta}
 [\widehat u(\zeta,\cdot)]\,d\zeta.
 \end{split}
\]
Integrate Proposition~\ref{lem:energy-decomposition} in $\zeta$.
The canonical maps $w\mapsto a$ and $w\mapsto r$ are bounded linear
maps on the relevant form spaces by its proof; their low-frequency
coefficients are measurable. They therefore preserve strong
measurability of the Fourier slices. The two-sided planar estimate and
the exact transverse splitting give \eqref{eq:joint-energy-equivalence},
first on the core and then by the completed identities.
Nonnegativity proves the sequence criterion. Finally, in the Hilbert
space $L^2(\mathbb R^{d-2}\times\mathbb R;H_{s,\tau})$,
\[
 \left|\left(\int\|a_{n,\zeta}\|_2^2\,d\zeta\right)^{1/2}-1\right|
 \leq\left(\int\|r_{n,\zeta}\|_{L^2(m_{s,\tau})}^2\,d\zeta\right)^{1/2}
 \longrightarrow0.
\]
This proves the last assertion.
\end{proof}

\begin{remark}[The higher-dimensional angular singularity]
\label{rem:dimension-scope}
For the original two-plane remainder, the spherical weight is
\[
 m_{d,s,\tau}(\omega)=\omega_d^{\tau-2s}
 (\omega_{d-1}^2+\omega_d^2)^{-\tau/2},
 \qquad \omega\in\mathbb S^{d-1}_+.
\]
Put $r=(\omega_{d-1}^2+\omega_d^2)^{1/2}$ and
$\omega_d=r\sin\vartheta$. Then
\begin{equation*}
 m_{d,s,\tau}(\omega)=r^{-2s}(\sin\vartheta)^{\tau-2s}.
\end{equation*}
The set
\[
 \Sigma=\{\omega\in\mathbb S^{d-1}:
 \omega_{d-1}=\omega_d=0\}\simeq\mathbb S^{d-3}
\]
has codimension two. In a local chart near $\Sigma$, write
$Z=(z,y)$, with $y\in\mathbb R^2_+$ normal and $z$ tangential to
$\Sigma$. The chart metric and density are uniformly comparable to the
Euclidean ones. Choose nonzero smooth functions $\chi$ and $\eta$ with
compact support and
\[
 \operatorname{supp}\chi\subset
 \{y:1<|y|<2,\ y_2>c|y|\},\qquad c>0.
\]
Set
\[
 \psi_\varepsilon(z,y)=\varepsilon^{s-1}
 \eta(z)\chi(y/\varepsilon).
\]
For $d=3$, the tangential factor and all tangential integrations below
are omitted. Direct scaling gives
\[
 \int m_{d,s,\tau}|\psi_\varepsilon|^2\,d\omega\asymp1,
 \qquad \|\psi_\varepsilon\|_2^2=O(\varepsilon^{2s}).
\]
Indeed, the first scaling factor is
$\varepsilon^{2s-2}\varepsilon^{2-2s}=1$. Writing the Fourier variables
as $(k,\xi)$, we have
\[
 \widehat\psi_\varepsilon(k,\xi)=
 \varepsilon^{s+1}\widehat\eta(k)\widehat\chi(\varepsilon\xi).
\]
The change of variables $\nu=\varepsilon\xi$ yields
\[
 \begin{split}
 \|\psi_\varepsilon\|_{\dot H^s}^2
 &=C\int
 (\varepsilon^2|k|^2+|\nu|^2)^s
 |\widehat\eta(k)|^2|\widehat\chi(\nu)|^2\,dk\,d\nu\leq C.
 \end{split}
\]
Here $\dot H^s$ is the homogeneous Gagliardo seminorm in the chart.
Extend the functions by zero outside it. Since
$\omega_d\asymp|y|\asymp\varepsilon$ on their support, the quantities
\[
 \|\psi_\varepsilon\|_{\dot H^s}^2
 +\int_{\mathbb S^{d-1}_+}\omega_d^{-2s}
 |\psi_\varepsilon|^2\,d\omega
 +\|\psi_\varepsilon\|_2^2
\]
are bounded. However, $\psi_\varepsilon\to0$ almost everywhere and their
$L^2(m_{d,s,\tau}d\omega)$ norms stay away from zero. Thus these local
regularity and Hardy bounds do not give a compact weighted embedding.
Moreover,
\[
 \frac{m_{d,s,\tau}}{m_{d,s,\tau_0}}
 =(\sin\vartheta)^{\tau-\tau_0}
\]
provides no small factor as $r\to0$, unlike the planar endpoint estimate
\eqref{eq:endpoint-tightness}. The planar compactness argument therefore
does not extend to this spherical weight. The transverse Fourier splitting
in Corollary~\ref{cor:lifted-stability} retains this additional channel.

Replacing the two-plane radius by $|x|$ gives the weaker weight
$\omega_d^{\tau-2s}$ and removes the factor $r^{-2s}$.
For the original remainder, Theorem~\ref{thm:joint-energy-decomposition}
gives the sequence criterion in joint Fourier--Mellin energy variables.
A physical-space profile decomposition would additionally have to resolve
translations along the edge and their scales relative to planar
dilations.
\end{remark}
\section{The local limit}
\label{sec:local-limit}

For $d\geq2$ and $p>1$, put
\[
 K_{d,p}=\frac1p\int_{\mathbb S^{d-1}}|\omega_d|^p\,d\omega
 =\frac{2\pi^{(d-1)/2}\Gamma((p+1)/2)}
 {p\Gamma((d+p)/2)}.
\]
Define
\[
 \begin{split}
 \mathcal H_{d,1,p}[u]
 & =\int_{\mathbb R^d_+}|\nabla u|^p\,dx
 -\left(\frac{p-1}{p}\right)^p
 \int_{\mathbb R^d_+}\frac{|u|^p}{x_d^p}\,dx,\\
 \mathcal R_{d,1,p,\tau}[u]
 & =\int_{\mathbb R^d_+}
 \frac{x_d^{\tau-p}|u(x)|^p}
 {(x_{d-1}^2+x_d^2)^{\tau/2}}\,dx.
 \end{split}
\]

\begin{proposition}[Limit on the smooth core]
\label{prop:local-core-limit}
Let $d\geq2$, $p>1$, and $\tau>0$. For every
$u\in C_c^\infty(\mathbb R^d_+)$,
\[
 \lim_{s\uparrow1}(1-s)\mathcal H_{d,s,p}[u]
 =K_{d,p}\mathcal H_{d,1,p}[u]
\]
and
\[
 \lim_{s\uparrow1}\mathcal R_{d,s,p,\tau}[u]
 =\mathcal R_{d,1,p,\tau}[u].
\]
Consequently,
\[
 \limsup_{s\uparrow1}
 \frac{(1-s)\Lambda_{d,s,p,\tau}}{K_{d,p}}
 \leq
 \inf_{u\in C_c^\infty(\mathbb R^d_+)\setminus\{0\}}
 \frac{\mathcal H_{d,1,p}[u]}
 {\mathcal R_{d,1,p,\tau}[u]}.
\]
\end{proposition}

\begin{proof}
The Bourgain--Brezis--Mironescu formula
\cite{BourgainBrezisMironescu} (see also
\cite[formula~(3)]{Ponce}) gives
\[
 \lim_{s\uparrow1}(1-s)[u]_{s,p,+}^p
 =K_{d,p}\int_{\mathbb R^d_+}|\nabla u|^p\,dx.
\]
The support of $u$ has positive distance from the boundary. Thus the
regional form has the same local contribution; its off-diagonal part
vanishes after multiplication by $1-s$.

For the Hardy constant, set in \eqref{eq:hardy-general-p}
\[
 b_s=\frac{sp-1}{p},\qquad \varepsilon_s=p(1-s).
\]
After $r=1-t$, its integral is
\[
 \int_0^1
 \left|\frac{1-(1-r)^{b_s}}r\right|^p
 r^{-1+\varepsilon_s}\,dr.
\]
The expression in absolute values is uniformly bounded for $s$ near
one and tends to $(p-1)/p$ as $(s,r)\to(1,0)$. The Abelian limit
\[
 \lim_{\varepsilon\downarrow0}
 \varepsilon\int_0^1F(r)r^{-1+\varepsilon}\,dr=F(0)
\]
follows by splitting at $\delta$, using continuity at zero on
$(0,\delta)$ and boundedness on $(\delta,1)$. For the present family, put
$F_s(r)=|[1-(1-r)^{b_s}]/r|^p$ and $F_s(0)=|b_s|^p$.
For every fixed $\delta>0$,
\[
 \begin{split}
 \left|\varepsilon_s\int_0^1F_s(r)r^{-1+\varepsilon_s}\,dr
 -F_s(0)\right|
 \leq{}&\sup_{0<r<\delta}|F_s(r)-F_s(0)|\\
 &+2\sup_{0<r<1}|F_s(r)|(1-\delta^{\varepsilon_s}).
 \end{split}
\]
The second term tends to zero as $s\uparrow1$. Taylor expansion at
$r=0$, uniformly for $b_s$ in a compact subset of $(0,\infty)$, makes
the first term tend to zero as $\delta\downarrow0$. Since
$\varepsilon_s=p(1-s)$, this gives
\[
 \lim_{s\uparrow1}(1-s)
 \int_0^1\frac{|1-t^{(sp-1)/p}|^p}{(1-t)^{1+sp}}\,dt
 =\frac1p\left(\frac{p-1}{p}\right)^p.
\]
Also,
\[
 2\pi^{(d-1)/2}\frac{\Gamma((1+p)/2)}
 {\Gamma((d+p)/2)}
 =\int_{\mathbb S^{d-1}}|\omega_d|^p\,d\omega.
\]
Thus
\[
 \lim_{s\uparrow1}(1-s)\mathcal D_{d,s,p}
 =K_{d,p}\left(\frac{p-1}{p}\right)^p.
\]
Dominated convergence on $\operatorname{supp}u$ proves both weighted
limits. Test the optimal quotient with a fixed $u$ and then take the
infimum to obtain the upper bound.
\end{proof}

For $p=2$, we also prove convergence of the sharp constants. We first
take $d=2$, so $K_{2,2}=\pi/2$. Set
\[
 \begin{split}
 \mathcal H_1[u]
 &=\int_{\mathbb R^2_+}|\nabla u|^2\,dx
 -\frac14\int_{\mathbb R^2_+}\frac{|u|^2}{x_2^2}\,dx,\\
 \mathcal R_{1,\tau}[u]
 &=\int_{\mathbb R^2_+}
 \frac{x_2^{\tau-2}|u(x)|^2}{|x|^\tau}\,dx,
 \end{split}
\]
and let $\Lambda_{1,\tau}$ be the infimum of
$\mathcal H_1[u]/\mathcal R_{1,\tau}[u]$ over nonzero functions in
$C_c^\infty(\mathbb R^2_+)$. Put
\[
 H_{1,\tau}=L^2((0,\pi),(\sin\theta)^{\tau-1}\,d\theta)
\]
and, on $C_c^\infty(0,\pi)$, define
\[
 \mathfrak l_{1,\xi}[y]
 =\int_0^\pi\sin\theta
 \left\{|y'|^2+\left(\xi^2+\frac14\right)|y|^2\right\}\,d\theta.
\]
Let $\mathscr V_\tau$ be the closure of the angular core under
$(\mathfrak l_{1,0}[y]+\|y\|_{H_{1,\tau}}^2)^{1/2}$.

\begin{lemma}[The local angular form]
\label{lem:local-angular-form}
For every $\tau>0$, the embedding
$\mathscr V_\tau\hookrightarrow H_{1,\tau}$ is compact. Moreover,
\[
 \Lambda_{1,\tau}=
 \min_{y\in\mathscr V_\tau\setminus\{0\}}
 \frac{\mathfrak l_{1,0}[y]}{\|y\|_{H_{1,\tau}}^2}.
\]
The minimum is positive and is attained by a strictly positive function,
unique up to a constant factor.
\end{lemma}

\begin{proof}
Write $\psi(\theta)=(\sin\theta)^{1/2}y(\theta)$ for an angular core
function. Direct differentiation and one integration by parts give
\[
 \int_0^\pi\left(|\psi'|^2-
 \frac{|\psi|^2}{4\sin^2\theta}\right)d\theta
 =\int_0^\pi\sin\theta
 \left(|y'|^2+\frac14|y|^2\right)d\theta.
\]
Also,
\[
 \int_0^\pi(\sin\theta)^{\tau-2}|\psi|^2\,d\theta
 =\|y\|_{H_{1,\tau}}^2.
\]

For compactness, use
\[
 t=\log\tan\frac\theta2,\qquad
 \sin\theta=\operatorname{sech}t,\qquad
 d\theta=\operatorname{sech}t\,dt.
\]
If $Y(t)=y(\theta(t))$, then
\[
 \int_0^\pi\sin\theta|y'|^2\,d\theta=\int_{\mathbb R}|Y'|^2\,dt,
 \quad
 \int_0^\pi\sin\theta|y|^2\,d\theta
 =\int_{\mathbb R}\operatorname{sech}^2t|Y|^2\,dt,
\]
whereas
\[
 \|y\|_{H_{1,\tau}}^2
 =\int_{\mathbb R}\operatorname{sech}^{\tau}t|Y|^2\,dt.
\]
Identifying $y$ with $Y$ under this change of variables, the form domain is
\[
 \mathscr V_\tau=\left\{Y\in H^1_{\rm loc}(\mathbb R):
 Y'\in L^2(\mathbb R),\quad
 \operatorname{sech}t\,Y\in L^2(\mathbb R),\quad
 \operatorname{sech}^{\tau/2}t\,Y\in L^2(\mathbb R)\right\}.
\]
For the reverse inclusion, let $Y$ belong to the set on the right.
Since $Y'\in L^2(\mathbb R)$,
\[
 Y(t)=o(|t|^{1/2})\qquad (|t|\to\infty).
\]
Indeed, for $t>T>0$,
\[
 |Y(t)|\leq |Y(T)|+(t-T)^{1/2}
 \|Y'\|_{L^2(T,\infty)}.
\]
Divide by $t^{1/2}$ and let first $t\to\infty$, then $T\to\infty$.
The other end is identical. Choose
$\chi\in C_c^\infty(\mathbb R)$ satisfying $\chi=1$ on $[-1,1]$ and
$\operatorname {supp}\chi\subset[-2,2]$. Put
$Y_R(t)=\chi(t/R)Y(t)$. The two weighted $L^2$ errors tend to zero by
dominated convergence. The derivative error away from the cutoff annulus
also tends to zero. On that annulus,
\[
 \int_{R<|t|<2R}\frac{|Y(t)|^2}{R^2}\,dt=o(1).
\]
Thus $Y_R\to Y$ in the graph norm. Mollification proves density of
the transformed core and the domain description.

The trace inequality on $[-1,1]$ gives
\[
 |Y(0)|^2\leq C\left\{
 \int_{\mathbb R}|Y'|^2\,dt+
 \int_{\mathbb R}\operatorname{sech}^2t\,|Y|^2\,dt
 \right\}.
\]
The fundamental theorem of calculus and Cauchy--Schwarz give
\[
 |Y(t)|^2\leq C(1+|t|)
 \left\{\int_{\mathbb R}|Y'|^2dt+
 \int_{\mathbb R}\operatorname{sech}^2t|Y|^2dt\right\}.
\]
Integration against $\operatorname{sech}^{\tau}t$ also gives
\[
 \|y\|_{H_{1,\tau}}^2\leq C_\tau\mathfrak l_{1,0}[y].
\]
In particular,
\[
 \int_{|t|>T}\operatorname{sech}^{\tau}t|Y(t)|^2\,dt
 \leq C_\tau(1+T)e^{-\tau T}\mathfrak l_{1,0}[y].
\]
For a form-bounded sequence, Rellich compactness gives a subsequence
convergent in $L^2(-T,T)$ for each integer $T$. The displayed tail bound
makes this subsequence Cauchy in $L^2(\operatorname{sech}^{\tau}t\,dt)$.
This proves compactness in $H_{1,\tau}$.

For $x=e^t(\cos\theta,\sin\theta)$, put
$w(t,\theta)=u(e^t\cos\theta,e^t\sin\theta)$. The local polar identities
and Fourier transformation in $t$ give
\[
 \mathcal H_1[u]=\int_{\mathbb R}\mathfrak l_{1,\xi}
 [y(\xi,\cdot)]\,d\xi,
 \qquad
 \mathcal R_{1,\tau}[u]=\int_{\mathbb R}
 \|y(\xi,\cdot)\|_{H_{1,\tau}}^2\,d\xi,
\]
where $\widehat w(\xi,\theta)=(\sin\theta)^{1/2}y(\xi,\theta)$.
Since
\[
 \mathfrak l_{1,\xi}[y]=\mathfrak l_{1,0}[y]
 +\xi^2\int_0^\pi\sin\theta|y|^2\,d\theta,
\]
the quotient is bounded below by the zero-fiber minimum. A logarithmic
cutoff, as in Lemma~\ref{lem:log-cutoff}, gives equality.

Compactness gives a minimizer. Taking its modulus and applying the
one-dimensional maximum principle gives positivity. For simplicity,
let $Y_1,Y_2$ be two positive minimizers in the $t$ variable. They satisfy
\[
 -Y_j''+\frac14\operatorname{sech}^2t\,Y_j
 =\Lambda_{1,\tau}\operatorname{sech}^{\tau}t\,Y_j
 \quad\text{on }\mathbb R.
\]
Their Wronskian
\[
 W(t)=Y_1'(t)Y_2(t)-Y_1(t)Y_2'(t)
\]
is constant. The form bounds give $Y_j'\in L^2(\mathbb R)$ and
$|Y_j(t)|\leq C(1+|t|)^{1/2}$. There is a sequence $t_n\to+\infty$
such that
\[
 t_n\{|Y_1'(t_n)|^2+|Y_2'(t_n)|^2\}\longrightarrow0.
\]
Otherwise their squared derivatives would dominate $c/t$ at infinity,
contrary to integrability. Hence $W(t_n)\to0$, so $W=0$ and
$Y_1/Y_2$ is constant.
\end{proof}

\begin{proposition}[The local benchmark at $\tau=2$]
\label{prop:local-benchmark}
For $\tau=2$, the local angular eigenvalues are
\begin{equation}\label{eq:local-legendre-spectrum}
 \lambda_j^{(1)}(\xi)=(j-1)j+\frac14+\xi^2,
 \qquad j\geq1.
\end{equation}
The positive normalized zero-fiber ground state in $H_{1,2}$ is
$y_{1,2}=2^{-1/2}$. Hence
\[
 \Lambda_{1,2}=\frac14,\qquad
 \lambda_2^{(1)}(0)-\lambda_1^{(1)}(0)=2,
 \qquad\int_0^\pi\sin\theta\,y_{1,2}^2\,d\theta=1.
\]
For the local cylinder form and the same multiplier
$h(\xi)=\min\{\xi^2,1\}$, the optimal stability coefficient is $1$
and is attained by nonzero finite-energy functions.
\end{proposition}

\begin{proof}
Put $z=\cos\theta$ and $y(\theta)=Y(z)$. Then
\[
 \|y\|_{H_{1,2}}^2=\int_{-1}^1|Y|^2\,dz,
 \qquad
 \mathfrak l_{1,\xi}[y]=\int_{-1}^1
 \left\{(1-z^2)|Y'|^2+
 \left(\frac14+\xi^2\right)|Y|^2\right\}\,dz.
\]
The Legendre polynomials belong to the closed form domain: under
$t=\log\tan(\theta/2)$ they are bounded with square-integrable
derivative, so the cutoff argument in
Lemma~\ref{lem:local-angular-form} applies. Their differential equation
and integration by parts show that their eigenvalues are
$n(n+1)+1/4+\xi^2$, $n\geq0$. Completeness of the Legendre basis in
$L^2(-1,1)$ proves \eqref{eq:local-legendre-spectrum}.

To specify the local stability statement, write $\widehat y(\xi,\theta)$
in this orthonormal basis of $H_{1,2}$ and let $c_n(\xi)$ be its
coefficients. The local sharp deficit and distance are respectively
\[
 \int_{\mathbb R}\sum_{n\geq0}
 \{n(n+1)+\xi^2\}|c_n(\xi)|^2\,d\xi
\]
and
\[
 \int_{\mathbb R}\left\{
 h(\xi)|c_0(\xi)|^2+\sum_{n\geq1}|c_n(\xi)|^2\right\}\,d\xi.
\]
The first is at least the second. Equality holds precisely when
$c_n=0$ for $n\geq1$ and $c_0$ vanishes outside $[-1,1]$, up to null
sets. Any nonzero $c_0\in L^2(-1,1)$ defines such a finite-energy
function. Thus the coefficient $1$ is sharp and attained.
\end{proof}

For $s<1$, use the ground-state conjugation
\[
 \psi(\theta)=(\sin\theta)^{s-1/2}y(\theta).
\]
It gives the exact identity
\[
 \|\psi\|_{H_{s,\tau}}^2=\|y\|_{H_{1,\tau}}^2.
\]
Define
\[
 \mathfrak a_{s,\xi}[y]=\frac{2(1-s)}\pi
 \mathfrak b_{s,\xi}[(\sin\theta)^{s-1/2}y].
\]
We use its closure in $H_{1,\tau}$, with domain
\[
 \mathcal D(\mathfrak a_{s,\xi})
 =\{y\in H_{1,\tau}:
 (\sin\theta)^{s-1/2}y\in\mathscr Y_s\}.
\]
The conjugation is unitary from $H_{1,\tau}$ onto $H_{s,\tau}$.
Thus the form is closed with compact resolvent. By the positive fiber
representation,
\[
 \begin{split}
 \mathfrak a_{s,\xi}[y]
 ={}&\frac{2(1-s)}\pi
 \int_0^\pi\int_0^\pi\int_{\mathbb R}
 K_s(q,\theta,\varphi)(\sin\theta\sin\varphi)^{s-1/2}\\
 &\quad\times|e^{-q/4}y(\theta)
 -e^{q/4-i\xi q}y(\varphi)|^2
 \,dq\,d\theta\,d\varphi.
 \end{split}
\]

\begin{lemma}[Uniform endpoint estimate]\label{lem:uniform-endpoint-limit}
Fix $s_0\in(1/2,1)$.
For every $\sigma>0$, there are $C_{s_0,\sigma,\tau}>0$ and
$\delta_0\in(0,1/8)$, depending only on $s_0$, $\sigma$, and $\tau$,
such that, for $s_0<s<1$, $0<\delta<\delta_0$, and
$y\in\mathcal D(\mathfrak a_{s,0})$,
\begin{equation}\label{eq:uniform-endpoint-local}
 \begin{split}
 \int_{(0,\delta)\cup(\pi-\delta,\pi)}
 (\sin\theta)^{\sigma-1}|y(\theta)|^2\,d\theta\leq
 C_{s_0,\sigma,\tau}\delta^{\sigma/2}
 \{\mathfrak a_{s,0}[y]+\|y\|_{H_{1,\tau}}^2\}.
 \end{split}
\end{equation}
The constants are independent of $s$, $\delta$, and $y$.
Here $\sigma$ may differ from the base-space exponent $\tau$.

\end{lemma}

\begin{proof}
It is enough to work near zero. Set
\[
 I_k=(2^{-k-1},2^{-k}),\qquad k\geq k_0.
\]
Let $E_{s,k}[y]$ be the contribution to
$\mathfrak a_{s,0}[y]$ from
\[
 \theta,\varphi\in I_k\cup I_{k+1},
 \qquad |q|<2^{-k-3}.
\]
These regions have bounded overlap. Hence
\[
 \sum_{k\geq k_0}E_{s,k}[y]\leq C\mathfrak a_{s,0}[y].
\]
On the region defining $E_{s,k}$,
\[
 K_s(q,\theta,\varphi)
 \asymp(q^2+|\theta-\varphi|^2)^{-1-s}
\]
and
\[
 (\sin\theta\sin\varphi)^{s-1/2}
 \asymp2^{-k(2s-1)}.
\]
Make the changes
\[
 q=2^{-k}Q,\qquad
 \theta=2^{-k}\Theta,\qquad
 \varphi=2^{-k}\Phi.
\]
All dyadic powers cancel. Put $J=(1/4,1)$ and
$f_J=|J|^{-1}\int_Jf$. The needed estimate is
\[
 \frac1{|J|}\int_J|f-f_J|^2\,d\Theta
 \leq C(1-s)\int_J\int_J\int_{|Q|<1/8}
 \frac{|f(\Theta)-f(\Phi)|^2}
 {(Q^2+|\Theta-\Phi|^2)^{1+s}}\,dQ\,d\Theta\,d\Phi,
\]
with $C$ independent of $s\in(s_0,1)$. In fact, if
$|\Theta-\Phi|<1/16$, integration over
$|Q|<|\Theta-\Phi|$ gives
\[
 \int_{|Q|<1/8}
 \frac{dQ}{(Q^2+|\Theta-\Phi|^2)^{1+s}}
 \geq c|\Theta-\Phi|^{-1-2s}.
\]
The constant $c$ is uniform. If the Poincar\'e bound fails, choose
$s_n\in(s_0,1)$ and $f_n$ with $(f_n)_J=0$, $\|f_n\|_2=1$, and
normalized truncated seminorm tending to zero. For $s_n\to\bar s<1$, one has $1-s_n\geq c>0$ for large $n$.
Since $s_n>s_0$, the truncated $H^{s_0}$ seminorm is bounded by the
truncated $H^{s_n}$ seminorm up to a fixed constant. Distances larger
than $1/16$ contribute at most $C\|f_n\|_2^2$. Rellich compactness gives
$f_n\to f$ in $L^2(J)$ and, along a subsequence, almost everywhere.
Fatou's lemma gives
\[
 \int_{\substack{\Theta,\Phi\in J\\|\Theta-\Phi|<1/16}}
 \frac{|f(\Theta)-f(\Phi)|^2}
 {|\Theta-\Phi|^{1+2\bar s}}\,d\Theta\,d\Phi=0.
\]
Thus $f$ is constant on each interval of length less than $1/16$;
overlapping such intervals shows that it is constant on $J$.
Strong convergence gives $f_J=0$ and $\|f\|_2=1$, a contradiction.

For the second case, put $r_*=1/16$ and
\[
 \rho_s(h)=\frac{1-s}{r_*^{2-2s}}|h|^{1-2s}
 \boldsymbol 1_{\{|h|<r_*\}}.
\]
Direct integration gives
\[
 \int_{\mathbb R}\rho_s(h)\,dh=1,\qquad
 \int_{|h|>\eta}\rho_s(h)\,dh
 =1-\left(\frac{\eta}{r_*}\right)^{2-2s}\longrightarrow0
 \quad(0<\eta<r_*).
\]
Since $s>s_0>1/2$, it is nonincreasing in $|h|$. Hence
\[
 \inf_{1/2\leq t\leq1}\rho_s(th)=\rho_s(h),
 \qquad
 \int_{\mathbb R}\inf_{1/2\leq t\leq1}\rho_s(th)\,dh=1.
\]
The additional one-dimensional condition in
\cite[Theorem~1.3]{Ponce} is thus satisfied. The normalized
truncated seminorm above bounds
\[
 \int_J\int_J
 \frac{|f_n(\Theta)-f_n(\Phi)|^2}{|\Theta-\Phi|^2}
 \rho_{s_n}(\Theta-\Phi)\,d\Theta\,d\Phi,
\]
because $r_*^{2-2s_n}$ stays bounded away from zero.
After rescaling $J$, that theorem gives $f_n\to f$ strongly in
$L^2(J)$ and $f'=0$, again contradicting the normalization.

Let
\[
 a_k=\frac1{|I_k|}\int_{I_k}y(\theta)\,d\theta,
 \qquad
 v_k=\frac1{|I_k|}\int_{I_k}|y-a_k|^2\,d\theta.
\]
For the drift factors, use
\[
 \begin{split}
 |y(\theta)-y(\varphi)|^2
 \leq{}&2|e^{-q/4}y(\theta)-e^{q/4}y(\varphi)|^2\\
 &+Cq^2(|y(\theta)|^2+|y(\varphi)|^2).
 \end{split}
\]
On the fixed interval $J$, the normalized integral of the drift kernel
is uniformly bounded:
\[
 \sup_{s_0<s<1}(1-s)
 \int_J\int_J\int_{|Q|<1/8}
 \frac{Q^2}{(Q^2+|\Theta-\Phi|^2)^{1+s}}
 \,dQ\,d\Theta\,d\Phi<\infty.
\]
Its integral operator is also uniformly bounded on $L^2(J)$, by
Schur's test. Rescaling therefore bounds the drift term by
$C4^{-k}(v_k+v_{k+1}+|a_k|^2+|a_{k+1}|^2)$.
Let $b_k$ be the average of $y$ on $I_k\cup I_{k+1}$.
The variance identity gives
\[
 \begin{split}
 \int_{I_k\cup I_{k+1}}|y-b_k|^2
 ={}&|I_k|v_k+|I_{k+1}|v_{k+1}\\
 &+\frac{|I_k||I_{k+1}|}{|I_k|+|I_{k+1}|}
 |a_k-a_{k+1}|^2.
 \end{split}
\]
Divide by $|I_k|+|I_{k+1}|$. The two length ratios are $2/3$ and $1/3$.
The rescaled Poincar\'e estimate therefore controls
$v_k+v_{k+1}+|a_{k+1}-a_k|^2$.
After increasing $k_0$,
the two variance terms are absorbed into the left-hand side, and we obtain
\[
 v_k+|a_{k+1}-a_k|^2
 \leq CE_{s,k}[y]
 +C4^{-k}(|a_k|^2+|a_{k+1}|^2).
\]
These constants are independent of $k_0$. For fixed $k_0$,
\[
 |a_{k_0}|^2\leq C_0(k_0,\tau)\|y\|_{H_{1,\tau}}^2,
\]
since the weight is bounded above and below on $I_{k_0}$.

Put
\[
 E_s[y]=\mathfrak a_{s,0}[y]+\|y\|_{H_{1,\tau}}^2.
\]
Telescoping the means and applying Cauchy--Schwarz give, for
$k_0\leq k\leq N$,
\[
 \frac{|a_k|^2}{1+k}
 \leq C_0(k_0,\tau)E_s[y]
 +C_1\sum_{j=k_0}^{k-1}
 \left\{E_{s,j}[y]+4^{-j}
 (|a_j|^2+|a_{j+1}|^2)\right\}.
\]
Here $C_1$ is independent of $s$, $k_0$, and $N$. Set
\[
 M_N=\max_{k_0\leq k\leq N}\frac{|a_k|^2}{1+k}.
\]
Using $\sum_{j\geq k_0}E_{s,j}[y]\leq CE_s[y]$, and increasing the
constants without introducing $k_0$ dependence in $C_1$, we obtain
\[
 M_N\leq C_0(k_0,\tau)E_s[y]
 +C_1M_N\sum_{j\geq k_0}4^{-j}(1+j).
\]
Choose $k_0$ after fixing $C_1$, so large that
\[
 C_1\sum_{j\geq k_0}4^{-j}(1+j)\leq\frac12.
\]
Then
\[
 M_N\leq2C_0(k_0,\tau)E_s[y]
\]
uniformly in $N$. Consequently,
\[
 |a_k|^2+v_k\leq C(1+k)E_s[y],
 \qquad k\geq k_0.
\]

Let $2^{-K-1}<\delta\leq2^{-K}$. For every $\sigma>0$,
\[
 \begin{split}
 \int_0^\delta\theta^{\sigma-1}|y|^2\,d\theta
 \leq C\sum_{k\geq K}2^{-k\sigma}(|a_k|^2+v_k) \leq CE_s[y]\sum_{k\geq K}2^{-k\sigma}(1+k) \leq C_{s_0,\sigma,\tau}\delta^{\sigma/2}E_s[y].
 \end{split}
\]
For the other endpoint put $\widetilde y(\theta)=y(\pi-\theta)$.
The sine weight and $K_s(q,\theta,\varphi)$ are invariant under simultaneous
reflection, so the same estimate applies to $\widetilde y$.
For a fixed $s$, take core approximants $y_j\to y$ in the graph norm
of $\mathfrak a_{s,0}$. Applying the estimate to $y_j-y_l$ shows that
they are Cauchy in the endpoint weighted norm for every $\sigma>0$.
Their local $L^2$ limits identify that weighted limit with $y$.
Letting $j\to\infty$ proves the estimate on the closed domain.

\end{proof}

\begin{lemma}[Interior compactness and the covariant lower bound]
\label{lem:covariant-local-limit}
Let $s_n\uparrow1$ and let $\xi_n$ be bounded. If
$\mathfrak a_{s_n,\xi_n}[y_n]+\|y_n\|_{H_{1,\tau}}^2$ is bounded,
then $y_n$ has a subsequence converging in $L^2(J)$ for every
$J\Subset(0,\pi)$. If also $\xi_n\to\xi$ and
$y_n\to y$ in $L^2(J')$, where $J\Subset J'\Subset(0,\pi)$, then
the localized lower bound \eqref{eq:localized-covariant-bbm} holds.
\end{lemma}

\begin{proof}
For compactness, let $J\Subset J'\Subset(0,\pi)$ and $|\xi|\leq M$. For $|q|+|\theta-\varphi|$ small and
$\theta,\varphi\in J'$,
\[
 \begin{split}
 |y(\theta)-y(\varphi)|^2\leq{}&
 2|e^{-q/4}y(\theta)-e^{q/4-i\xi q}y(\varphi)|^2 +C_Mq^2(|y(\theta)|^2+|y(\varphi)|^2).
 \end{split}
\]
After multiplication by the normalized kernel and integration, the last
term is bounded by $C_{J',M}\|y\|_{L^2(J')}^2$. Indeed, write
$h=\varphi-\theta$ and fix
$0<R<\operatorname {dist}(J,\partial J')$. Then
\[
 (1-s)\int_{q^2+h^2<R^2}
 q^2(q^2+h^2)^{-1-s}\,dq\,dh\leq C.
\]
For $|h|<R/4$, restrict the $q$ integration to $|q|<|h|$.
The resulting kernel is bounded below by
$c(1-s)|h|^{-1-2s}$. The one-dimensional compactness theorem applies with radius $R/4$:
the normalized kernels have unit mass, concentrate at zero, and satisfy
the dilation condition. It gives compactness in $L^2(J)$. This proves interior compactness.

For the lower limit, let
$J\Subset J'\Subset(0,\pi)$, let $s_n\uparrow1$ and $\xi_n\to\xi$, and
suppose that $y_n\to y$ in $L^2(J')$. Then
\begin{equation}\label{eq:localized-covariant-bbm}
 \liminf_{n\to\infty}\mathfrak a_{s_n,\xi_n}[y_n]
 \geq \int_J\sin\theta\left\{|y'|^2+
 \left(\xi^2+\frac14\right)|y|^2\right\}\,d\theta.
\end{equation}
Here the right-hand side is $+\infty$ unless
$y\in H^1(J)$.

Write
$(q,h)=r\omega$, where
$\omega=(\omega_1,\omega_2)\in\mathbb S^1$. Set
\[
 d\nu_n(r)=2(1-s_n)r^{1-2s_n}
 \boldsymbol 1_{(0,R)}(r)\,dr
\]
and
\[
 \begin{split}
 A_n(\theta,r,\omega)={}&r^{2+2s_n}
 K_{s_n}(r\omega_1,\theta,\theta+r\omega_2)
 (\sin\theta\sin(\theta+r\omega_2))^{s_n-1/2},\\
 V_n(\theta,r,\omega)={}&\frac1r\left\{
 e^{-r\omega_1/4}y_n(\theta)
 -e^{r\omega_1/4-i\xi_n r\omega_1}
 y_n(\theta+r\omega_2)\right\}.
 \end{split}
\]
Restricting the positive integral defining
$\mathfrak a_{s_n,\xi_n}$ to
$\theta,\theta+r\omega_2\in J'$ and $0<r<R$ gives
\begin{equation}\label{eq:localized-covariant-energy}
 \mathfrak a_{s_n,\xi_n}[y_n]
 \geq \frac1\pi\int_0^R\int_{\mathbb S^1}\int_{J'}
 \boldsymbol 1_{J'}(\theta+r\omega_2)A_n|V_n|^2
 \,d\theta\,d\omega\,d\nu_n(r).
\end{equation}
The local kernel expansion gives, uniformly on compact angular
subintervals,
\[
 r^{2+2s_n}K_{s_n}(r\omega_1,\theta,\theta+r\omega_2)
 =1+O(r^2).
\]
Moreover,
$\nu_n((0,R))=R^{2-2s_n}\to1$ and
$\nu_n((\eta,R))\to0$ for each $\eta>0$. Thus $\nu_n$ concentrates at
zero.

Let $F\in C_c^\infty(J\times\mathbb S^1)$ and extend it by zero in the
angular variable. From
$|V|^2\geq2\operatorname {Re}(V\overline F)-|F|^2$ and
\eqref{eq:localized-covariant-energy},
\begin{equation}\label{eq:covariant-duality}
 \begin{split}
 \liminf_{n\to\infty}\mathfrak a_{s_n,\xi_n}[y_n]
 \geq\frac1\pi\lim_{n\to\infty}
 \int_0^R\int_{\mathbb S^1}\int_{J'}A_n
 \{2\operatorname {Re}(V_n\overline F)-|F|^2\}
 \,d\theta\,d\omega\,d\nu_n(r).
 \end{split}
\end{equation}
On the support of $F$, the indicator is one for small $r$.
The part away from zero has vanishing $\nu_n$-mass.

The $|F|^2$ term converges to
\[
 \int_{\mathbb S^1}\int_J\sin\theta\,|F(\theta,\omega)|^2
 \,d\theta\,d\omega.
\]
For the linear term, change $\theta+r\omega_2$ to $\theta$ in its second
summand. The coefficient of $y_n(\theta)$ is
\[
 \begin{split}
 \frac1r\big[&e^{-r\omega_1/4}A_n(\theta,r,\omega)
 \overline{F(\theta,\omega)}-e^{r\omega_1/4-i\xi_n r\omega_1}
 A_n(\theta-r\omega_2,r,\omega)
 \overline{F(\theta-r\omega_2,\omega)}\big].
 \end{split}
\]
Taylor expansion of the smooth test factor, together with $s_n\to1$,
gives, uniformly on its support as $r\downarrow0$,
\[
 \omega_2\partial_\theta(\sin\theta\,\overline F)
 +\left(-\frac12+i\xi\right)\omega_1
 \sin\theta\,\overline F+o(1).
\]
To see the first coefficient explicitly, put $\beta_n=s_n-1/2$. Then
\[
 \begin{split}
 &(\sin\theta\sin(\theta+r\omega_2))^{\beta_n} =\sin^{2\beta_n}\theta+
 \beta_n r\omega_2\sin^{2\beta_n-1}\theta\cos\theta+O(r^2).
 \end{split}
\]
The same factor based at $\theta-r\omega_2$ has the opposite linear
term. The part $r\geq\eta$ is harmless because
$\nu_n((\eta,R))\to0$. Hence the strong $L^2(J')$ convergence of $y_n$
implies that the linear term in \eqref{eq:covariant-duality} converges,
in distributional notation, to
\[
 \int_{\mathbb S^1}\int_J\sin\theta
 \left\{-\omega_2y'+
 \left(-\frac12+i\xi\right)\omega_1y\right\}
 \overline F\,d\theta\,d\omega.
\]
To justify the regularity implicit in the duality step, write
\[
 L(\theta,\omega)=-\omega_2y'
 +\left(-\frac12+i\xi\right)\omega_1y
\]
in distributions. If the energy liminf is finite, optimizing the dual
inequality over scalar multiples of any test $F$ shows that $L$ defines
a bounded functional on $L^2(J\times\mathbb S^1,\sin\theta\,d\theta\,d\omega)$.
It is therefore represented by an $L^2$ function. Since
\[
 \int_{\mathbb S^1}\omega_2L(\theta,\omega)\,d\omega=-\pi y'
\]
and $\sin\theta$ is bounded below on $J$, one obtains $y'\in L^2(J)$.
Density of smooth tests and
$\sup_F(2\operatorname{Re}\langle L,F\rangle-\|F\|^2)=\|L\|^2$
now give
\[
 \begin{split}
 \liminf_{n\to\infty}\mathfrak a_{s_n,\xi_n}[y_n]
 \geq\frac1\pi\int_{\mathbb S^1}\int_J\sin\theta
 \left|-\omega_2y'+
 \left(-\frac12+i\xi\right)\omega_1y\right|^2
 \,d\theta\,d\omega.
 \end{split}
\]
Since
$\int_{\mathbb S^1}\omega_1\omega_2\,d\omega=0$ and
\[
 \int_{\mathbb S^1}\omega_1^2\,d\omega
 =\int_{\mathbb S^1}\omega_2^2\,d\omega=\pi,
\]
this is \eqref{eq:localized-covariant-bbm}.

\end{proof}

\begin{lemma}[Compact convergence of the angular fibers]
\label{lem:angular-local-mosco}
Fix $\tau>0$. Let $s_n\uparrow1$ and $\xi_n\to\xi$. The forms
$\mathfrak a_{s_n,\xi_n}$ converge to $\mathfrak l_{1,\xi}$ in the
compact Mosco sense on $H_{1,\tau}$. More precisely:
\begin{enumerate}
\item if
\[
 \sup_n\{\mathfrak a_{s_n,\xi_n}[y_n]
 +\|y_n\|_{H_{1,\tau}}^2\}<\infty,
\]
then a subsequence of $y_n$ converges strongly in $H_{1,\tau}$;
\item if $y_n\rightharpoonup y$ in $H_{1,\tau}$, then
\[
 \mathfrak l_{1,\xi}[y]\leq
 \liminf_{n\to\infty}\mathfrak a_{s_n,\xi_n}[y_n];
\]
\item for each $y\in\mathscr V_\tau$, there are $y_n\to y$ strongly in
$H_{1,\tau}$ such that
\[
 \limsup_{n\to\infty}\mathfrak a_{s_n,\xi_n}[y_n]
 \leq\mathfrak l_{1,\xi}[y].
\]
\end{enumerate}
\end{lemma}

\begin{proof}
Let
$y\in C_c^\infty(0,\pi)$ and put $\alpha=\varphi-\theta$. Uniformly on
compact angular subintervals,
\[
 K_s(q,\theta,\theta+\alpha)
 =(q^2+\alpha^2)^{-1-s}
 \bigl(1+O(q^2+\alpha^2)\bigr).
\]
Also,
$(\sin\theta\sin(\theta+\alpha))^{s-1/2}\to\sin\theta$.
Taylor's formula gives
\[
 \begin{split}
 e^{-q/4}y(\theta)-e^{q/4-i\xi q}y(\theta+\alpha)
 =-\alpha y'(\theta)+\left(-\frac12+i\xi\right)qy(\theta)
 +O(q^2+\alpha^2).
 \end{split}
\]
The mixed $q\alpha$ term is odd. Furthermore,
\[
 \lim_{s\uparrow1}\frac{2(1-s)}\pi
 \int_{q^2+\alpha^2<\delta^2}
 \frac{q^2}{(q^2+\alpha^2)^{1+s}}\,dq\,d\alpha=1.
\]
The same formula holds with $q^2$ replaced by $\alpha^2$. The normalized
integral outside a fixed neighborhood of the origin tends to zero. The
exponential $q$-tail is uniform for $s$ near one. Hence
\[
 \mathfrak a_{s,\xi}[y]\longrightarrow\mathfrak l_{1,\xi}[y]
\]
locally uniformly for bounded $\xi$.

For part~(1), the diamagnetic inequality and
Lemma~\ref{lem:modulus-domain} give
$\mathfrak a_{s_n,0}[|y_n|]\leq\mathfrak a_{s_n,\xi_n}[y_n]$.
Apply Lemma~\ref{lem:uniform-endpoint-limit} with $\sigma=\tau$.
The weighted endpoint mass tends to zero uniformly as $\delta\downarrow0$.
Lemma~\ref{lem:covariant-local-limit} gives compactness on every interior
interval. A diagonal argument proves strong compactness in $H_{1,\tau}$.

For part~(2), take a subsequence attaining the finite liminf; the
infinite case is immediate. Weak convergence bounds the norms, so
part~(1) gives strong convergence in $H_{1,\tau}$, hence locally in
$L^2(0,\pi)$, to the same limit $y$.
Apply Lemma~\ref{lem:covariant-local-limit} with
$J=(\delta,\pi-\delta)$ and let $\delta\downarrow0$. The finite bound
gives $y'\in L^2((0,\pi),\sin\theta\,d\theta)$ and
$y\in L^2((0,\pi),\sin\theta\,d\theta)$. The weak limit already belongs
to $H_{1,\tau}$. The closed-form domain description in the proof of
Lemma~\ref{lem:local-angular-form} therefore places $y$ in
$\mathscr V_\tau$. It also gives
\[
 \mathfrak l_{1,\xi}[y]\leq
 \liminf_{n\to\infty}\mathfrak a_{s_n,\xi_n}[y_n].
\]
To prove part~(3), approximate $y\in\mathscr V_\tau$ in the local graph
norm by $y^{(j)}\in C_c^\infty(0,\pi)$. For fixed $j$, the core
calculation gives
\[
 \mathfrak a_{s_n,\xi_n}[y^{(j)}]\longrightarrow
 \mathfrak l_{1,\xi}[y^{(j)}]
\]
as $n\to\infty$. Choose strictly increasing integers $N_j\geq j$ such that
\[
 |\mathfrak a_{s_n,\xi_n}[y^{(j)}]
 -\mathfrak l_{1,\xi}[y^{(j)}]|\leq1/j
 \quad\text{for all }n\geq N_j.
\]
Set $j(n)=\max\{j:N_j\leq n\}$ for $n\geq N_1$ and
$y_n=y^{(j(n))}$; define the finitely many earlier terms arbitrarily
in the core. Then $j(n)\to\infty$. Then
$y_n\to y$ strongly in $H_{1,\tau}$ and
\[
 \limsup_{n\to\infty}\mathfrak a_{s_n,\xi_n}[y_n]
 \leq\mathfrak l_{1,\xi}[y].
\]
This proves part~(3).
\end{proof}

\begin{lemma}[Convergence of the second-moment form]
\label{lem:second-moment-local-limit}
Fix $s_0\in(1/2,1)$. For $s_0<s<1$ and
$y,z\in C_c^\infty(0,\pi)$, define
\[
 \begin{split}
 \mathfrak t_s[y,z]
 ={}&\frac{2(1-s)}\pi
 \int_0^\pi\int_0^\pi\int_{\mathbb R}
 q^2K_s(q,\theta,\varphi)
 (\sin\theta\sin\varphi)^{s-1/2}
 y(\theta)\overline{z(\varphi)}
 \,dq\,d\theta\,d\varphi.
 \end{split}
\]
There is $C=C(s_0)>0$ such that
\[
 |\mathfrak t_s[y,z]|
 \leq C
 \left(\int_0^\pi\sin\theta\,|y|^2\,d\theta\right)^{1/2}
 \left(\int_0^\pi\sin\theta\,|z|^2\,d\theta\right)^{1/2}.
\]
Thus $\mathfrak t_s$ extends uniquely to a bounded form on
$L^2((0,\pi),\sin\theta\,d\theta)$.
Suppose that $s_n\uparrow1$, that $y_n\to y$ in
$L^2_{\rm loc}(0,\pi)$, and that
\[
 \sup_n\{\mathfrak a_{s_n,0}[y_n]
 +\|y_n\|_{H_{1,\tau}}^2\}<\infty.
\]
Then
\[
 \mathfrak t_{s_n}[y_n,y_n]
 \longrightarrow
 \int_0^\pi\sin\theta\,|y(\theta)|^2\,d\theta.
\]
\end{lemma}

\begin{proof}
Put
\[
 L_s(\alpha)=\frac{2(1-s)}\pi
 \int_{\mathbb R}q^2K_s(q,\alpha)\,dq.
\]
The local kernel comparison and the exponential tail give
\[
 0\leq L_s(\alpha)
 \leq C(1-s)(|\alpha|^{1-2s}+1),
 \qquad 0<|\alpha|<\pi,
\]
with a constant uniform for $s_0<s<1$. Indeed, for
$0<|\alpha|<1$, scale $q=|\alpha|r$ in the local part. The remaining
$q$-integral is bounded by a constant times $|\alpha|^{1-2s}$.
The tail is uniform. Consequently,
\[
 \sup_{s_0<s<1}\|L_s\|_{L^1(-\pi,\pi)}<\infty.
\]

Set
\[
 f(\theta)=(\sin\theta)^{1/2}y(\theta),\qquad
 g(\theta)=(\sin\theta)^{1/2}z(\theta).
\]
The kernel acting on unweighted $L^2(0,\pi)$ is
\[
 \mathcal L_s(\theta,\varphi)
 =L_s(\theta-\varphi)
 (\sin\theta\sin\varphi)^{s-1}.
\]
We apply the weighted Schur test with
$p_*(\theta)=(\sin\theta)^{-1/2}$. Let $0<\theta<\varepsilon/4$,
where $0<\varepsilon<1/8$ is fixed. If
$\theta/2<\varphi<2\theta$, both sine factors are comparable to
$\theta$. Since
\[
 (1-s)\theta^{2s-2}
 \int_{|h|<\theta}|h|^{1-2s}\,dh\leq C,
\]
the kernel estimate gives
\[
 \int_{\theta/2}^{2\theta}
 \mathcal L_s(\theta,\varphi)
 \frac{p_*(\varphi)}{p_*(\theta)}\,d\varphi\leq C.
\]
If $0<\varphi\leq\theta/2$, then
$|\theta-\varphi|\asymp\theta$, and the singular part is bounded by
\[
 C(1-s)\theta^{1/2-s}\varphi^{s-3/2}.
\]
Its integral is at most $C(1-s)/(s-1/2)$. If
$2\theta\leq\varphi$ and both variables are near zero, the singular
part is bounded by
\[
 C(1-s)\theta^{s-1/2}\varphi^{-s-1/2}.
\]
Its integral is bounded by $C(1-s)/(s-1/2)$. For the nonsingular
term in $L_s$, the entire Schur integral is bounded by
\[
 C(1-s)(\sin\theta)^{s-1/2}
 \int_0^\pi(\sin\varphi)^{s-3/2}\,d\varphi
 \leq C_{s_0}.
\]
Indeed, $s>s_0>1/2$ and $0<\sin\theta\leq1$. For
$\varepsilon/2\leq\varphi\leq\pi-\varepsilon/2$, the angles are
separated and $L_s(\theta-\varphi)\leq C_\varepsilon(1-s)$.
For $\pi-\varepsilon/2<\varphi<\pi$, they are also separated, and
\[
 \begin{split}
 &\int_{\pi-\varepsilon/2}^\pi
 \mathcal L_s(\theta,\varphi)
 \frac{p_*(\varphi)}{p_*(\theta)}\,d\varphi \leq C_\varepsilon(1-s)
 (\sin\theta)^{s-1/2}
 \int_{\pi-\varepsilon/2}^\pi
 (\sin\varphi)^{s-3/2}\,d\varphi\leq C_\varepsilon.
 \end{split}
\]
The range between $2\theta$ and $\varepsilon/2$ is covered by the
preceding near-zero estimate. Reflection treats $\theta$ near $\pi$.
For $\theta$ in a fixed middle interval, the near-diagonal term is
bounded by $\|L_s\|_1$. On its complement the kernel is at most
$C_\varepsilon(1-s)$; the endpoint integral is finite because
$s_0>1/2$. Thus
\[
 \sup_\theta\frac1{p_*(\theta)}
 \int_0^\pi\mathcal L_s(\theta,\varphi)
 p_*(\varphi)\,d\varphi\leq C.
\]
The kernel is symmetric. Thus the transposed Schur bound also holds.
Schur's test proves
\[
 |\mathfrak t_s[y,z]|
 \leq C\|y\|_{L^2(\sin\theta\,d\theta)}
 \|z\|_{L^2(\sin\theta\,d\theta)}.
\]

The endpoint estimate \eqref{eq:uniform-endpoint-local}, used with
$\sigma=2$, gives
\[
 \int_{(0,\delta)\cup(\pi-\delta,\pi)}
 \sin\theta\,|y_n|^2\,d\theta\leq C\delta.
\]
Together with local convergence, this proves
\[
 y_n\longrightarrow y
 \quad\text{in }L^2((0,\pi),\sin\theta\,d\theta).
\]
For a fixed $v\in C_c^\infty(0,\pi)$, the moment computation in the core
part of Lemma~\ref{lem:angular-local-mosco} gives
\[
 \mathfrak t_s[v,v]\longrightarrow
 \int_0^\pi\sin\theta\,|v|^2\,d\theta.
\]
Write $\|\cdot\|_*=\|\cdot\|_{L^2(\sin\theta\,d\theta)}$.
For a fixed core function $v$, the uniform Schur bound gives
\[
 \begin{split}
 |\mathfrak t_{s_n}[y_n,y_n]-\mathfrak t_{s_n}[v,v]|
 &\leq C\|y_n-v\|_*(\|y_n\|_*+\|v\|_*),\\
 |\|y\|_*^2-\|v\|_*^2|
 &\leq\|y-v\|_*(\|y\|_*+\|v\|_*).
 \end{split}
\]
First let $n\to\infty$, using $y_n\to y$ in this norm and the core
moment limit for $v$. Then choose $v$ with $\|y-v\|_*$ arbitrarily
small. The two bounds prove the asserted limit.
\end{proof}

\begin{proof}[Proof of Theorem~\ref{thm:sharp-local-limit}]
After the ground-state conjugation, the denominator of the zero-fiber
Rayleigh quotient is the fixed $H_{1,\tau}$ norm. Hence
\[
 \frac{2(1-s)}\pi\Lambda_{s,\tau}
 =\inf_{y\neq0}\frac{\mathfrak a_{s,0}[y]}
 {\|y\|_{H_{1,\tau}}^2}.
\]
A recovery sequence for the local minimizer gives
\[
 \limsup_{s\uparrow1}\frac{2(1-s)}\pi\Lambda_{s,\tau}
 \leq\Lambda_{1,\tau}.
\]
For the opposite inequality, choose a sequence $s_n\uparrow1$ such that
\[
 \frac{2(1-s_n)}\pi\Lambda_{s_n,\tau}
 \longrightarrow
 \liminf_{s\uparrow1}\frac{2(1-s)}\pi\Lambda_{s,\tau},
\]
and put $y_n=(\sin\theta)^{1/2-s_n}\phi_{s_n,\tau}$. The functions
$y_n$ are normalized in $H_{1,\tau}$. The preceding upper bound makes
$\mathfrak a_{s_n,0}[y_n]$ bounded. Compact Mosco convergence therefore
gives, after passage to a subsequence,
$y_n\to y$ strongly in $H_{1,\tau}$. Thus $\|y\|_{H_{1,\tau}}=1$, and
the Mosco lower bound gives
\[
 \liminf_{s\uparrow1}\frac{2(1-s)}\pi\Lambda_{s,\tau}
 =\lim_{n\to\infty}\mathfrak a_{s_n,0}[y_n]
 \geq\mathfrak l_{1,0}[y]\geq\Lambda_{1,\tau}.
\]
This proves convergence of the first eigenvalue. Finally, take an
arbitrary sequence $s_n\uparrow1$. Every subsequence of the corresponding
positive normalized $y_n$ has a further subsequence converging strongly to
a nonnegative normalized local minimizer. The simplicity of the local
first eigenvalue identifies that limit with $y_{1,\tau}$. Hence the whole
family converges strongly to $y_{1,\tau}$.

For the higher eigenvalues, we use the min--max principle. For
the upper bound, fix $\varepsilon>0$. Approximate an orthonormal basis of
the first $j$ local eigenspaces in the local graph norm by angular core
functions. Their span $E_\varepsilon$ has dimension $j$ and largest local
Rayleigh quotient at most $\lambda_j^{(1)}(\xi)+\varepsilon$. Core convergence is uniform on its unit sphere. The min--max principle,
followed by $\varepsilon\downarrow0$, gives the upper bound.

For the lower bound, write $\lambda_j^{(s)}(\xi)$ for the nonlocal
eigenvalue in this paragraph. Choose $s_n\uparrow1$ so that
\[
 L:=\lim_{n\to\infty}\frac{2(1-s_n)}\pi
 \lambda_j^{(s_n)}(\xi)
 =\liminf_{s\uparrow1}\frac{2(1-s)}\pi\lambda_j^{(s)}(\xi).
\]
For $1\leq k\leq j$, let $e_{k,n}$ be orthonormal eigenfunctions of
$\mathfrak a_{s_n,\xi}$ associated with its first $j$ eigenvalues. The
upper bound just proved makes their form values bounded. By compactness
and a diagonal extraction, $e_{k,n}\to e_k$ strongly in $H_{1,\tau}$
for every $k\leq j$; the limits are orthonormal. For every
$c=(c_1,\ldots,c_j)$ with rational real and imaginary parts, the Mosco
lower bound and the diagonal spectral representation give
\[
 \mathfrak l_{1,\xi}\!\left[\sum_{k=1}^jc_ke_k\right]
 \leq\liminf_{n\to\infty}
 \sum_{k=1}^j\frac{2(1-s_n)}\pi
 \lambda_k^{(s_n)}(\xi)|c_k|^2
 \leq L\sum_{k=1}^j|c_k|^2.
\]
By finite-dimensional density the inequality holds for all
$c\in\mathbb C^j$. The min--max principle gives
$\lambda_j^{(1)}(\xi)\leq L$.

For the last assertion, put
\[
 y_s=(\sin\theta)^{1/2-s}\phi_{s,\tau}.
\]
The preceding part of the proof gives
$y_s\to y_{1,\tau}$ locally in $L^2(0,\pi)$ and
\[
 \sup_{s\text{ near }1}
 \{\mathfrak a_{s,0}[y_s]+\|y_s\|_{H_{1,\tau}}^2\}<\infty.
\]
Moreover, the definition of $\kappa_{s,\tau}$ gives
\[
 \frac{2(1-s)}\pi\kappa_{s,\tau}
 =\mathfrak t_s[y_s,y_s].
\]
Lemma~\ref{lem:second-moment-local-limit} now yields
\[
 \lim_{s\uparrow1}\frac{2(1-s)}\pi\kappa_{s,\tau}
 =\int_0^\pi\sin\theta\,y_{1,\tau}^2\,d\theta.
\]
\end{proof}

\begin{corollary}[Limit of the sharp deficit]
\label{cor:sharp-deficit-local-limit}
For every $\tau>0$ and $u\in C_c^\infty(\mathbb R^2_+)$,
\[
 \begin{split}
 &\lim_{s\uparrow1}\frac{2(1-s)}\pi
 \{\mathcal H_s[u]-\Lambda_{s,\tau}\mathcal R_{s,\tau}[u]\} =\mathcal H_1[u]
 -\Lambda_{1,\tau}\mathcal R_{1,\tau}[u].
 \end{split}
\]
For every $d\geq2$,
\[
 \lim_{s\uparrow1}
 \frac{(1-s)\Lambda_{d,s,\tau}}{K_{d,2}}
 =\Lambda_{1,\tau}.
\]
\end{corollary}

\begin{proof}
The first assertion follows from Proposition~\ref{prop:local-core-limit}
and Theorem~\ref{thm:sharp-local-limit}. For the second one, use
$\Lambda_{d,s,\tau}=A_{d,s}\Lambda_{s,\tau}$ and
\[
 A_{d,s}\longrightarrow
 \frac{\pi^{(d-2)/2}}{\Gamma((d+2)/2)}
 =\frac{K_{d,2}}{K_{2,2}}.
\]
Since $K_{2,2}=\pi/2$, the conclusion follows.
\end{proof}

\begin{remark}
For $\tau=1$, Lemma~\ref{lem:local-angular-form} gives the
Maz'ya--Shaposhnikova eigenvalue
\[
 \Lambda_{1,1}=\min_{y\neq0}
 \frac{\displaystyle\int_0^\pi\sin\theta
 (|y'|^2+|y|^2/4)\,d\theta}
 {\displaystyle\int_0^\pi|y|^2\,d\theta}.
\]
For $\tau=2$, Proposition~\ref{prop:local-benchmark} gives the full
local spectrum and an attained local stability coefficient. In contrast,
Theorem~\ref{thm:optimal-hp} proves non-attainment for every $s<1$;
both statements use the same fixed multiplier $h$.

Convergence of $\Gamma^{\rm hp}_{s,\tau}$ would require both
large-frequency coercivity uniform in $s$ and uniform control of the
generalized bands at zero. The fixed-$s$ estimates above do not give
these bounds.
\end{remark}
\appendix

\section{The spectrally periodized cylinder}
\label{sec:periodic}

Fix $0<s<1$ and $\tau>0$. For $L>0$, let
$\mathcal C_L=(\mathbb R/L\mathbb Z)\times(0,\pi)$. Its form is defined
by the discrete fibers $\mathfrak b_{s,2\pi k/L}$. Let
$\mathcal M_L$ be the one-dimensional space of functions
$c\phi_{s,\tau}(\theta)$. The optimal stability constant for the
spectrally periodized cylinder is
\begin{equation}\label{eq:intro-periodic-gap}
 G_{s,\tau}(L)=\inf_{w\notin\mathcal M_L}
 \frac{\mathfrak q_{s,L}[w]-\Lambda_{s,\tau}
 \|w\|_{L^2(\mathcal C_L,m_{s,\tau})}^2}
 {\operatorname{dist}_{L^2(\mathcal C_L,m_{s,\tau})}
 (w,\mathcal M_L)^2}.
\end{equation}

Let $L>0$. A function on $\mathcal C_L$ has the expansion
\[
 w(t,\theta)=L^{-1/2}\sum_{k\in\mathbb Z}
 e^{2\pi ikt/L}\psi_k(\theta).
\]
Define
\begin{equation}\label{eq:periodic-form}
 \mathfrak q_{s,L}[w]=
 \sum_{k\in\mathbb Z}\mathfrak b_{s,2\pi k/L}[\psi_k]
\end{equation}
and
\[
 \|w\|_{L^2(\mathcal C_L,m_{s,\tau})}^2
 =\sum_{k\in\mathbb Z}\|\psi_k\|_{L^2(m_{s,\tau})}^2.
\]
The periodic form domain is
\[
 \mathscr X_{s,L}=
 \left\{w:\psi_k\in\mathscr Y_s,\quad
 \sum_{k\in\mathbb Z}\left(
 \mathfrak b_{s,2\pi k/L}[\psi_k]
 +\|\psi_k\|_{H_{s,\tau}}^2\right)<\infty\right\}.
\]
The graph norm makes this a Hilbert space. Finite Fourier sums with
coefficients in $C_c^\infty(0,\pi)$ form a core.

\begin{proposition}[Periodic extremal]
The sharp spectrally periodized quotient is
\[
 \inf_{w\neq0}
 \frac{\mathfrak q_{s,L}[w]}
 {\|w\|_{L^2(\mathcal C_L,m_{s,\tau})}^2}
 =\Lambda_{s,\tau}.
\]
It is attained precisely by
$w(t,\theta)=c\phi_{s,\tau}(\theta)$.
\end{proposition}

\begin{proof}
Each summand in \eqref{eq:periodic-form} is bounded below by
$\lambda_1(2\pi k/L)\|\psi_k\|^2$.
Lemma~\ref{lem:diamagnetic} gives
$\lambda_1(2\pi k/L)\geq\lambda_1(0)$. Equality is possible only when
$k=0$. In that fiber, Lemma~\ref{lem:positive-simple} shows that equality is
possible only for a multiple of $\phi_{s,\tau}$.
\end{proof}

\begin{theorem}[Attainment on the spectrally periodized cylinder]
\label{thm:D}
Let $0<s<1$, $\tau>0$, and $L>0$. The constant $G_{s,\tau}(L)$ is positive
and is attained.
If $\lambda_2(0)$ denotes the second eigenvalue of the zero fiber, then
\begin{equation}\label{eq:intro-mainD}
 G_{s,\tau}(L)=\min\left\{\lambda_2(0)-\lambda_1(0),
 \inf_{k\in\mathbb Z\setminus\{0\}}
 [\lambda_1(2\pi k/L)-\lambda_1(0)]\right\}.
\end{equation}
Moreover,
\begin{equation}\label{eq:intro-periodic-asymptotic}
 G_{s,\tau}(L)=\frac{4\pi^2\kappa_{s,\tau}}{L^2}+O(L^{-4})
 \quad\text{as }L\to\infty.
\end{equation}
For all sufficiently large $L$, the orthogonal component
$w-P_{\mathcal M_L}w$ of every extremizer is supported in the modes
$k=\pm1$. A complex representative orthogonal to $\mathcal M_L$ may use
one mode. A real representative has conjugate coefficients.
\end{theorem}

\begin{proof}[Proof of Theorem~\ref{thm:D}]
Write $w=P_{\mathcal M_L}w+w^\perp$. Since the zero-mode ground state
solves \eqref{eq:EL}, the numerator in
\eqref{eq:intro-periodic-gap} is unchanged when
$P_{\mathcal M_L}w$ is removed. The denominator is
$\|w^\perp\|^2$. Hence it is enough to take
$w\perp\mathcal M_L$. Then
$\psi_0\perp\phi_{s,\tau}$. The spectral theorem gives
\[
 \mathfrak b_{s,0}[\psi_0]-\Lambda_{s,\tau}\|\psi_0\|^2
 \geq(\lambda_2(0)-\lambda_1(0))\|\psi_0\|^2.
\]
For $k\neq0$,
\[
 \mathfrak b_{s,2\pi k/L}[\psi_k]
 -\Lambda_{s,\tau}\|\psi_k\|^2
 \geq(\lambda_1(2\pi k/L)-\lambda_1(0))\|\psi_k\|^2.
\]
Summation proves the lower bound in \eqref{eq:intro-mainD}. Equality is
obtained by a second eigenfunction of the zero fiber or by a first
eigenfunction of a minimizing nonzero fiber. Such a fiber exists because
$\lambda_1(2\pi k/L)\to\infty$ as $|k|\to\infty$. Thus the constant is
positive and attained.

Proposition~\ref{prop:dispersion} gives
\begin{equation}\label{eq:k1-asymptotic}
 \lambda_1(2\pi/L)-\lambda_1(0)
 =\frac{4\pi^2\kappa_{s,\tau}}{L^2}+O(L^{-4}).
\end{equation}
Differentiating the even analytic expansion gives
\[
 \lambda_1'(\xi)=2\kappa_{s,\tau}\xi+O(\xi^3).
\]
Choose $\delta>0$ so that $\lambda_1'(\xi)>0$ for $0<\xi<\delta$.
By Lemma~\ref{lem:away-gap}, there is $c_\delta>0$ with
$\lambda_1(\xi)-\lambda_1(0)\geq c_\delta$ for $|\xi|\geq\delta$.
For sufficiently large $L$,
\[
 \frac{2\pi}{L}<\delta,\qquad
 \lambda_1(2\pi/L)-\lambda_1(0)
 <\min\{c_\delta,\lambda_2(0)-\lambda_1(0)\}.
\]
If $|k|\geq2$ and $2\pi|k|/L<\delta$, strict monotonicity gives a
larger gap than at $|k|=1$. If $2\pi|k|/L\geq\delta$, the gap is at
least $c_\delta$. Thus only $k=\pm1$ minimize the nonzero modal gap. Equations \eqref{eq:intro-mainD} and
\eqref{eq:k1-asymptotic} prove \eqref{eq:intro-periodic-asymptotic}.
For the equality cases, subtract
$G_{s,\tau}(L)\|w^\perp\|^2$ from the numerator. Each modal term is
nonnegative. For large $L$, equality forces all modes except $k=\pm1$
to vanish, and each remaining coefficient belongs to the corresponding
first eigenspace. The two complex coefficients are independent; for real
functions they are complex conjugates. An arbitrary element of
$\mathcal M_L$ may be added without changing the quotient.
\end{proof}

\section{Angular Ritz approximation}\label{sec:ritz}

Dima and Ignat \cite{DimaIgnat} obtain sharp finite element error rates
for the full-space fractional Hardy constant. The compact angular problem
here has the following Ritz convergence; no mesh error rate is claimed.

\begin{proposition}[Angular Ritz approximation]\label{prop:ritz-convergence}
Let $Y_N\subset\mathscr Y_s$ be nested finite-dimensional spaces of
dimension at least two whose union is form-dense. Let $\Lambda_N$ and
$\lambda_{2,N}$ be their first two Ritz eigenvalues, and let $\phi_N$ be
a normalized first Ritz eigenvector with
$\langle\phi_N,\phi_{s,\tau}\rangle_{H_{s,\tau}}\geq0$. Then
\[
 \Lambda_N\downarrow\Lambda_{s,\tau},\qquad
 \lambda_{2,N}\downarrow\lambda_2(0),\qquad
 \phi_N\to\phi_{s,\tau}\quad\text{in }\mathscr Y_s.
\]
Consequently,
$\lambda_{2,N}-\Lambda_N\to\Gamma^{\rm ang}_{s,\tau}$.
If $\kappa_N$ is obtained from \eqref{eq:intro-kappa} by replacing
$\phi_{s,\tau}(\theta)\phi_{s,\tau}(\varphi)$ with
$\phi_N(\theta)\overline{\phi_N(\varphi)}$, then
$\kappa_N\to\kappa_{s,\tau}$.
\end{proposition}

\begin{proof}
Nesting and the min--max principle give monotone eigenvalue bounds.
Approximate the first two angular eigenfunctions in the form norm by
vectors in $Y_N$. Their Gram and form matrices tend to $I_2$ and
$\operatorname{diag}(\lambda_1(0),\lambda_2(0))$, proving the eigenvalue
limits. Compactness and simplicity identify every weak form limit of
$\phi_N$ with $\phi_{s,\tau}$. Moreover,
\[
 \mathfrak b_{s,0}[\phi_N]+1=\Lambda_N+1
 \longrightarrow\mathfrak b_{s,0}[\phi_{s,\tau}]+1.
\]
The form norm is induced by
$\mathfrak b_{s,0}(\cdot,\cdot)+\langle\cdot,\cdot\rangle_{H_{s,\tau}}$.
Weak convergence and convergence of these squared norms give
\[
 \begin{split}
 \|\phi_N-\phi_{s,\tau}\|_{\mathscr Y_s}^2
 ={}&\|\phi_N\|_{\mathscr Y_s}^2
 +\|\phi_{s,\tau}\|_{\mathscr Y_s}^2\\
 &-2\operatorname{Re}\langle\phi_N,\phi_{s,\tau}\rangle_{\mathscr Y_s}
 \longrightarrow0.
 \end{split}
\]
Every subsequence has the same limit by the phase normalization, so this
holds for the whole sequence. Subtraction of the eigenvalue limits gives
the gap limit.
By \eqref{eq:unweighted-control} and the integrable second-moment kernel
in \eqref{eq:q-moment-integral}, Young's inequality gives
\[
 |\kappa_N-\kappa_{s,\tau}|
 \leq C(\|\phi_N\|_2+\|\phi_{s,\tau}\|_2)
 \|\phi_N-\phi_{s,\tau}\|_2\longrightarrow0.
\]
\end{proof}

\section{Subcritical nonlinear coercivity}\label{sec:nonlinear}

Throughout this appendix, $p\geq2$ and all functions are real-valued. Put
\[
 \omega_{s,p}(x)=x_d^{(sp-1)/p}.
\]
For $u\in C_c^\infty(\mathbb R^d_+)$, define
\[
 \begin{split}
 \mathcal G_{d,s,p}[u]
 ={}&\int_{\mathbb R^d_+}\int_{\mathbb R^d_+}
 \left|\frac{u(x)}{\omega_{s,p}(x)}
 -\frac{u(y)}{\omega_{s,p}(y)}\right|^p 
 \frac{\omega_{s,p}(x)^{p/2}\omega_{s,p}(y)^{p/2}}
 {|x-y|^{d+sp}}\,dx\,dy.
 \end{split}
\]
The functional is $p$-homogeneous. If $\mathcal G_{d,s,p}[u]=0$,
Fubini's theorem gives $u/\omega_{s,p}=c$ almost everywhere.
Since $u$ has compact support, $c=0$. Thus its kernel on the test core
is trivial.

\begin{lemma}[Nonlinear Hardy ground-state remainder]
\label{lem:nonlinear-Hardy-ground-state}
Let $d\geq2$, $0<s<1$, and $p\geq2$. Then
\[
 \mathcal H_{d,s,p}[u]\geq c_p\mathcal G_{d,s,p}[u]
 \quad\text{for }u\in C_c^\infty(\mathbb R^d_+),
\]
where
\[
 c_p=\min_{0<t<1/2}
 \{(1-t)^p-t^p+pt^{p-1}\}>0.
\]
For $p=2$, one has $c_2=1$, and the inequality is an identity.
\end{lemma}

\begin{proof}
For $sp\neq1$, this is \cite[Theorem~1.2]{FrankSeiringerHalf},
based on the ground-state representation
\cite[Proposition~2.3 and Lemma~2.6]{FrankSeiringer}.
The half-space ground state there is $x_d^{(sp-1)/p}$, and its Hardy
constant agrees with \eqref{eq:hardy-general-p} in our
ordered-double-integral normalization. Every test-core function belongs
to the energy space of that theorem. Its remainder is finite: near the
diagonal, smoothness of $u/\omega_{s,p}$ gives an integrable power because
$p-sp>0$; away from the support, tangential integration reduces the
boundary and infinity bounds to the integrable powers
$y_d^{(sp-1)/2}$ and $y_d^{-(sp+3)/2}$, respectively.
For $sp=1$, one has $\omega_{s,p}=1$ and
$\mathcal D_{d,s,p}=0$, so
$\mathcal H_{d,s,p}[u]=\mathcal G_{d,s,p}[u]$.
Since $0<c_p\leq1$ by \cite[Lemma~2.6]{FrankSeiringer}, the asserted
inequality also follows in this case. The cited representation is an
identity for $p=2$, with $c_2=1$.
\end{proof}

\begin{corollary}[Subcritical nonlinear coercivity]
\label{cor:nonlinear-subcritical-stability}
\label{cor:nonlinear-two-remainder}
Let $d\geq2$, $0<s<1$, $p\geq2$, $\tau>0$, and
$0\leq\lambda<\Lambda_{d,s,p,\tau}$. Then
\[
 \mathcal H_{d,s,p}[u]-\lambda\mathcal R_{d,s,p,\tau}[u]
 \geq c_p\left(1-\frac{\lambda}{\Lambda_{d,s,p,\tau}}\right)
 \mathcal G_{d,s,p}[u]
\]
for $u\in C_c^\infty(\mathbb R^d_+)$. Together with the sharp
inequality, this gives
\[
 \begin{split}
 \mathcal H_{d,s,p}[u]-\lambda\mathcal R_{d,s,p,\tau}[u]\geq\frac{\Lambda_{d,s,p,\tau}-\lambda}{2}
 \left\{\mathcal R_{d,s,p,\tau}[u]
 +\frac{c_p}{\Lambda_{d,s,p,\tau}}\mathcal G_{d,s,p}[u]\right\}.
 \end{split}
\]
The first bound is equivalently
\[
 \mathcal H_{d,s,p}[u]\geq
 (1-\delta)\Lambda_{d,s,p,\tau}\mathcal R_{d,s,p,\tau}[u]
 +\delta c_p\mathcal G_{d,s,p}[u],\qquad 0<\delta\leq1,
\]
with $\lambda=(1-\delta)\Lambda_{d,s,p,\tau}$.
Taking $\lambda=\Lambda_{d,s,p,\tau}-\varepsilon$ in the
two-remainder bound gives the corresponding estimate for the regularized
sharp deficit, for every $0<\varepsilon\leq\Lambda_{d,s,p,\tau}$.
\end{corollary}

\begin{proof}
Write $\Lambda=\Lambda_{d,s,p,\tau}$ and
$\delta=1-\lambda/\Lambda\in(0,1]$. The exact decomposition is
\[
 \mathcal H_{d,s,p}[u]-\lambda\mathcal R_{d,s,p,\tau}[u]
 =(1-\delta)(\mathcal H_{d,s,p}[u]
 -\Lambda\mathcal R_{d,s,p,\tau}[u])
 +\delta\mathcal H_{d,s,p}[u].
\]
The first summand is nonnegative. Applying
Lemma~\ref{lem:nonlinear-Hardy-ground-state} to the second gives
$\delta c_p\mathcal G_{d,s,p}[u]$. Average this estimate with
$\mathcal H_{d,s,p}[u]-\lambda\mathcal R_{d,s,p,\tau}[u]
\geq(\Lambda_{d,s,p,\tau}-\lambda)\mathcal R_{d,s,p,\tau}[u]$
to obtain the two-remainder estimate.
\end{proof}

\begin{remark}[Scope of the nonlinear estimate]
For $p>2$, the coefficient
$c_p(1-\lambda/\Lambda_{d,s,p,\tau})$ vanishes at the sharp level.
The estimate therefore gives subcritical stability. Fischer's
two-sided representation \cite{Fischer} is proved for weighted graphs.
To apply a corresponding ground-state argument here, one would need a
positive ground state for the full Hardy--Maz'ya potential.
\end{remark}

\medskip

\noindent{\bf Data availability statement}

No datasets were generated or analysed during the present study.

\noindent{\bf Conflict of interest}

The authors declare that they have no conflict of interest.

\end{document}